\documentclass[reqno,12pt]{amsart}
\usepackage{a4}
\usepackage[latin1]{inputenc}
\usepackage[T1]{fontenc}
\usepackage{amssymb}
\usepackage{amsmath}
\usepackage{amsfonts}
\usepackage{amsthm}
\usepackage{latexsym}
\usepackage{amscd}
\usepackage{verbatim}
\usepackage{url}
\usepackage{stmaryrd}
\usepackage{enumerate}
\usepackage{dsfont}
\usepackage{appendix}
\usepackage{mathtools}
\usepackage[all,cmtip,2cell]{xy}
\UseAllTwocells
\usepackage{enumitem}
\usepackage{fullpage}
\usepackage{shuffle}

\usepackage{txfonts}

\usepackage{hyperref}

\usepackage{yhmath}

\usepackage[alpine]{ifsym}

\numberwithin{equation}{section}

{\bf}{\it}
\newtheorem*{thm*}{Theorem}
\newtheorem{thm}{Theorem}[section]{\bf}{\it}
\newtheorem{prop}[thm]{Proposition}
\newtheorem{lemma}[thm]{Lemma}

\newtheorem{cor}[thm]{Corollary}
\newtheorem{conj}[thm]{Conjecture}

\theoremstyle{definition}
\newtheorem{dfn}[thm]{Definition}
\newtheorem{cons}[thm]{Construction}
\theoremstyle{remark}
\newtheorem{rmk}[thm]{Remark}

\newtheorem{nota}[thm]{Notation}

\newtheorem{situ}[thm]{Situation}

\newcommand{\B}{\mathrm{B}}
\renewcommand{\AA}{\mathbb{A}}
\newcommand{\DD}{\mathbb{D}}
\newcommand{\BB}{\mathbb{B}}
\newcommand{\C}{\mathbb{C}}

\newcommand{\N}{\mathbb{N}}
\newcommand{\Q}{\mathbb{Q}}

\newcommand{\PP}{\mathbb{P}}

\newcommand{\Z}{\mathbb{Z}}
\newcommand{\CAlg}{\mathrm{CAlg}}

\newcommand{\Cat}{\mathrm{Cat}}
\newcommand{\CAT}{\mathrm{CAT}}

\newcommand{\one}{\mathbf{1}}

\newcommand{\Prl}{{\rm Pr}^{\rm L}}
\newcommand{\Prr}{{\rm Pr}^{\rm R}}

\newcommand{\pt}{\mathrm{pt}}

\newcommand{\alg}{\mathrm{alg}}
\newcommand{\an}{\mathrm{an}}
\newcommand{\An}{\mathrm{An}}

\newcommand{\colim}{\mathrm{colim}}

\newcommand{\ct}{\mathrm{ct}}

\newcommand{\dR}{\mathrm{dR}}
\newcommand{\eff}{\mathrm{eff}}
\newcommand{\et}{\acute{\rm e}{\rm t}}

\newcommand{\Et}{\acute{\rm E}{\rm t}}
\newcommand{\Ev}{\mathrm{Ev}}
\newcommand{\Sus}{\mathrm{Sus}}

\newcommand{\Hom}{\mathrm{Hom}}

\newcommand{\id}{\mathrm{id}}

\newcommand{\Map}{\mathrm{Map}}

\newcommand{\op}{\mathrm{op}}

\newcommand{\qcqs}{\ensuremath{\mathrm{qcqs}}}

\newcommand{\Sm}{\mathrm{Sm}}
\newcommand{\RigSm}{\mathrm{RigSm}}
\newcommand{\Spa}{\mathrm{Spa}}
\newcommand{\Spec}{\mathrm{Spec}}

\newcommand{\Sp}{\mathcal{S}p}

\newcommand{\M}{\mathrm{M}}
\newcommand{\Th}{\mathrm{Th}}

\newcommand{\FSH}{\mathbf{FSH}}

\newcommand{\Mod}{\mathrm{Mod}}
\newcommand{\coMod}{\mathrm{coMod}}

\newcommand{\Psh}{\mathrm{Psh}}

\newcommand{\Shv}{\mathrm{Shv}}
\newcommand{\RigSH}{\mathbf{RigSH}}
\newcommand{\SH}{\mathbf{SH}}

\newcommand{\Rder}{\mathrm{R}}
\newcommand{\Lder}{\mathrm{L}}

\newcommand{\mot}{\mathrm{mot}}

\newcommand{\ellcpl}{\ell\text{-}\mathrm{cpl}}

\newcommand{\cofiber}{\mathrm{cofib}}

\newcommand{\Sect}{\mathrm{Sect}}

\newcommand{\WCT}{\mathrm{WCT}}
\newcommand{\RigWCT}{\mathrm{RigWCT}}
\newcommand{\WSp}{\mathrm{WSp}}
\newcommand{\RigWSp}{\mathrm{RigWSp}}
\newcommand{\Real}{\mathrm{Real}}
\newcommand{\RigReal}{\mathrm{RigReal}}
\newcommand{\Hmot}{\mathcal{H}_{\mathrm{mot}}}
\newcommand{\Cech}{\check{\mathrm{C}}}

\newcommand{\new}{\mathrm{new}}

\newcommand{\Rig}{\mathrm{Rig}}
\newcommand{\rig}{\mathrm{rig}}

\newcommand{\Betti}{\mathrm{B}}

\newcommand{\Lmot}{\mathrm{L}_{\mathrm{mot}}}

\newcommand{\Deltasimp}{\mathbf{\Delta}}

\newcommand{\Hopf}{\mathrm{Hopf}}

\newcommand{\Gm}{\mathbb{G}_{\mathrm{m}}{}}

\renewcommand{\SS}{\mathbb{S}}

\newcommand{\CpSm}{\mathrm{CpSm}}

\newcommand{\cl}{\mathrm{cl}}
\newcommand{\AnSH}{\mathbf{AnSH}}

\newcommand{\Gmot}{\mathcal{G}_{\mathrm{mot}}}

\newcommand{\Fun}{\mathrm{Fun}}

\newcommand{\tor}{\mathrm{tor}}

\begin{document}
\title{Weil cohomology theories and their motivic Hopf algebroids}

\dedicatory{\hfill In memory of Jacob Murre}

\author[Ayoub]{Joseph Ayoub}
\address{University of Zurich / 
LAGA - Universit\'e Sorbonne Paris Nord}
\email{joseph.ayoub@math.uzh.ch}
\urladdr{user.math.uzh.ch/ayoub/}

%\thanks{The author is partially supported by the 
%{\it Swiss National Science Foundation} (SNF), 
%project 200020\_178729.}


\begin{abstract}
In this paper we discuss a general notion of Weil cohomology theories,
both in algebraic geometry and in rigid analytic geometry. We allow our Weil cohomology theories to have coefficients in arbitrary commutative ring spectra. Using the theory of motives, we give three equivalent 
viewpoints on Weil cohomology theories: as a cohomology theory on smooth varieties, as a motivic spectrum and as a realization functor. 
We also associate to every Weil cohomology theory a motivic Hopf algebroid
generalizing the construction we gave in 
\cite{gal-mot-1} for the Betti cohomology. Exploiting results and 
constructions from \cite{Nouvelle-Weil}, we are able to prove
that the motivic Hopf algebroids of all the 
classical Weil cohomology theories
are connective. In particular, they give rise to motivic Galois groupoids
that are spectral affine groupoid schemes.
\end{abstract}

\maketitle

\tableofcontents

\section*{Introduction}

Let $k$ be a ground field. Given a complex 
embedding $\sigma:k\hookrightarrow \C$, we introduced in 
\cite{gal-mot-1,gal-mot-2}
a motivic Hopf algebra $\Hmot(k,\sigma)$ coacting on the 
Betti realization of motives over $k$ in a universal way. 
The motivic Hopf algebra 
$\Hmot(k,\sigma)$, which is derived by construction,
is known to be connective by 
\cite[Corollaire 2.105]{gal-mot-1} and hence defines a spectral affine
group scheme $\Gmot(k,\sigma)$ called the motivic Galois group.
(In fact, it is conjectured that $\Hmot(k,\sigma)$
is classical, i.e., concentrated in degree zero, but we 
will not discuss this conjecture in this paper.)

A motivic Hopf algebra can be associated to any Weil cohomology theory. 
This follows from considerations in 
\cite[\S3]{from-mot-to-comod}, but will revisit the construction in 
Section 
\ref{sect:-motivic-Hopf-algebra-weil-coh} taking advantage of the modern 
language of higher category theory and higher algebra.
Using the usual comparison isomorphisms relating Betti cohomology with 
$\ell$-adic and de Rham cohomologies, it is easy to deduce from
\cite[Corollaire 2.105]{gal-mot-1} that 
the motivic Hopf algebras associated to the classical Weil cohomology
theories are connective when $k$ has characteristic zero.

Our initial motivation for writing this paper was the 
desire to extend the connectivity of the motivic Hopf algebras
to the case where $k$ has positive characteristic. 
Our proof of \cite[Corollaire 2.105]{gal-mot-1} 
relies on the explicit model of
in \cite[Corollaire 2.63]{gal-mot-1}
which is very specific to the Betti realization.
When $k$ has positive characteristic, 
the lack a Betti realization for motives over $k$
suggests, at first sight, that a new approach is necessary.
This turned out not to be the case:
we explain in this paper how to prove the connectivity of the 
motivic Hopf algebras in positive 
characteristic by reducing to the zero characteristic case!
Our strategy is to use the new Weil cohomology theories
introduced in \cite{Nouvelle-Weil}.
More precisely, let $K$ be a valued field of height $1$, of unequal characteristic $(0,p)$ and with residue field $k$.  
Fix a complex embedding $K\hookrightarrow \C$.
Then, there is a Weil cohomology theory $\Gamma_{\new,\,\Betti}$ 
for smooth $k$-varieties which is constructed from
the Betti realization for motives over $K$ and the motivic 
rigid analytification functor associated to the valuation of $K$.
The Weil cohomology theory $\Gamma_{\new,\,\Betti}$ compares to 
all the classical Weil cohomology theories for $k$-varieties: 
if $A$ is the coefficient ring of $\Gamma_{\new,\,\Betti}$,
there is a morphism $A \to \Q_{\ell}$, for every prime $\ell\neq p$, 
such that $\Gamma_{\new,\,\Betti}\otimes_A\Q_{\ell}$ 
is canonically identified with the $\ell$-adic cohomology theory.
A similar identification exists also for Berthelot's rigid cohomology
\cite{Rig-Coh-Bert}. 
Therefore, it is enough to show that the motivic Hopf algebra 
$\Hmot(\Gamma_{\new,\,\Betti})$ associated to $\Gamma_{\new,\,\Betti}$ 
is connective. Using that $\Gamma_{\new,\,\Betti}$ is 
constructed from the Betti realization and the rigid analytification 
functor, it is possible to express the underlying
algebra of $\Hmot(\Gamma_{\new,\,\Betti})$
very explicitly so that connectivity can be seen to hold directly.
(See Section \ref{sect:proof-of-main-thm} 
for more details).

To implement the strategy described above, it is natural 
to adopt a generalized notion of a Weil cohomology theory. 
In particular, we allow our Weil cohomology theories to
have coefficients in an arbitrary commutative ring spectrum
(whereas in \cite{Nouvelle-Weil} we insisted on having 
ordinary rings). A large portion of the paper is devoted
to recasting the theory of Weil cohomology theories in 
its natural generality, taking advantage of the modern techniques
of higher category theory. 
In particular, we establish equivalences of $\infty$-categories
that enable us to move freely between the different 
incarnations of Weil cohomology theories, namely: 
as a cohomology theory on smooth varieties,
as a motivic spectrum, and as a realization functor.
See Theorems 
\ref{thm:equivalence-between-weil-cohomol-spectra} and 
\ref{thm:equiv-between-WSp-and-Plain-real}.
We also clarify in Section
\ref{sect:from-alg-geo-to-rig-analytic}
the various relations between Weil cohomology theories
in algebraic geometry and in rigid analytic geometry. 
A notable result is Theorem 
\ref{thm:rig-upper-star-weil-coh-theory}
which is an improved version of the key result used in
\cite{Nouvelle-Weil} for constructing 
the new Weil cohomology theories.
Another notable result is Theorem 
\ref{thm:essential-surjectivity-of-psi-lower-star}
which, roughly speaking, asserts that all Weil cohomology theories
in rigid analytic geometry comes from Weil cohomology theories
in algebraic geometry.
In Section \ref{sect:-motivic-Hopf-algebra-weil-coh},
we recall the definition of the motivic Hopf algebra associated
to a Weil cohomology theory. In Section 
\ref{sect:example-of-weil-coh-theo},
we gather many examples of Weil cohomology theories.
Finally, we prove our connectivity theorem for the motivic Hopf algebras
in Section \ref{sect:proof-of-main-thm}.

\subsection*{Notations and conventions}$\,$

\subsubsection*{Higher categories}
We use the language of higher category theory following Lurie's books
\cite{Lurie}, \cite{Lurie-HA} and 
\cite{Lurie-SAG}, and we assume that the reader is 
familiar with this language. Our notations pertaining to 
higher category theory are very close to that of 
loc.~cit. Nevertheless, we list below some of the notations
that we use frequently. 

Fixing Grothendieck universes, we denote by
$\Cat_{\infty}$ the $\infty$-category of small $\infty$-categories
and $\CAT_{\infty}$ the $\infty$-category of locally small but 
possibly large $\infty$-categories. 
We denote by $\Prl$ (resp. $\Prr$) the $\infty$-category of 
presentable $\infty$-categories and left (resp. right) adjoint functors.
$1$-Categories are typically referred to as just `categories' and 
viewed as $\infty$-categories via the nerve construction.
We denote by $\mathcal{S}$ the $\infty$-category of spaces of small
spectra, by $\Sp$ the $\infty$-category of small spectra
and by $\Sp_{\geq 0}\subset \Sp$ its full subcategory of connective 
spectra.

If $\mathcal{C}^{\otimes}$ is a symmetric monoidal $\infty$-category,
we denote by $\CAlg(\mathcal{C})$ the $\infty$-category of commutative 
algebras in $\mathcal{C}^{\otimes}$. (In particular, 
$\CAlg(\Prl)$ is the $\infty$-category of presentable symmetric monoidal
$\infty$-categories and left adjoint symmetric monoidal functors.)
If $A$ is a commutative algebra in $\mathcal{C}$, we denote by 
$\Mod_A(\mathcal{C})$ the $\infty$-category of $A$-modules. 
When $\mathcal{C}=\Sp$, we often write $\CAlg$ instead of
$\CAlg(\Sp)$ and $\Mod_A$ instead of $\Mod_A(\Sp)$.

Given an $\infty$-category $\mathcal{C}$, we denote by $\Map_{\mathcal{C}}(x,y)$ the mapping space between two objects $x$ and $y$ in $\mathcal{C}$.
Given another $\infty$-category $\mathcal{D}$, we denote by $\Fun(\mathcal{C},\mathcal{D})$ the $\infty$-category of functors from 
$\mathcal{C}$ to $\mathcal{D}$. If $\mathcal{C}$ is small, we denote by 
$\mathcal{P}(\mathcal{C})=\Fun(\mathcal{C}^{\op},\mathcal{S})$
the $\infty$-category of presheaves on $\mathcal{C}$.
If $A$ is a commutative ring spectrum, we denote by 
$\Psh(\mathcal{C};A)=\Fun(\mathcal{C}^{\op},\Mod_A)$
the $\infty$-category of presheaves of $A$-modules
on $\mathcal{C}$. Given a topology $\tau$ on $\mathcal{C}$, 
we denote by $\Shv_{\tau}(\mathcal{C})\subset \mathcal{P}(\mathcal{C})$
and $\Shv_{\tau}(\mathcal{C};A)\subset \Psh(\mathcal{C};A)$
the full sub-$\infty$-categories of $\tau$-hypersheaves.
For an object $X\in \mathcal{C}$, we denote by $\Lambda_{\tau}(X)
\in \Shv_{\tau}(\mathcal{C};\Lambda)$ 
the $\tau$-hypersheaf associated to the presheaf of 
$\Lambda$-modules freely generated on $X$.

\subsubsection*{Varieties (algebraic, rigid analytic)}
We always denote by $k$ the ground field for 
algebraic varieties and by 
$K$ the ground field for rigid analytic varieties. 
By `algebraic $k$-variety' we mean a finite type $k$-scheme 
and by `rigid analytic $K$-variety' we mean
an adic space over $K$, in the sense of Huber, 
which is locally of finite type.
(We don't assume that rigid analytic $K$-varieties are quasi-compact
since the analytification of an algebraic $K$-variety is rarely 
quasi-compact.)
We denote by $\Sm_k$ the category of smooth $k$-varieties
and by $\RigSm_K$ the category of smooth rigid analytic $K$-varieties.
We also denote by $\RigSm^{\qcqs}_K\subset \RigSm_K$ the 
full subcategory of quasi-compact and quasi-separated smooth rigid 
analytic $K$-varieties.
These three categories will be endowed with the \'etale topology
which we abbreviate by `$\et$'. 
%Occasionally, we also consider 
%the category $\Sm_S$ (resp. $\RigSm_S$)
%of smooth $S$-schemes (resp. smooth rigid analytic $S$-spaces).

It will be convenient for us to allow the ground valued field
$K$ to be non necessary complete, and we denote by $\widehat{K}$
its completion. In particular, 
the expression `rigid analytic $K$-variety'
really means `rigid analytic $\widehat{K}$-variety'. 
We usually assume that the residue field of $K$ is the ground field $k$ 
for algebraic varieties, although we will also consider algebraic varieties
over $K$. In general, we write `$\pt$' for $\Spec(k)$ or 
$\Spa(\widehat{K})$. As usual, we write $\AA^1$ for the 
affine line, $\PP^1$ for the projective line and 
$\BB^1$ for the Tate ball.

\subsubsection*{Motives (algebraic, rigid analytic)}
We fix a connective commutative ring spectrum $\Lambda$,
and we always work $\Lambda$-linearly. (In particular, the 
tensor product of $\Lambda$-modules will be denoted 
by $-\otimes-$ instead of $-\otimes_{\Lambda}-$.) 
We will always assume that 
the exponent characteristic of $k$ in invertible in $\Lambda$.
This applies to the ground field for our algebraic varieties 
and to the residue field of the ground valued field for 
our rigid analytic varieties.

Given a scheme $S$, we denote by $\SH_{\et}(S;\Lambda)$ 
the Morel--Voevodsky $\infty$-category of \'etale 
motives over $S$ with coefficients in $\Lambda$.
Similarly, given a rigid analytic space $S$, we 
denote by $\RigSH_{\et}(S;\Lambda)$ 
the $\infty$-category of \'etale rigid analytic motives over $S$
with coefficients in $\Lambda$. We are mainly interested in the case
where $S$ is the spectrum (resp. adic spectrum) of the ground field,
and in this case we write $\SH_{\et}(k;\Lambda)$ 
(resp. $\RigSH_{\et}(K;\Lambda)$). Given a smooth algebraic $k$-variety
(resp. smooth rigid analytic $K$-variety) $X$, we denote by 
$\M(X)=\Sigma_T^{\infty}\Lmot(\Lambda_{\et}(X))$ its associated motive.

\section{Weil cohomology theories in algebraic geometry}

\label{sect:generalized-weil-coh-theories-alg-geo}

Let $k$ be a ground field and let $\Lambda\in\CAlg$ 
be a connective commutative ring spectrum.
We will always assume that the exponent characteristic 
of $k$ is invertible in $\pi_0(\Lambda)$. 
We denote by $\Sm_k$ the category of smooth algebraic $k$-varieties. 
In this article, we adopt the following general notion 
of a Weil cohomology theory. (Compare with 
\cite[Definition 2.1.4]{Weil-Coh} and 
\cite[D\'efinition 1.1]{Nouvelle-Weil}.)

\begin{dfn}
\label{dfn:generalised-weil-coh-theory}
A Weil cohomology theory $\Gamma_W$ for algebraic 
$k$-varieties is a presheaf of commutative $\Lambda$-algebras on 
$\Sm_k$ satisfying the following properties.
\begin{enumerate}

\item ($\AA^1$-invariance)
The obvious morphism $\Gamma_W(\pt)\to \Gamma_W(\AA^1)$
is an equivalence.

\item The $\Gamma_W(\pt)$-module $\Gamma_W(\PP^1,\infty)=
\cofiber\{\Gamma_W(\pt) \to \Gamma_W(\PP^1)\}$ is invertible.

\item (K\"unneth formula) For every $X,Y\in\Sm_k$, the obvious morphism
$$\Gamma_W(X)\otimes_{\Gamma_W(\pt)}\Gamma_W(Y)\to 
\Gamma_W(X\times Y)$$
is an equivalence.

\item The presheaf $\Gamma_W$ admits \'etale hyperdescent.

\end{enumerate}
The commutative $\Lambda$-algebra $\Gamma_W(\pt)$ is called the 
coefficient ring of $\Gamma_W$.
\end{dfn}

\begin{dfn}
\label{dfn:infty-category-of-WCT}
We denote by $\WCT(k;\Lambda)$ the $\infty$-category of 
Weil cohomology theories for algebraic $k$-varieties. 
This is the nonfull sub-$\infty$-category of 
$\Fun((\Sm_k)^{\op},\CAlg_{\Lambda\backslash})$
spanned by morphisms between Weil cohomology theories 
$\Gamma_W \to \Gamma_{W'}$ such that the induced morphism
$$\Gamma_W(\PP^1,\infty)
\otimes_{\Gamma_W(\pt)}\Gamma_{W'}(\pt) \to \Gamma_{W'}(\PP^1,\infty)$$
is an equivalence.
\end{dfn}

\begin{rmk}
\label{rmk:on-tate-twists-of-weil-coh-theories}
Let $\Gamma_W \in \WCT(k;\Lambda)$ 
be a Weil cohomology theory.
For $n\in \Z$, we set: 
$$\Gamma_W(n)=\Gamma_W\otimes_{\Gamma_W(\pt)}
\left(\Gamma_W(\PP^1,\infty)[-2]\right)^{\otimes -n}.$$
This is an invertible $\Gamma_W$-module. 
It follows from the K\"unneth formula that 
\begin{equation}
\label{eq-rmk:on-tate-twists-of-weil-coh-theories-3}
\Gamma_W(-\times (\PP^1,\infty))(n)[2n] \simeq 
\Gamma_W(-)(n-1)[2n-2].
\end{equation}
Observe that, given a morphism 
of Weil cohomology theories $\Gamma_W \to \Gamma_{W'}$, 
we have natural morphisms of $\Gamma_W$-modules
$\Gamma_W(n) \to \Gamma_{W'}(n)$ for all $n\in \Z$.
\end{rmk}

Weil cohomology theories are representable in the Morel--Voevodsky
stable homotopy category, and even in its \'etale localization.
To explain this, we start by recalling a few basic definitions.

\begin{dfn}
\label{dfn:stable-homotopy-category-Lambda-etale}
We denote by $\SH_{\et}^{\eff}(k;\Lambda) \subset 
\Shv_{\et}(\Sm_k;\Lambda)$ the full sub-$\infty$-category of
$\AA^1$-local \'etale hypersheaves of $\Lambda$-modules on $\Sm_k$.
We denote by 
$$\Lmot:\Shv_{\et}(\Sm_k;\Lambda) \to 
\SH_{\et}^{\eff}(k;\Lambda)$$
the motivic localisation functor.
Note that $\SH_{\et}^{\eff}(k;\Lambda)$ underlies a symmetric monoidal 
structure. We denote by $\SH_{\et}(k;\Lambda)^{\otimes}$ 
the symmetric monoidal $\infty$-category obtained from 
$\SH_{\et}^{\eff}(k;\Lambda)^{\otimes}$ by 
inverting the object $T=\Lmot(\Lambda_{\et}(\PP^1,\infty))$
for the tensor product. Given $X\in \Sm_k$, we denote by 
$\M^{\eff}(X)$ and $\M(X)$ the objects $\Lmot(\Lambda_{\et}(X))$
and $\Sigma^{\infty}_T\Lmot(\Lambda_{\et}(X))$ in 
$\SH_{\et}^{\eff}(k;\Lambda)$ and $\SH_{\et}(k;\Lambda)$.
\end{dfn}

\begin{rmk}
\label{rmk:motivic-spectra-object-SH}
By \cite[Corollary 2.22]{robalo:k-theory-bridge}, the $\infty$-category 
$\SH_{\et}(k;\Lambda)$ is the limit of the tower
$$\cdots \xrightarrow{\Omega^1_T} \SH^{\eff}_{\et}(k;\Lambda)
\xrightarrow{\Omega^1_T} \SH^{\eff}_{\et}(k;\Lambda).$$
Thus, an object of $\SH_{\et}(k;\Lambda)$ is a $T$-spectrum, i.e., 
a sequence $E=(E_n)_{n\in \N}$ of $\AA^1$-local \'etale hypersheaves 
of $\Lambda$-modules together with equivalences
$E_n\simeq \Omega^1_T(E_{n+1})=\underline{\Hom}(T,E_{n+1})$.
The functor $E\mapsto E_n$ will be denoted by $\Ev_T^n$
and its left adjoint will be denoted by $\Sus^n_T$.
For $n=0$, these functors are more commonly denoted by 
$\Omega^{\infty}_T$ and $\Sigma_T^{\infty}$ respectively.
\end{rmk}

\begin{rmk}
\label{rmk:six-functor-for-SH-}
The $\infty$-category $\SH_{\et}(k;\Lambda)$
is part of a six-functor formalism: for every scheme $S$, we
have an $\infty$-category of motivic sheaves 
$\SH_{\et}(S;\Lambda)$
underlying a closed symmetric monoidal structure, and for 
every finite type 
morphism of schemes $f$ we have functors $f^*$, $f_*$, 
$f_!$ and $f^!$. 
We will occasionally 
make use of this formalism, for example in Lemma 
\ref{lemma:generation-by-strongly-dualizable-objects}
below.
\end{rmk}

We will need the following simple 
fact concerning the notion of 
idempotent algebras in the sense of 
\cite[Definition 2.6.0.1]{Lurie-SAG}.

\begin{lemma}
\label{lemma:idempotent-algebra-}
Let $\mathcal{C}^{\otimes}$ be a symmetric monoidal $\infty$-category,
and denote by $\one$ its unit object. Let 
$A$ be an $E_0$-algebra in $\mathcal{C}$, i.e., an object 
of $\mathcal{C}$ endowed with a morphism $u:\one \to A$. 
Then the following conditions are equivalent.
\begin{enumerate}

\item  The morphism $u\otimes \id_A: A \to A\otimes A$
is an equivalence. (We express this by saying that 
the $E_0$-algebra $A$ is idempotent.)

\item The endofunctor 
$\rho:\mathcal{C} \to \mathcal{C}$, $M\mapsto M\otimes A$, 
together with the natural transformation $\id \to \rho$ induced by $u$,
defines a localisation functor. 

\end{enumerate}
Moreover, if these conditions are satisfied, there is a unique 
commutative algebra structure on $A$ extending the given 
$E_0$-algebra structure, and the commutative algebra $A$ is idempotent
in the sense of \cite[Definition 2.6.0.1]{Lurie-SAG}.
\end{lemma}

\begin{proof}
The implication (1) $\Rightarrow$ (2) follows from 
\cite[Proposition 5.2.7.4]{Lurie} and the converse is obvious. 
To prove the second assertion, we note that the localisation 
functor $\rho$ is compatible with the 
symmetric monoidal structure on $\mathcal{C}$
in the sense of \cite[Definition 2.2.1.6 \& 
Example 2.2.1.7]{Lurie-HA}. It follows from 
\cite[Proposition 2.2.1.9]{Lurie-HA} that  
$\rho$ is right-lax monoidal. In particular, $A$ is 
naturally a commutative algebra. For the unicity, we note that a 
commutative algebra structure on $A$ extending the $E_0$-algebra structure
determines (and is determined by) a symmetric monoidal functor 
$\widetilde{\rho}:\mathcal{C}^{\otimes} \to 
\Mod_A(\mathcal{C})^{\otimes}$ 
whose underlying functor is $\rho:\mathcal{C} \to \rho(\mathcal{C})$.
Thus, we may invoke again
\cite[Proposition 2.2.1.9]{Lurie-HA} to conclude.
\end{proof}

\begin{nota}
\label{nota:SH-modules-over-Gamma-}
If $R$ be a commutative algebra in $\SH^{\eff}_{\et}(k;\Lambda)$,
we write $\SH^{\eff}_{\et}(k;R)$ and $\SH_{\et}(k;R)$ for the 
$\infty$-categories $\Mod_R(\SH^{\eff}_{\et}(k;\Lambda))$
and $\Mod_R(\SH_{\et}(k;\Lambda))$.
Similarly, if $R'$ is a commutative algebra in $\SH_{\et}(k;\Lambda)$,
we write $\SH_{\et}(k;R')$ for the 
$\infty$-category $\Mod_{R'}(\SH_{\et}(k;\Lambda))$.
\end{nota}

\begin{prop}
\label{prop:key-for-constructing-bf-Gamma-W}
Let $\Gamma_W\in \WCT(k;\Lambda)$ be a Weil cohomology theory. 
Then $\Gamma_W$ is a commutative algebra in 
$\SH^{\eff}_{\et}(k;\Lambda)$ and 
there is a unique commutative algebra $\mathbf{\Gamma}_W$
in $\SH_{\et}(k;\Gamma_W)$ satisfying the following conditions:
\begin{enumerate}

\item the morphism
$\Gamma_W \to \Omega^{\infty}_T(\mathbf{\Gamma}_W)$ 
is an equivalence;

\item the underlying spectrum of $\mathbf{\Gamma}_W$ is given by 
$\Gamma_W(n)[2n]$ in level $n$ and has assembly maps induced 
from the equivalence in 
\eqref{eq-rmk:on-tate-twists-of-weil-coh-theories-3}.

\end{enumerate}
Moreover, $\mathbf{\Gamma}_W$ is an idempotent algebra in
$\SH^{\eff}_{\et}(k;\Gamma_W)$.
\end{prop}

\begin{proof}
The first statement is clear. By Remark
\ref{rmk:motivic-spectra-object-SH}, 
there is a $T$-spectrum 
$\mathbf{\Gamma}_W \in \SH_{\et}(k;\Gamma_W)$ 
satisfying the decription in (2).
We have a morphism $u:\Sigma^{\infty}_T\Gamma_W
\to \mathbf{\Gamma}_W$ from the unit object of 
$\SH_{\et}(k;\Gamma_W)$
and, by Lemma 
\ref{lemma:idempotent-algebra-},
it is enough to show that $\mathbf{\Gamma}_W$
is idempotent as an $E_0$-algebra. 
To do so, we first reduce to the case where
$k$ has finite virtual $\Lambda$-cohomological dimension in the sense of
\cite[Definition 2.4.8]{AGAV}.
Consider the family $(k_{\alpha})_{\alpha}$ 
of subfields of $k$ that are finitely generated over their prime field.
Precomposing $\Gamma_W$ with the base change functor 
$\Sm_{k_{\alpha}}\to \Sm_k$ yields a Weil cohomology theory 
$\Gamma_{W,\,\alpha}\in \WCT(k_{\alpha};\Lambda)$. 
The associated motivic spectrum $\mathbf{\Gamma}_{W,\,\alpha}$ 
is nothing but the image of $\mathbf{\Gamma}_W$ by the functor 
$(k/k_{\alpha})_*:\SH_{\et}(k;\Lambda) \to \SH_{\et}(k_{\alpha};\Lambda)$.
In fact, the family $(\mathbf{\Gamma}_{W,\,\alpha})_{\alpha}$ 
defines an object of the $\infty$-category 
$$\underset{\alpha}{\colim}\,\SH_{\et}(k_{\alpha};\Lambda),$$
provided the colimit is computed in $\Prl$
(see \cite[Corollary 5.5.3.4 \& Theorem 5.5.3.18]{Lurie}). 
Also, the obvious symmetric monoidal functor 
$$\underset{\alpha}{\colim}\,\SH_{\et}(k_{\alpha};\Lambda)
\to \SH_{\et}(k;\Lambda),$$
which is a localisation by \cite[Proposition 2.5.11]{AGAV}, takes 
$(\mathbf{\Gamma}_{W,\,\alpha})_{\alpha}$
to $\mathbf{\Gamma}_W$. Thus, to prove that the $E_0$-algebra
$\mathbf{\Gamma}_W$ is idempotent, 
it is enough to do so for the $E_0$-algebras
$\mathbf{\Gamma}_{W,\,\alpha}$. Said differently, we may assume that
$k$ is finitely generated over its prime field, and hence 
of finite virtual $\Lambda$-cohomological dimension in the sense of
\cite[Definition 2.4.8]{AGAV}. Now, the result follows
from Lemma \ref{lem:on-tensor-product-of-spectra} 
below since
$\Gamma_W(r)\otimes_{\Gamma_W}\Gamma_W(s)
=\Gamma_W(r+s)$, for all $r,s\in \Z$.
\end{proof}

The following lemma is needed for the proof of Proposition
\ref{prop:key-for-constructing-bf-Gamma-W}.

\begin{lemma}
\label{lem:on-tensor-product-of-spectra}
Assume that $k$ has finite virtual 
$\Lambda$-cohomological dimension in the sense of 
\cite[Definition 2.4.8]{AGAV}.
Let $R$ be a commutative algebra in $\SH^{\eff}_{\et}(k;\Lambda)$,
and let $E=(E_n)_{n\in \N}$ and $F=(F_n)_{n\in \N}$ 
be two $T$-spectra in $\SH_{\et}(k;R)$.
Then the $T$-spectrum $E\otimes_R F$ is given in level $n$ by
\begin{equation}
\label{eq-lem:on-tensor-product-of-spectra-1}
\Ev^n_T(E\otimes_RF) =
\underset{r+s\geq n}{\colim}\, \underline{\Hom}(T^{\otimes r+s-n},E_r\otimes_RF_s).
\end{equation}
\end{lemma}

\begin{proof}
We have
$E\simeq \colim_r\,\Sus^r_T(E_r)$ and similarly for $F$. It follows that
\begin{equation}
\label{eq-lem:on-tensor-product-of-spectra-3}
E\otimes_R F=\underset{r,s\geq 0}{\colim}\,
\Sus^r_T(E_r)\otimes_R \Sus^s_T(F_s)
\simeq 
\underset{r,s\geq 0}{\colim}\,
\Sus^{r+s}_T(E_r\otimes_R F_s).
\end{equation}
(See \cite[Corollaire 4.3.72]{these-doctorat-II}.)
Using \cite[Th\'eor\`eme 4.3.61]{these-doctorat-II}, we deduce that
\begin{equation}
\label{eq-lem:on-tensor-product-of-spectra-5}
\Ev^n_T(E\otimes_R F)=
\underset{e+r+s\geq n}{\colim}\,
\underline{\Hom}(T^{\otimes e+r+s-n},
T^{\otimes e}\otimes E_r\otimes_R F_s).
\end{equation}
(We note that \cite[Th\'eor\`eme 4.3.61]{these-doctorat-II}
is applicable since 
\cite[Hypoth\`ese 4.3.56]{these-doctorat-II}
is satisfied under our assumption on $k$
by \cite[Proposition 3.2.3]{AGAV}.)
We now notice that, in the colimit in 
\eqref{eq-lem:on-tensor-product-of-spectra-5}, 
the map from the $(e,r,s)$-th term to 
the $(e,r+e,s)$-th term factors through the 
$(0,r+e,s)$-th term using the assembly morphism 
$T^{\otimes e} \otimes E_r \to E_{r+e}$.
The same is true with $s$ in place of $r$. 
This implies that the colimit in
\eqref{eq-lem:on-tensor-product-of-spectra-5}
is equivalent to the one in the statement.
\end{proof}

The commutative algebra $\mathbf{\Gamma}_W$ has the 
following remarkable property that was uncovered in 
\cite[Theorem 2.6.2]{Weil-Coh} in the case where the coefficient ring
of $\Gamma_W$ is a field.

\begin{prop}
\label{prop:every-module-is-extended-constant}
Let $\Gamma_W\in \WCT(k;\Lambda)$ be a Weil cohomology theory, 
and let $\mathbf{\Gamma}_W$ be the associated motivic commutative ring spectrum. Then, the obvious functor 
\begin{equation}
\label{eq-prop:every-module-is-extended-constant-1}
\Mod_{\Gamma_W(\pt)} \to \SH_{\et}(k;\mathbf{\Gamma}_W), \quad 
M\mapsto \mathbf{\Gamma}_W \otimes_{\Gamma_W(\pt)}M
\end{equation}
is an equivalence.
\end{prop}

\begin{proof}
The functor 
\eqref{eq-prop:every-module-is-extended-constant-1}
is fully faithful.
Indeed, let $P$ and $Q$ be two $\Gamma_W(\pt)$-modules. 
The motivic spectrum $\mathbf{\Gamma}_W\otimes_{\Gamma_W(\pt)}P$
is given in level $r$ by 
$\Gamma_W\otimes_{\Gamma_W(\pt)}P(r)[2r]$, and similarly for $Q$ 
in place of $P$.
It follows that we have a chain of equivalences: 
$$\begin{array}{cl}
&\Map_{\SH_{\et}(k;\mathbf{\Gamma}_W)}(\mathbf{\Gamma}_W\otimes_{\Gamma_W(\pt)}P,\mathbf{\Gamma}_W\otimes_{\Gamma_W(\pt)}Q)\\
& \vspace{-.3cm}\\
\overset{(1)}{\simeq} & \Map_{\SH_{\et}(k;\Gamma_W)}(\mathbf{\Gamma}_W\otimes_{\Gamma_W(\pt)}P,\mathbf{\Gamma}_W\otimes_{\Gamma_W(\pt)}Q)\\
\vspace{-.3cm} &\\
\overset{(2)}{\simeq} & \lim_r
\Map_{\SH^{\eff}_{\et}(k;\Gamma_W)}(\Gamma_W\otimes_{\Gamma_W(\pt)}P(r)[2r],
\Gamma_W\otimes_{\Gamma_W(\pt)}Q(r)[2r])\\
\vspace{-.3cm} &\\
\overset{(3)}{\simeq} & 
\Map_{\SH^{\eff}_{\et}(k;\Gamma_W)}(\Gamma_W\otimes_{\Gamma_W(\pt)}P,
\Gamma_W\otimes_{\Gamma_W(\pt)}Q)\\
\vspace{-.3cm} &\\
\overset{(4)}{\simeq} & 
\Map_{\Mod_{\Gamma_W(\pt)}}(P,\Gamma(\pt;
\Gamma_W\otimes_{\Gamma_W(\pt)}Q))\\
\overset{(5)}{\simeq} & 
\Map_{\Mod_{\Gamma_W(\pt)}}(P,Q)
\end{array}$$
where:
\begin{enumerate}

\item follows from the fact that $\mathbf{\Gamma}_W$
is an idempotent commutative algebra in $\SH_{\et}(k;\Gamma_W)$;

\item follows from Remark 
\ref{rmk:motivic-spectra-object-SH};

\item follows by noticing that the pro-system on the 
previous line is constant;

\item is by adjunction;

\item is obvious.

\end{enumerate}
Since the functor 
\eqref{eq-prop:every-module-is-extended-constant-1}
is colimit-preserving, it remains to see that its essential
image contains a set of objects generating 
$\SH_{\et}(k;\mathbf{\Gamma}_W)$
under colimits. By Lemma 
\ref{lemma:generation-by-strongly-dualizable-objects}
below, we are reduced to showing that 
$(\pi_{X,\,*}\Lambda)\otimes \mathbf{\Gamma}_W$
belongs to the image of functor 
\eqref{eq-prop:every-module-is-extended-constant-1}.
But, since $\pi_{X,\,*}\Lambda$ is dualizable with dual 
$\pi_{X,\,\sharp}\Lambda$, we have equivalences
$$(\pi_{X,\,*}\Lambda)\otimes \mathbf{\Gamma}_W\simeq 
\underline{\Hom}(\pi_{X,\,\sharp}\Lambda,\mathbf{\Gamma}_W)\simeq 
\pi_{X,\,*}\pi_X^*\mathbf{\Gamma}_W.$$
The motivic spectrum 
$\pi_{X,\,*}\pi_X^*\mathbf{\Gamma}_W$
is given in level $r$ by 
$\Gamma_W(-\times X)(r)[2r]$
which, by the K\"unneth formula, can be identified with 
$\Gamma_W(-)\otimes_{\Gamma_W(\pt)}\Gamma_W(X)(r)[2r]$.
This shows that $\pi_{X,\,*}\pi_X^*\mathbf{\Gamma}_W$
is equivalent to 
$\mathbf{\Gamma}_W\otimes_{\Gamma_W(\pt)}\Gamma_W(X)$
as needed.
\end{proof}

Given an algebraic $k$-variety $X$, we denote by 
$\pi_X:X \to \pt$ its structural morphism. 
The following lemma was used in the proof of 
Proposition
\ref{prop:every-module-is-extended-constant}.

\begin{lemma}
\label{lemma:generation-by-strongly-dualizable-objects}
For every $X\in \Sm_k$, the motive 
$\pi_{X,\,\sharp}\Lambda\in \SH_{\et}(k;\Lambda)$ is dualizable 
with dual $\pi_{X,\,*}\Lambda$. Moreover, 
the following holds.
\begin{enumerate}

\item The $\infty$-category $\SH_{\et}(k;\Lambda)$ is generated 
under colimits, desuspension and Tate twists 
by the objects $\pi_{X,\,\sharp}\Lambda$
for $X\in \Sm_k$.

\item The $\infty$-category $\SH_{\et}(k;\Lambda)$ is generated 
under colimits, desuspension and Tate twists 
by the objects $\pi_{X,\,*}\Lambda$
for $X\in \Sm_k$.

\end{enumerate}
\end{lemma}

\begin{proof}
It suffices to prove this for $\SH(k;\Lambda)$,
the Nisnevich local version of 
$\SH_{\et}(k;\Lambda)$.
The lemma then follows from
\cite[Proposition 3.1.3]{perfection-SH}
which is based on \cite{Bondarko-Deglise}.
Indeed, by loc.~cit., the sub-$\infty$-category 
of dualizable objets coincides with the stable 
thick sub-$\infty$-category generated by $\pi_{X,\,\sharp}\Lambda(m)[n]$,
for $X\in \Sm_k$ and $m,n\in \Z$.
But since the duality functor $M\mapsto M^{\vee}$ 
is an anti-equivalence of the sub-$\infty$-category of dualizable objects, 
we deduce that the latter is also generated by 
$\pi_{X,\,*}\Lambda(m)[n]$,
for $X\in \Sm_k$ and $m,n\in \Z$.
\end{proof}

\begin{prop}
\label{prop:equivalent-condition-for-a-motivic-ring-spectrum}
Let $R=(R_n)_{n\in \N}$ be a commutative algebra in 
$\SH_{\et}(k;\Lambda)$. Then the following conditions
are equivalent.
\begin{enumerate}

\item There is a Weil cohomology theory $\Gamma_W\in \WCT(k;\Lambda)$
such that $R$ is equivalent to $\mathbf{\Gamma}_W$.

\item The presheaf of commutative algebras $R_0=\Omega^{\infty}_T(R)$
is a Weil cohomology theory and the $R_0$-modules 
$R_n$ are invertible. 

\item For every $R$-module $M=(M_n)_{n\in \N}$, 
the obvious morphism $R\otimes_{R_0}M_0 \to M$
is an equivalence.
 
\end{enumerate}
\end{prop}

\begin{proof}
The implication (1) $\Rightarrow$ (2) is clear. 
To prove the converse, we note that the condition in (2)
implies that $R$, viewed as an $E_0$-algebra in 
$\SH_{\et}(k;R_0)$, is equivalent to the 
$E_0$-algebra associated to the Weil cohomology theory $R_0$
as in the proof of Proposition
\ref{prop:key-for-constructing-bf-Gamma-W}.
The result then follows from Lemma
\ref{lemma:idempotent-algebra-}. The implication (1) $\Rightarrow$ (3)
follows from Proposition
\ref{prop:every-module-is-extended-constant}.
To finish the proof, it remains to see that (3) implies (2).
Assume that $R=(R_n)_{n\in \N}$ satisfies (3).
Since the $T$-spectrum $(R_{n+m})_{n\in \N}$ is naturally an 
$R$-module, we get that $R_m\simeq R_0\otimes_{R_0(\pt)}R_m(\pt)$.
This gives the equivalences
$$R_0(\pt)\simeq R_m((\PP^1,\infty)^{\wedge m})
\simeq R_0((\PP^1,\infty)^{\wedge m})\otimes_{R_0(\pt)}R_m(\pt)$$
showing that $R_m(\pt)$ is invertible. To conclude, it 
remains to see that $R_0$ is a Weil cohomology theory. 
Clearly, the presheaf $R_0$ is $\AA^1$-invariant and 
admits \'etale hyperdescent. Also, we have just proven that 
$R_0(\PP^1,\infty)$ is invertible with inverse $R_1(\pt)$. 
We now check the K\"unneth formula for two smooth 
algebraic $k$-varieties $X$ and $Y$. In $\SH_{\et}(k;\Lambda)$,
we have an equivalence
$\pi_{X,\,*}\Lambda \otimes \pi_{Y,\,*}\Lambda
\simeq \pi_{X\times Y,\,*}\Lambda$
where $\pi_X$, $\pi_Y$ and $\pi_{X\times Y}$ are the structural morphisms
of $X$, $Y$ and $X\times Y$.
This follows for instance from the fact that 
$\pi_{X,\,*}\Lambda$ is dualizable with dual $\pi_{X,\,\sharp}\Lambda$
(see Lemma 
\ref{lemma:generation-by-strongly-dualizable-objects}).
Tensoring with $R$, we obtain an equivalence
\begin{equation}
\label{eq-prop:equivalent-condition-for-a-motivic-ring-spectrum-13}
(R\otimes \pi_{X,\,*}\Lambda)\otimes_R (R\otimes 
\pi_{Y,\,*}\Lambda)
\simeq R\otimes \pi_{X\times Y,\,*}\Lambda.
\end{equation}
Using again that $\pi_{X,\,*}\Lambda$ is dualizable 
with dual $\pi_{X,\,\sharp}\Lambda$, we have 
$R\otimes \pi_{X,\,*}\Lambda\simeq \pi_{X,\,*}R$
and similarly for $Y$ and $X\times Y$ in place of $X$.
Thus, we can rewrite the equivalence
\eqref{eq-prop:equivalent-condition-for-a-motivic-ring-spectrum-13}
as 
\begin{equation}
\label{eq-prop:equivalent-condition-for-a-motivic-ring-spectrum-17}
\pi_{X,\,*}R\otimes_R \pi_{Y,\,*}R \simeq \pi_{X\times Y,\,*}R.
\end{equation}
Since $\pi_{X,\,*}R$ is an $R$-module, it is equivalent to 
$R\otimes_{R_0(\pt)}R_0(X)$, and similarly for $Y$ and $X\times Y$ 
in place of $X$.
Thus, passing to the $0$-level and taking global sections 
in the equivalence 
\eqref{eq-prop:equivalent-condition-for-a-motivic-ring-spectrum-17}
yield the equivalence 
$R_0(X)\otimes_{R_0(\pt)}R_0(Y)\simeq R_0(X\times Y)$.
This finishes the proof. 
\end{proof}

\begin{dfn}
\label{dfn:weil-spectrum-in-SH}
A Weil spectrum for algebraic $k$-varieties is a commutative algebra
in the $\infty$-category $\SH_{\et}(k;\Lambda)$ satisfying the equivalent conditions in Proposition 
\ref{prop:equivalent-condition-for-a-motivic-ring-spectrum}.
We denote by $\WSp(k;\Lambda)$ the full sub-$\infty$-category of 
$\CAlg(\SH_{\et}(k;\Lambda))$ spanned by Weil spectra.
\end{dfn}

\begin{cor}
\label{cor:algebra-over-weil-spectrum-weil-spectrum}
Let $\Gamma_W\in \WCT(k;\Lambda)$ be a Weil cohomology theory,
and let $R$ be a commutative $\mathbf{\Gamma}_W$-algebra
in $\SH_{\et}(k;\Lambda)$. Then $R$ is a Weil spectrum and we have
an equivalence
\begin{equation}
\label{eq-cor:algebra-over-weil-spectrum-weil-spectrum-1}
R\simeq \mathbf{\Gamma}_W\otimes_{\Gamma_W(\pt)}\Rder\Gamma(\pt;\Omega^{\infty}_T(R)).
\end{equation}
\end{cor}

\begin{proof}
The equivalence 
\eqref{eq-cor:algebra-over-weil-spectrum-weil-spectrum-1}
follows from Proposition
\ref{prop:every-module-is-extended-constant}. 
Letting $A=\Rder\Gamma(\pt;\Omega^{\infty}_T(R))$, 
we deduce that $R$ is given by $\Gamma_W(n)[2n]\otimes_{\Gamma_W(\pt)}A$
in level $n$. In particular, $R_0$ is a Weil cohomology theory
and the $R_0$-module $R_n$ is invertible for every $n\geq 0$. 
We conclude using Proposition
\ref{prop:equivalent-condition-for-a-motivic-ring-spectrum}.
\end{proof}

\begin{thm}
\label{thm:equivalence-between-weil-cohomol-spectra}
The functor $\Omega^{\infty}_T:\WSp(k;\Lambda) \to \WCT(k;\Lambda)$ is an 
equivalence of $\infty$-categories.
\end{thm}

\begin{proof}
We first construct a functor 
$\beta:\WCT(k;\Lambda)\to \WSp(k;\Lambda)$ sending a Weil cohomology theory 
$\Gamma_W$ to its Weil spectrum $\mathbf{\Gamma}_W$.
Consider the cocartesian fibrations
$$p^{(\eff)}:\int_{\Gamma_W\in \WCT(k;\Lambda)}
\SH^{(\eff)}_{\et}(k;\Gamma_W)
\to \WCT(k;\Lambda).$$
By Remark
\ref{rmk:motivic-spectra-object-SH},
the domain of $p$ is the limit of a tower having
the domain of $p^{\eff}$ on every stage 
and where the functors between successive stages are
given by $\Omega^1_T$ fiberwise. In particular, we have 
$$\Sect(p)=\lim\left(\cdots \xrightarrow{\Omega^1_T} \Sect(p^{\eff})
\xrightarrow{\Omega^1_T} \Sect(p^{\eff})\right).$$
For $r\in \N$, we have a section
$\sigma_r\in \Sect(p)$
given by $\Gamma_W \mapsto \Sus^r_T(\Gamma_W(r)[2r])$.
Noting that the section of $p^{\eff}$ sending 
$\Gamma_W$ to the presheaf 
$\Gamma_W((\PP^1,\infty)\times -)(r)[2r]\simeq \Gamma_W(r-1)[2r-2]$
maps naturally to the section
$\Ev_T^{r-1}(\sigma_r)$,
we obtain natural morphisms $\sigma_{r-1} \to \sigma_r$ for $r\geq 1$. 
We let $\sigma=\colim_r \,\sigma_r$.
By Proposition 
\ref{prop:key-for-constructing-bf-Gamma-W},
the $E_0$-algebra $\sigma$ in $\Sect(p)$
(with unit $\sigma_0 \to \sigma$) is idempotent. 
Thus, $\sigma$ is naturally a commutative algebra in 
$\Sect(p)$. Said differently, we can view $\sigma$ as a 
section of the cocartesian fibration 
$$\widetilde{p}:\int_{\Gamma_W\in \WCT(k;\Lambda)}
\CAlg(\SH_{\et}(k;\Gamma_W))
\to \WCT(k;\Lambda).$$
Composing $\sigma$ with the obvious projection 
to $\CAlg(\SH_{\et}(k;\Lambda))$, we obtain a functor 
$$\beta:\WCT(k;\Lambda) \to \CAlg(\SH_{\et}(k;\Lambda))$$
whose image lies in the sub-$\infty$-category $\WSp(k;\Lambda)$.
By construction, the composition of 
$$\WCT(k;\Lambda) \xrightarrow{\beta} 
\WSp(k;\Lambda) \xrightarrow{\Omega^{\infty}_T}
\WCT(k;\Lambda)$$
is the identity functor. To see that the composition 
$\beta\circ \Omega^{\infty}_T$ is also the identity functor,
we are reduced to showing that the section 
$\sigma'=\sigma\circ \Omega^{\infty}_T$ of the cocartesian fibration 
$$\widetilde{p}{}':\int_{\mathbf{\Gamma}_W\in \WSp(k;\Lambda)}
\CAlg(\SH_{\et}(k;\Omega^{\infty}_T\mathbf{\Gamma}_W))
\to \WSp(k;\Lambda)$$
is equivalent to the obvious diagonal section $\delta$.
By Lemma \ref{lemma:idempotent-algebra-},
it is enough to do so for the underlying $E_0$-algebras,
which is clear.
\end{proof}

\begin{dfn}
\label{dfn:homological-realization-weil-coh}
Let $\Gamma_W\in \WCT(k;\Lambda)$ be a Weil cohomology theory.
The realization functor associated to $\Gamma_W$
is the functor 
$\Rder_W^*:\SH_{\et}(k;\Lambda)\to 
\Mod_{\Gamma_W(\pt)}$  
given by the composition of
$$\SH_{\et}(k;\Lambda) \xrightarrow{\mathbf{\Gamma}_W\otimes -}
\SH_{\et}(k;\mathbf{\Gamma}_W)\simeq \Mod_{\Gamma_W(\pt)},$$
where the equivalence is provided by Proposition 
\ref{prop:every-module-is-extended-constant}. The functor 
$\Rder_W^*$ underlies a symmetric monoidal functor 
and it admits a right adjoint $\Rder_{W,\,*}$ sending 
a $\Gamma_W(\pt)$-module $M$ to 
$\mathbf{\Gamma}_W\otimes_{\Gamma_W(\pt)}M$.
\end{dfn}

\begin{rmk}
\label{rmk:weil-cohomology-recovered-realization}
A Weil cohomology theory $\Gamma_W$ can be recovered from the 
associated realization functor $\Rder^*_W$ since 
$\mathbf{\Gamma}_W\simeq \Rder_{W,\,*}\Rder^*_W\Lambda$.
\end{rmk}

\begin{dfn}
\label{dfn:plain-realization-functor-k-varieties}
A plain realization functor for algebraic $k$-varieties is a morphism 
$${\rm R}^*:\SH_{\et}(k;\Lambda)^{\otimes}
\to \Mod^{\otimes}_A$$
in $\CAlg(\Prl)$,
where $A\in \CAlg$ is a commutative ring spectrum. 
The $\infty$-category $\Real(k;\Lambda)$
of plain realization functors is the full sub-$\infty$-category of 
$\CAlg(\Prl)_{\SH_{\et}(k;\Lambda)^{\otimes}\backslash}$
spanned by functors with codomain of the form 
$\Mod^{\otimes}_A$ for some $A\in \CAlg$.
\end{dfn}

\begin{prop}
\label{prop:weil-spectrum-from-realization-functor}
For $\Rder^*\in \Real(k;\Lambda)$, the commutative algebra
$\Rder_*\Rder^*\Lambda$ is 
a Weil spectrum.
\end{prop}

\begin{proof}
Let $\Mod_A^{\otimes}$ be the codomain of $\Rder^*$.
The functor $\Rder^*$ factors through the functor 
$$\widetilde{\Rder}{}^*:\SH_{\et}(k;\Rder_*A)
\to \Mod_A, \quad M\mapsto \Rder^*(M)\otimes_{\Rder^*\Rder_*A}A.$$
By Corollary
\ref{prop:equivalent-condition-for-a-motivic-ring-spectrum},
it is enough to show that $\widetilde{\Rder}{}^*$ is an equivalence.
To do so, we first reduce to the case where
$k$ has finite virtual $\Lambda$-cohomological dimension in the sense of
\cite[Definition 2.4.8]{AGAV}.
Consider the family $(k_{\alpha})_{\alpha}$ 
of subfields of $k$ that are finitely generated over their prime field.
Precomposing with the functors $(k/k_{\alpha})^*$ yields 
plain realizations functors 
$\Rder_{\alpha}^*\in\Real(k_{\alpha};\Lambda)$ 
and we have the induced functor 
$$\widetilde{\Rder}{}_{\alpha}^*:\SH_{\et}(k_{\alpha};\Rder_{\alpha,\,*}A)
\to \Mod_A.$$
It is easy to see that we have a commutative triangle
$$\xymatrix{\underset{\alpha}{\colim}\, 
\SH_{\et}(k_{\alpha};\Rder_{\alpha,\,*}A) \ar[r]^-{\phi} \ar[dr]_-{\widetilde{\Rder}{}^*_{\infty}} & 
\SH_{\et}(k;\Rder_*A) \ar[d]^-{\widetilde{\Rder}{}^*}\\
& \Mod_A}$$
where the colimit is taken in $\Prl$.
Moreover, it follows from \cite[Proposition 2.5.11]{AGAV}
that the functor $\phi$ is a localization.
If we knew that the $\Rder_{\alpha}^*$'s were equivalences, 
we would deduce that the localisation functor $\phi$ admits 
a retraction, and hence must be an equivalence, and we would be done.
Thus, it suffices to treat the case where $k$ is finitely
generated over its prime field, and hence 
of finite virtual $\Lambda$-cohomological dimension in the sense of
\cite[Definition 2.4.8]{AGAV}.

We are now in the favourable situation were the $\infty$-category
$\SH_{\et}(k;\Lambda)$ is compactly generated by its dualizable 
objects. (Use \cite[Proposition 3.2.3]{AGAV}
and \cite[Proposition 3.1.3]{perfection-SH}.)
The same is true for $\SH_{\et}(k;\Rder_*A)$.
Since dualizable objects in $\Mod_A$ are compact
(by \cite[Proposition 7.2.4.4]{Lurie-HA}), 
the functor $\widetilde{\Rder}{}^*$ preserves compact objects
and its right adjoint
$\widetilde{\Rder}_*$ is colimit-preserving.
Since $\widetilde{\Rder}{}^*$ is clearly essentially surjective, 
we only need to show that the unit morphism 
$\id \to \widetilde{\Rder}_*\widetilde{\Rder}{}^*$
is an equivalence. Since the domain and codomain of this morphism
are colimit-preserving, it is enough to check this after 
evaluation on compact generators, and hence on objects of the form
$M\otimes \Rder_*A$ for $M\in \SH_{\et}(k;\Lambda)$ dualizable. 
In this case, the morphism we are considering can be written 
as $M\otimes \Rder_*A \to \Rder_*\Rder^*M$. 
We now conclude using \cite[Lemme 2.8]{gal-mot-1}.
\end{proof}

\begin{thm}
\label{thm:equiv-between-WSp-and-Plain-real}
The functor $\Real(k;\Lambda) \to 
\WSp(k;\Lambda)$, given by 
$\Rder^*\mapsto \Rder_*\Rder^*\Lambda$,
is an equivalence of $\infty$-categories.
\end{thm}

\begin{proof}
Let us call $\alpha$ the functor of the statement. 
As in Definition 
\ref{dfn:homological-realization-weil-coh},
we also have a functor 
$\beta:\WSp(k;\Lambda) \to \Real(k;\Lambda)$ 
sending a Weil spectrum $\mathbf{\Gamma}_W$ to the associated 
realization functor $\Rder_W^*$. By Remark 
\ref{rmk:weil-cohomology-recovered-realization},
we have $\alpha\circ \beta=\id$. To prove that $\beta\circ \alpha=\id$, 
we observe that there is a commutative diagram
$$\xymatrix{\SH_{\et}(k;\Lambda) \ar[dr]^-{\Rder^*} \ar[d]_-{\Rder_*(A)\otimes -} & \\
\SH_{\et}(k;\Rder_*A) \ar[r]_-{\widetilde{\Rder}{}^*} & \Mod_A,}$$
which is natural in $\Rder^*\in \Real(k;\Lambda)$.
But we have $\beta\circ \alpha(\Rder^*)=
\widetilde{\Rder}{}^*(\Rder_*(A)\otimes -)$
since $\widetilde{\Rder}{}^*$ is a quasi-inverse to the obvious
functor $\Mod_A\to \SH_{\et}(k;\Rder_*A)$. This finishes the proof.
\end{proof}

We end this section with the following result.

\begin{prop}
\label{prop:the-right-fibration-WCT-to-CAlg}
The functor $\WCT(k;\Lambda) 
\to \CAlg_{\Lambda\backslash}$, given by 
$\Gamma_W \mapsto \Gamma_W(\pt)$,
is a left fibration. The analogous statement is also true for 
$\WSp(k;\Lambda)$ and $\Real(k;\Lambda)$.
\end{prop}

\begin{proof}
Since the $\infty$-categories $\WCT(k;\Lambda)$, 
$\WSp(k;\Lambda)$ and $\Real(k;\Lambda)$ are all equivalent,
it is enough to prove the first statement. 
Recall that a left fibration is a cocartesian fibration 
whose fibers are groupoids (see \cite[Proposition 2.4.2.4]{Lurie}).
Clearly, $\WCT(k;\Lambda) 
\to \CAlg_{\Lambda\backslash}$ is a cocartesian fibration 
classified by the functor sending a commutative 
$\Lambda$-algebra $A$ to the $\infty$-category $\WCT(k)_A$ 
of Weil cohomology theories $\Gamma_W$ such that 
$\Gamma_W(\pt)=A$. Thus, to conclude, it is enough to show
that a morphism of Weil cohomology theories $\Gamma_W \to \Gamma_{W'}$
is an equivalence provided that the induced morphism 
$\Gamma_W(\pt)\to \Gamma_{W'}(\pt)$
is an equivalence. To prove this, we note that since
$\mathbf{\Gamma}_{W'}$ is a $\mathbf{\Gamma}_W$-module, 
we have $\mathbf{\Gamma}_{W'}=\mathbf{\Gamma}_W\otimes_{\Gamma_W(\pt)}\Gamma_{W'}(\pt)$ by Proposition
\ref{prop:every-module-is-extended-constant}.
\end{proof}

\section{Weil cohomology theories in rigid analytic geometry}

\label{sect:generalized-weil-coh-theories-rigid-analytic}

Let $K$ be a field endowed with a nontrivial valuation 
of height $1$, and let $k$ be the residue field of $K$. 
We denote by $\widehat{K}$ the completion of $K$, and 
we denote by $K^{\circ}$ and $\widehat{K}{}^{\circ}$ the 
rings of integers of $K$ and $\widehat{K}$.
By `rigid analytic $K$-variety' we mean 
a locally finite type adic $\widehat{K}$-space in the sense of Huber.
We denote by $\RigSm_K$ the category of smooth rigid analytic 
$K$-varieties, and by $\RigSm^{\qcqs}_K$ its full subcategory 
spanned by the quasi-compact and quasi-separated ones.
We fix a connective 
commutative ring spectrum $\Lambda\in \CAlg$
and assume that the exponent characteristic of 
$k$ is invertible in $\pi_0(\Lambda)$.
In this section, we extend the theory developed in 
Section \ref{sect:generalized-weil-coh-theories-alg-geo}
to the setting of rigid analytic $K$-varieties. This turns out to be 
straightforward most of the time.

\begin{dfn}
\label{dfn:generalised-weil-coh-theory-rigid-anal}
A Weil cohomology theory $\Gamma_W$
for rigid analytic $K$-varieties is a presheaf of commutative 
$\Lambda$-algebras on 
$\RigSm_K$ satisfying the following properties.
\begin{enumerate}

\item ($\BB^1$-invariance)
The obvious morphism $\Gamma_W(\pt)\to \Gamma_W(\BB^1)$
is an equivalence.

\item The $\Gamma_W(\pt)$-module $\Gamma_W(\PP^1,\infty)=
\cofiber\{\Gamma_W(\pt) \to \Gamma_W(\PP^1)\}$ is invertible.

\item (K\"unneth formula) For every $X,Y\in\RigSm^{\qcqs}_K$,
the obvious morphism
$$\Gamma_W(X)\otimes_{\Gamma_W(\pt)}\Gamma_W(Y)\to 
\Gamma_W(X\times Y)$$
is an equivalence.

\item The presheaf $\Gamma_W$ admits \'etale hyperdescent.

\end{enumerate}
The commutative $\Lambda$-algebra $\Gamma_W(\pt)$ is called the 
coefficient ring of $\Gamma_W$.
\end{dfn}

\begin{rmk}
\label{rmk:definition-in-ANT-incorrect}
Definition 
\ref{dfn:generalised-weil-coh-theory-rigid-anal}
corrects \cite[D\'efinition 2.15]{Nouvelle-Weil}
where the K\"unneth formula was requested for all $X,Y\in \RigSm_K$.
This is clearly an unreasonable demand since the tensor 
product of $\Lambda$-modules 
does not commute with infinite direct products in both variables,
unless $\Lambda$ is zero. However, using Proposition
\ref{prop:every-module-is-extended-constant-rig-anal}
below, one can prove that the K\"unneth formula holds if 
only $X$ or $Y$ is assumed to be quasi-compact and quasi-separated. 
Indeed, if $X$ is quasi-compact and quasi-separated,
the $\Gamma_W(\pt)$-module $\Gamma_W(X)$ is dualizable, and hence 
the functor $\Gamma_W(X)\otimes_{\Gamma_W(\pt)}(-)$ is limit-preserving. 
Without using Proposition
\ref{prop:every-module-is-extended-constant-rig-anal},
we still can conclude that the K\"unneth formula is true when 
$X=\PP^1$ and $Y$ is general using that $\Gamma_W(\PP^1,\infty)$
is invertible.
In particular, defining the twisted presheaves
$\Gamma_W(n)$, for $n\in \Z$, as in Remark 
\ref{rmk:on-tate-twists-of-weil-coh-theories},
we still have equivalences 
\begin{equation}
\label{eq-rmk:definition-in-ANT-incorrect-1}
\Gamma_W(-\times (\PP^1,\infty))(n)[2n] \simeq 
\Gamma_W(-)(n-1)[2n-2]
\end{equation}
of presheaves on $\RigSm_K$.
\end{rmk}

\begin{dfn}
\label{dfn:infty-category-of-rigid-WCT}
We denote by $\RigWCT(K;\Lambda)$ the $\infty$-category of 
Weil cohomology theories for rigid analytic $K$-varieties. 
This is the nonfull sub-$\infty$-category of 
$\Fun((\RigSm_K)^{\op};\CAlg_{\Lambda\backslash})$
spanned by morphisms between Weil cohomology theories 
$\Gamma_W \to \Gamma_{W'}$ such that the induced morphism
$$\Gamma_W(\PP^1,\infty)
\otimes_{\Gamma_W(\pt)}\Gamma_{W'}(\pt) \to \Gamma_{W'}(\PP^1,\infty)$$
is an equivalence.
\end{dfn}

We need to recall the rigid analytic version of the Morel--Voevodsky
stable homotopy category which is the natural home for rigid analytic
Weil spectra.

\begin{dfn}
\label{dfn:stable-homotopy-category-Lambda-etale-rigid-anal}
We denote by $\RigSH_{\et}^{\eff}(K;\Lambda) \subset 
\Shv_{\et}(\RigSm_K;\Lambda)$ the full sub-$\infty$-category of
$\BB^1$-local \'etale hypersheaves of $\Lambda$-modules on 
$\RigSm_K$.
We denote by 
$$\Lmot:\Shv_{\et}(\RigSm_K;\Lambda) \to 
\RigSH_{\et}^{\eff}(K;\Lambda)$$
the motivic localisation functor.
We denote by $\RigSH_{\et}(K;\Lambda)^{\otimes}$ 
the symmetric monoidal $\infty$-category obtained from 
$\RigSH_{\et}^{\eff}(K;\Lambda)^{\otimes}$ by 
inverting the object $T=\Lmot(\Lambda_{\et}(\PP^1,\infty))$
for the tensor product. Given $X\in \RigSm_K$, we denote by 
$\M^{\eff}(X)$ and $\M(X)$ the objects $\Lmot\Lambda_{\et}(X)$
and $\Sigma^{\infty}_T\Lmot\Lambda_{\et}(X)$ in 
$\RigSH_{\et}^{\eff}(K;\Lambda)$ and $\RigSH_{\et}(K;\Lambda)$.
\end{dfn}

Remarks 
\ref{rmk:motivic-spectra-object-SH}
and 
\ref{rmk:six-functor-for-SH-}
apply in the rigid analytic setting. 
Below, we will use the obvious analogue of Notation
\ref{nota:SH-modules-over-Gamma-}.

\begin{prop}
\label{prop:key-for-constructing-bf-Gamma-W-rig-anal}
Let $\Gamma_W\in \RigWCT(K;\Lambda)$ be a Weil cohomology theory. 
Then $\Gamma_W$ is a commutative algebra in 
$\RigSH^{\eff}_{\et}(K;\Lambda)$ and 
there is a unique commutative algebra $\mathbf{\Gamma}_W$
in $\RigSH_{\et}(K;\Gamma_W)$ 
satisfying the following conditions:
\begin{enumerate}

\item the morphism
$\Gamma_W \to \Omega^{\infty}_T(\mathbf{\Gamma}_W)$ 
is an equivalence;

\item the underlying spectrum of $\mathbf{\Gamma}_W$ is given by 
$\Gamma_W(n)[2n]$ in level $n$ and has assembly maps induced 
from the equivalence in 
\eqref{eq-rmk:definition-in-ANT-incorrect-1}.

\end{enumerate}
Moreover, $\mathbf{\Gamma}_W$ is an idempotent algebra in
$\RigSH^{\eff}_{\et}(K;\Gamma_W)$.
\end{prop}

\begin{proof}
The proof of Proposition
\ref{prop:key-for-constructing-bf-Gamma-W}
extends easily to the rigid analytic setting. 
We need to reduce to the case where $\widehat{K}$ has 
finite virtual 
$\Lambda$-cohomological dimension. To do so, we write $K$ as the 
union of its subfields $K_{\alpha}$ 
which are non discrete and finitely 
generated over their prime field.
Then, by \cite[Theorem 2.5.1]{AGAV},
$$\underset{\alpha}{\colim}\,\RigSH_{\et}(K_{\alpha};\Lambda)
\to \RigSH_{\et}(K;\Lambda)$$
is a localisation functor, as needed for the reduction. 
Finally, we note that the 
rigid analytic version of Lemma
\ref{lem:on-tensor-product-of-spectra}
holds true and can be proven similarly.
\end{proof}

\begin{prop}
\label{prop:every-module-is-extended-constant-rig-anal}
Let $\Gamma_W\in \RigWCT(K;\Lambda)$ be a Weil cohomology theory, 
and let $\mathbf{\Gamma}_W$ be the associated motivic commutative ring spectrum. Then, the obvious functor 
\begin{equation}
\label{eq-prop:every-module-is-extended-constant-rig-anal-1}
\Mod_{\Gamma_W(\pt)} \to \RigSH_{\et}(K;\mathbf{\Gamma}_W), \quad 
M\mapsto \mathbf{\Gamma}_W \otimes_{\Gamma_W(\pt)}M
\end{equation}
is an equivalence.
\end{prop}

\begin{proof}
The proof of Proposition
\ref{prop:every-module-is-extended-constant}
extends literally to the rigid analytic setting 
provided that we have a rigid analytic version of Lemma 
\ref{lemma:generation-by-strongly-dualizable-objects}.
This is the subject of Lemma
\ref{lemma:generation-good-reduction-dualizability}
below.
\end{proof}

\begin{lemma}
\label{lemma:push-forward-strongly-dualizable-object}
Let $f:T \to S$ be a smooth and proper morphism of rigid analytic 
spaces. Let $M\in \RigSH_{\et}(T)$ be a dualizable object. 
Then $f_*M$ is dualizable with dual $f_{\sharp}M^{\vee}$.
\end{lemma}

\begin{proof}
For $N\in \RigSH_{\et}(S)$, we have the following chain of equivalences
$$\underline{\Hom}(f_*M,N)\overset{(1)}{\simeq} 
\underline{\Hom}(f_!M,N) \overset{(2)}{\simeq} f_*\underline{\Hom}(M,f^!N)
\overset{(3)}{\simeq} f_*\underline{\Hom}(M,f^!\Lambda \otimes f^*N)$$
$$\overset{(4)}{\simeq} f_*\left(\underline{\Hom}(M,f^!\Lambda)
\otimes f^*N\right)
\overset{(5)}{\simeq} f_*\underline{\Hom}(M,f^!\Lambda) \otimes N
\overset{(6)}{\simeq} f_!(\Th(\Omega_f)\otimes M^{\vee}) \otimes N$$
where:
\begin{enumerate}

\item[(1)] follows from the identification $f_!\simeq f_*$
(see \cite[Proposition 4.4.27]{AGAV});

\item[(2)] follows from \cite[Corollary 4.5.4]{AGAV};

\item[(3)] follows from \cite[Theorem 4.4.29]{AGAV};

\item[(4)] follows from the assumption that $M$ is dualizable;

\item[(5)] follows from \cite[Proposition 4.1.7]{AGAV};

\item[(6)] follows from \cite[Theorem 4.4.29]{AGAV}.

\end{enumerate}
This proves that $f_*M$ is dualizable with dual 
$f_!(\Th(\Omega_f)\otimes M^{\vee})=f_{\sharp}M^{\vee}$
as needed.
\end{proof}

Given a rigid analytic $K$-variety $X$, we denote by
$\pi_X:X \to \pt=\Spa(\widehat{K})$ its structural morphism. 
The following lemma is used in the proof of Proposition
\ref{prop:every-module-is-extended-constant-rig-anal}.

\begin{lemma}
\label{lemma:generation-good-reduction-dualizability}
For every $X\in \RigSm^{\qcqs}_K$, the motive 
$\pi_{X,\,\sharp}\Lambda\in \RigSH_{\et}(K;\Lambda)$ is dualizable 
with dual $\pi_{X,\,*}\Lambda$. Moreover, 
the following holds.
\begin{enumerate}

\item The $\infty$-category $\RigSH_{\et}(K;\Lambda)$ is generated 
under colimits, desuspension and Tate twsists 
by the objects $\pi_{X,\,\sharp}\Lambda$
for $X\in \RigSm^{\qcqs}_K$.

\item The $\infty$-category $\RigSH_{\et}(K;\Lambda)$ is generated 
under colimits, desuspension and Tate twsists 
by the objects $\pi_{X,\,*}\Lambda$
for $X\in \RigSm^{\qcqs}_K$.

\end{enumerate}
In fact, in (1) and (2), we may restrict to those 
$X$'s with potential good reduction.
\end{lemma}

\begin{proof}
Since $\underline{\Hom}(\pi_{X,\,\sharp}\Lambda,\Lambda)=
\pi_{X,\,*}\Lambda$, the dual of $\pi_{X,\,\sharp}\Lambda$ 
is necessarily $\pi_{X,\,*}\Lambda$ provided that $\pi_{X,\,\sharp}\Lambda$
is dualizable. Therefore, it suffices to show point (1) and that 
the thick stable sub-$\infty$-category generated by 
$\pi_{X,\,\sharp}\Lambda(m)[n]$, for $m,n\in \N$ and 
$X\in \RigSm^{\qcqs}_K$,
is closed under taking the dual. To prove this, we may 
employ the devise used in the proof of Proposition 
\ref{prop:key-for-constructing-bf-Gamma-W-rig-anal}
to reduce to the case where $\widehat{K}$ is has finite virtual 
$\Lambda$-cohomological dimension. In this case, we know
that $\RigSH_{\et}(K;\Lambda)$ is compactly generated by 
the $\pi_{X,\,\sharp}\Lambda(m)[n]$, for $m,n\in \N$ and 
$X\in \RigSm^{\qcqs}_K$, 
and it remains to see that these generators are 
dualizable. (Indeed, the dual is then necessary compact, and the 
second point to check is automatic.)

By \cite[Proposition 3.7.17]{AGAV},
the $\infty$-category 
$\RigSH_{\et}(K;\Lambda)$ is generated 
under colimits by objects of the form
$\pi_{X,\,\sharp}\Lambda(-m)[-n]$
where $m,n\in \Z$ are integers and $X$
has potentially good reduction. 
More precisely, we can assume that there are a finite \'etale extension
$L/\widehat{K}$ and a smooth formal $L^{\circ}$-scheme $\mathcal{X}$
such that $X=\mathcal{X}_{\eta}$.
By Lemma \ref{lemma:push-forward-strongly-dualizable-object}
applied to $f:\Spa(L) \to \Spa(\widehat{K})=\pt$, it is enough to show that 
$\pi_{X/L,\,\sharp}\Lambda$ is dualizable.
Said differently, we may assume that $X$ has good reduction, i.e., that 
$X=\mathcal{X}_{\eta}$ with $\mathcal{X}$ smooth over 
$\widehat{K}{}^{\circ}$. 
In this case, $\pi_{X,\,\sharp}\Lambda$ is the image of 
$\pi_{\mathcal{X}_{\sigma},\,\sharp}\Lambda$ by the 
symmetric monoidal functor 
$$\xi:\SH_{\et}(k;\Lambda) \simeq \FSH_{\et}(\widehat{K}{}^{\circ};\Lambda)
\to \RigSH_{\et}(K;\Lambda)$$
described in \cite[Notation 3.1.12]{AGAV}.
(See also Remark
\ref{rmk:two-functors-relating-motive-rigan-mot}
below.) By Lemma
\ref{lemma:generation-by-strongly-dualizable-objects}, 
$\pi_{\mathcal{X}_{\sigma},\,\sharp}\Lambda$
is dualizable, and this enables us to conclude.
\end{proof}

\begin{prop}
\label{prop:equivalent-cond-for-motivic-ring-spectr-rigan}
Let $R=(R_n)_{n\in \N}$ be a commutative algebra in 
$\RigSH_{\et}(K;\Lambda)$. Then the following conditions
are equivalent.
\begin{enumerate}

\item There is a Weil cohomology theory $\Gamma_W\in 
\RigWCT(K;\Lambda)$
such that $R$ is equivalent to $\mathbf{\Gamma}_W$.

\item The presheaf of commutative algebras $R_0=\Omega^{\infty}_T(R)$
is a Weil cohomology theory and the $R_0$-modules 
$R_n$ are invertible. 

\item For every $R$-module $M=(M_n)_{n\in \N}$, 
the obvious morphism $R\otimes_{R_0}M_0 \to M$
is an equivalence.
 
\end{enumerate}
\end{prop}

\begin{proof}
The proof of Proposition
\ref{prop:equivalent-condition-for-a-motivic-ring-spectrum}
extends literally to the rigid analytic setting.
Of course, one needs to use Lemma 
\ref{lemma:generation-good-reduction-dualizability}
in place of Lemma 
\ref{lemma:generation-by-strongly-dualizable-objects}.
\end{proof}

\begin{dfn}
\label{dfn:weil-spectrum-in-SH-rigan}
A Weil spectrum for rigid analytic $K$-varieties is a
commutative algebra in the $\infty$-category
$\RigSH_{\et}(K;\Lambda)$ 
satisfying the equivalent conditions
in Proposition 
\ref{prop:equivalent-cond-for-motivic-ring-spectr-rigan}.
We denote by $\RigWSp(K;\Lambda)$ 
the full sub-$\infty$-category of 
$\CAlg(\RigSH_{\et}(K;\Lambda))$ spanned by Weil spectra.
\end{dfn}

\begin{cor}
\label{cor:algebra-over-weil-spectrum-weil-spectrum-rigan}
Let $\Gamma_W\in \RigWCT(K;\Lambda)$ be a Weil cohomology theory,
and let $R$ be a commutative $\mathbf{\Gamma}_W$-algebra
in $\RigSH_{\et}(K;\Lambda)$. Then $R$ is a Weil spectrum and we have
an equivalence
\begin{equation}
\label{eq-cor:algebra-over-weil-spectrum-weil-spect-rigan-1}
R\simeq \mathbf{\Gamma}_W\otimes_{\Gamma_W(\pt)}\Rder\Gamma(\pt;\Omega^{\infty}_T(R)).
\end{equation}
\end{cor}

\begin{proof}
The proof of Corollary 
\ref{cor:algebra-over-weil-spectrum-weil-spectrum}
extends literally. 
\end{proof}

\begin{thm}
\label{thm:equivalence-between-weil-cohomol-spectra-rigan}
The functor $\Omega^{\infty}_T:\RigWSp(K;\Lambda) \to 
\RigWCT(K;\Lambda)$ is an 
equivalence of $\infty$-categories.
\end{thm}

\begin{proof}
The proof of Theorem 
\ref{thm:equivalence-between-weil-cohomol-spectra}
extends literally. 
\end{proof}

\begin{dfn}
\label{dfn:homological-realization-weil-coh-rigan}
Let $\Gamma_W\in \RigWCT(K;\Lambda)$ be a Weil cohomology theory.
The realization functor associated to $\Gamma_W$
is the functor 
$\Rder_W^*:\RigSH_{\et}(K;\Lambda)\to \Mod_{\Gamma_W(\pt)}$  
given by the composition of
$$\RigSH_{\et}(K;\Lambda) \xrightarrow{\mathbf{\Gamma}_W\otimes -}
\RigSH_{\et}(K;\mathbf{\Gamma}_W)\simeq \Mod_{\Gamma_W(\pt)},$$
where the equivalence is provided by Proposition 
\ref{prop:every-module-is-extended-constant-rig-anal}. The functor 
$\Rder_W^*$ underlies a symmetric monoidal functor 
and it admits a right adjoint $\Rder_{W,\,*}$ sending 
a $\Gamma_W(\pt)$-module $M$ to 
$\mathbf{\Gamma}_W\otimes_{\Gamma_W(\pt)}M$.
\end{dfn}

\begin{dfn}
\label{dfn:plain-realization-functor-k-varieties-rigan}
A plain realization functor for rigid analytic 
$K$-varieties is a morphism
$${\rm R}^*:\RigSH_{\et}(K;\Lambda)^{\otimes}
\to \Mod^{\otimes}_A$$
in $\CAlg(\Prl)$,
where $A$ is a commutative ring spectrum. 
The $\infty$-category $\RigReal(K;\Lambda)$
of plain realization functors is the full sub-$\infty$-category of 
$\CAlg(\Prl)_{\RigSH_{\et}(K;\Lambda)^{\otimes}\backslash}$
spanned by functors with codomain of the form 
$\Mod^{\otimes}_A$ for some $A\in \CAlg$.
\end{dfn}

\begin{prop}
\label{prop:weil-spectrum-from-realization-functor-rigan}
For $\Rder^*\in \RigReal(K;\Lambda)$, the commutative algebra
$\Rder_*\Rder^*\Lambda$ is a Weil spectrum.
\end{prop}

\begin{proof}
The proof of Proposition 
\ref{prop:weil-spectrum-from-realization-functor}
can be easily adapted to the rigid analytic setting.
The reduction to the case where $K$ has finite virtual 
$\Lambda$-cohomological dimension is obtained as in the proof of
Proposition
\ref{prop:key-for-constructing-bf-Gamma-W-rig-anal}. 
In this case $\RigSH_{\et}(K;\Lambda)$ is 
compactly generated by its dualizable objects 
as it follows from Lemma
\ref{lemma:generation-good-reduction-dualizability}
and \cite[Proposition 2.4.22]{AGAV}.
\end{proof}

\begin{thm}
\label{thm:equiv-between-WSp-and-Plain-real-rigan}
The functor $\RigReal(K;\Lambda) \to 
\RigWSp(K;\Lambda)$, given by 
$\Rder^*\mapsto \Rder_*\Rder^*\Lambda$,
is an equivalence of $\infty$-categories.
\end{thm}

\begin{proof}
The proof of Theorem 
\ref{thm:equiv-between-WSp-and-Plain-real}
extends literally.
\end{proof}

\begin{prop}
\label{prop:the-right-fibration-WCT-to-CAlg-rigan}
The functor $\RigWCT(K;\Lambda) 
\to \CAlg_{\Lambda\backslash}$, given by 
$\Gamma_W \mapsto \Gamma_W(\pt)$,
is a left fibration. The analogous statement is also true for 
$\RigWSp(K;\Lambda)$ and $\RigReal(K;\Lambda)$.
\end{prop}

\begin{proof}
The proof of Proposition
\ref{prop:the-right-fibration-WCT-to-CAlg}
extends literally.
\end{proof}

\section{From algebraic geometry to rigid analytic geometry and back}

\label{sect:from-alg-geo-to-rig-analytic}

We keep the running notations and assumptions from Sections
\ref{sect:generalized-weil-coh-theories-alg-geo} and
\ref{sect:generalized-weil-coh-theories-rigid-analytic}.
We will discuss the 
relations between Weil cohomology theories for algebraic varieties 
and for rigid analytic varieties. 
For this, we need to recall the two functors
relating algebraic motives and rigid analytic motives.

\begin{rmk}
\label{rmk:two-functors-relating-motive-rigan-mot}
There are two symmetric monoidal functors 
$$\Rig^*:\SH_{\et}(K;\Lambda) \to 
\RigSH_{\et}(K;\Lambda)
\qquad \text{and} \qquad
\xi^*:\SH_{\et}(k;\Lambda)
\to \RigSH_{\et}(K;\Lambda).$$
The functor $\Rig^*$ is induced by the analytification functor 
$X\mapsto X^{\an}$ sending a smooth $K$-variety $X$ to its 
rigid analytification. (The effective version of this functor 
is constructed in \cite[Proposition 1.3.6]{mot-rig} 
and its stabilisation is introduced on page 101 of \cite{mot-rig};
this functor is also recalled on page 34 of
\cite{AGAV} under the name `$\An^*$'.) 
Modulo the equivalence 
$\FSH_{\et}(\widehat{K}{}^{\circ};\Lambda) \simeq 
\SH_{\et}(k;\Lambda)$ of \cite[Corollaire 1.4.24]{mot-rig}, 
the functor $\xi^*$ is induced by the functor 
$\mathcal{X}\mapsto \mathcal{X}_{\eta}$
taking a formal $\widehat{K}{}^{\circ}$-scheme to its Raynaud generic fiber.
(This functor is discussed in \cite[\S 3.1]{AGAV}.) 
As usual, we denote by $\Rig_*$ and $\xi_*$ the right adjoints 
of $\Rig^*$ and $\xi^*$.
\end{rmk}

\begin{prop}
\label{prop:rig-lower-star-of-weil-cohomology-theory}
Let $\mathbf{\Gamma}_W\in \RigWSp(K;\Lambda)$ be a Weil spectrum
for rigid analytic $K$-varieties. Then $\Rig_*\mathbf{\Gamma}_W$ 
and $\xi_*\mathbf{\Gamma}_W$
are Weil spectra for algebraic varieties.
Moreover, if $\Rder_W^*$ is the realization functor associated 
to $\mathbf{\Gamma}_W$, then $\Rder_W^*\circ \Rig^*$
and $\Rder_W^*\circ \xi^*$ are the realization functors associated 
to $\Rig_*\mathbf{\Gamma}_W$ 
and $\xi_*\mathbf{\Gamma}_W$ respectively.
\end{prop}

\begin{proof}
It is clear that 
$\Rder_W^*\circ \Rig^*$
and $\Rder_W^*\circ \xi^*$ belong 
$\Real(K;\Lambda)$ and $\Real(k;\Lambda)$
respectively, and that the Weil spectra associated to these plain 
realization functors are  
$\Rig_*\mathbf{\Gamma}_W$ 
and $\xi_*\mathbf{\Gamma}_W$.
\end{proof}

\begin{rmk}
\label{rmk:compatibility-of-WCT-rig-star-xi}
We have a commutative diagram of $\infty$-categories
$$\xymatrix{\Real(k;\Lambda) \ar[r]^-{\sim} & \WSp(k;\Lambda)
\ar[r]^-{\Omega^{\infty}_T}_-{\sim} & \WCT(k;\Lambda)\\
\RigReal(K;\Lambda) \ar[r]^-{\sim} \ar[u]^-{-\circ \xi^*} 
\ar[d]_-{-\circ \Rig^*}
& \RigWSp(K;\Lambda)
\ar[r]^-{\Omega^{\infty}_T}_-{\sim} \ar[u]_-{\xi_*} \ar[d]^-{\Rig_*} 
& \RigWCT(K;\Lambda) \ar[u]_-{\xi_*} \ar[d]^-{\Rig_*} \\
\Real(K;\Lambda) \ar[r]^-{\sim} & \WSp(K;\Lambda)
\ar[r]^-{\Omega^{\infty}_T}_-{\sim} & \WCT(K;\Lambda)}$$
where the horizontal arrows are equivalences.
Moreover, the functor 
$$\Rig_*:\RigWCT(K;\Lambda) \to 
\WCT(K;\Lambda)$$
is just the naive one given by composing a Weil cohomology theory 
on rigid analytic $K$-varieties with the rigid analytification functor 
$X\mapsto X^{\an}$.
\end{rmk}

The following result will play an important role in the sequel. 

\begin{thm}
\label{thm:rig-upper-star-weil-coh-theory}
Let $\mathbf{\Gamma}_W\in \WSp(K;\Lambda)$ be a Weil spectrum for 
algebraic $K$-varieties. Then $\Rig^*\mathbf{\Gamma}_W$ 
is a Weil spectrum for rigid analytic $K$-varieties. 
Thus, we have a functor 
$$\Rig^*:\WSp(K;\Lambda) \to \RigWSp(K;\Lambda)$$
which is left adjoint to the functor $\Rig_*$
from Remark
\ref{rmk:compatibility-of-WCT-rig-star-xi}.
\end{thm}

To prove Theorem
\ref{thm:rig-upper-star-weil-coh-theory}, 
we need the following result. 

\begin{prop}
\label{prop:rig-upper-star-generates-under-colim}
The image of the functor 
$\Rig^*:\SH_{\et}(K;\Lambda) \to \RigSH_{\et}(K;\Lambda)$
generates the $\infty$-category $\RigSH_{\et}(K;\Lambda)$
under colimits.
\end{prop}

\begin{proof}
The case where $\Lambda$ is a commutative $\Q$-algebra
was discussed in \cite[Proposition 2.31]{Nouvelle-Weil} and it is based on
\cite[Th\'eor\`eme 2.5.35]{mot-rig}. (In loc.~cit., $\Lambda$ 
is supposed to be classical, but this is not 
a real restriction since the statement for $\Lambda=\Q$ implies 
the statement for any commutative $\Q$-algebra.) 
To prove the proposition, we reduce the general case 
to the case where $\Lambda$ is a $\Q$-algebra using rigidity. 
Given $M\in \RigSH_{\et}(K;\Lambda)$, 
we may consider the cofiber sequence
$$M \to M\otimes \Q \to M\otimes \Q/\Z.$$
Since the forgetful functor $\Mod_{\Lambda\otimes \Q}
\to \Mod_{\Lambda}$ is colimit-preserving, 
there is a commutative square
$$\xymatrix@C=3pc{\SH_{\et}(K;\Lambda\otimes \Q) \ar[r]^-{\Rig^*\otimes \,\Q} \ar[d]
& \RigSH_{\et}(K;\Lambda\otimes \Q) \ar[d] \\
\SH_{\et}(K;\Lambda) \ar[r]^-{\Rig^*} & \RigSH_{\et}(K;\Lambda).\!}$$
The object $M\otimes \Q$ belongs to the image of the right vertical 
functor. Thus, by \cite[Proposition 2.31]{Nouvelle-Weil},
$M\otimes \Q$ belongs to the localizing subcategory generated by
the image of $\Rig^*$.
To conclude, it remains to prove the same for 
$M\otimes \Q/\Z$. Using 
\cite[Theorems 2.10.3 \& 2.10.4]{AGAV}, 
we are reduced to showing that the image of the functor
$$\Shv_{\et}(\Et_K;\Lambda) \to \Shv_{\et}(\Et_{\widehat{K}};\Lambda)$$
generates the $\infty$-category $\Shv_{\et}(\Et_{\widehat{K}};\Lambda)$
under colimits. This follows from the fact that every finite separable 
extension of $\widehat{K}$ is a base change 
of a separable extension of $K$ (by Krasner's lemma).
\end{proof}

\begin{nota}
\label{nota:widetilde-Rig-upper-star-module-Rig-lower-star}
The functor $\Rig_*:\RigSH_{\et}(K;\Lambda)
\to \SH_{\et}(K;\Lambda)$ is right-lax monoidal. In particular, 
$\Rig_*\Lambda$ is a commutative algebra in 
$\SH_{\et}(K;\Lambda)$. We have an adjunction 
$$\widetilde{\Rig}{}^*:
\SH_{\et}(K;\Rig_*\Lambda) \rightleftarrows
\RigSH_{\et}(K;\Lambda):\widetilde{\Rig}_*,$$
where $\widetilde{\Rig}{}^*$ is given by the formula
$\widetilde{\Rig}{}^*(M)=\Rig^*(M)\otimes_{\Rig^*\Rig_*\Lambda}\Lambda$.
\end{nota}

\begin{thm}
\label{thm:the-main-equivalence-for-new-Weil-coh}
The functor $\widetilde{\Rig}{}^*$ 
is an equivalence of $\infty$-categories.
\end{thm}

\begin{proof}
The functor $\widetilde{\Rig}{}^*$ commutes with colimits. 
Using Proposition 
\ref{prop:rig-upper-star-generates-under-colim}, 
it remains to see that it is fully faithful, i.e., 
that the unit of the adjunction
$\id \to \widetilde{\Rig}_*\widetilde{\Rig}{}^*$
is an equivalence. By Lemma 
\ref{lemma:Rig-lower-star-commutes-with-colimits}
below, it is enough to check this on a set of objects generating 
$\SH_{\et}(K;\Rig_*\Lambda)$ under colimits. By Lemma
\ref{lemma:generation-by-strongly-dualizable-objects}
and Proposition 
\ref{prop:rig-upper-star-generates-under-colim}, 
such a set is given by objects of the form $M\otimes \Rig_*\Lambda$,
with $M\in \SH_{\et}(K;\Lambda)$ dualizable.
The unit morphism evaluated at such an object 
coincides with the obvious morphism
$M \otimes \Rig_* \Lambda \to \Rig_*\Rig^*M$
which is an equivalence by 
\cite[Lemme 2.8]{gal-mot-1}.
\end{proof}

\begin{lemma}
\label{lemma:Rig-lower-star-commutes-with-colimits}
The functor $\Rig_*:\RigSH_{\et}(K;\Lambda)
\to \SH_{\et}(K;\Lambda)$
is colimit-preserving.
\end{lemma}

\begin{proof}
Let $(M_i)_{i\in I}$ be an inductive system in 
$\RigSH_{\et}(K;\Lambda)$. We want to show that the morphism
$$\underset{i}{\colim}\,\Rig_*(M_i) \to 
\Rig_*\left(\underset{i}{\colim}\,M_i\right)$$
is an equivalence. It is enough to do so for the systems 
$(M_i\otimes \Q)_i$ and $(M_i\otimes \Q/\Z)_i$.
For the system $(M_i\otimes \Q)_i$, we use the commutative 
square
$$\xymatrix{\RigSH_{\et}(K;\Lambda\otimes \Q) \ar[r]^-{\Rig_*} 
\ar[d] & 
\SH_{\et}(K;\Lambda \otimes \Q) \ar[d]\\
\RigSH_{\et}(K;\Lambda) \ar[r]^-{\Rig_*} & \SH_{\et}(K;\Lambda)}$$
where all the functors are colimit-preserving, except possibly the 
horizontal bottom one.

For the system $(M_i\otimes \Q/\Z)_i$, we need to work a bit more.
We first prove that the functor 
\begin{equation}
\label{eq-lemma:Rig-lower-star-commutes-with-colimits-47}
\iota_*:\Shv_{\et}(\Et_{\widehat{K}};\Lambda)
\to \Shv_{\et}(\Et_K;\Lambda)
\end{equation}
is colimit-preserving.
To do so, it is enough to prove that the functor 
$$\iota_*:\Shv_{t_{\amalg}}(\Et_{\widehat{K}}) \to 
\Shv_{t_{\amalg}}(\Et_K),$$
which is obviously colimit-preserving, 
preserves the $\et$-local equivalences.
(Here, we denote by $t_{\amalg}$ the Grothendieck topology generated
by covers of the form $(U_j \to \coprod_{i\in I}U_i)_{j\in I}$.)
Given a finite separable extension $L/K$, any faithfully flat 
\'etale $\widehat{K}\otimes_KL$-algebra can be refined by an
algebra of the form $\widehat{K}\otimes_KL'$ for a finite separable
extension $L'/L$. This can be used to show that if a morphism $u:A \to B$ 
of pointed $t_{\amalg}$-sheaves of spaces on $\Et_{\widehat{K}}$ 
induces isomorphisms $\Lder_{\et}\pi_n(A) \simeq \Lder_{\et}\pi_n(B)$
for $n\in \N$, then $\iota_*(u)$ also induces isomorphisms
$\Lder_{\et}\pi_n(\iota_*(A))\simeq \Lder_{\et}\pi_n(\iota_*(B))$.

Since the functor
\eqref{eq-lemma:Rig-lower-star-commutes-with-colimits-47}
is colimit-preserving, it also preserves torsion objects. 
Consider the following square
\begin{equation}
\label{eq-lemma:Rig-lower-star-commutes-with-colimits-123}
\begin{split}
\xymatrix{\Shv_{\et}(\Et_{\widehat{K}};\Lambda)_{\tor}
\ar[r]^-{\iota_*} \ar[d] & 
\Shv_{\et}(\Et_K;\Lambda)_{\tor} \ar[d]\\
\RigSH_{\et}(K;\Lambda) \ar[r]^-{\Rig_*} & 
\SH_{\et}(K;\Lambda)}
\end{split}
\end{equation}
where the vertical arrows are colimit-preserving, fully faithful
and induce equivalences with the sub-$\infty$-categories
of torsion objects by \cite[Theorems 2.10.3 \& 2.10.4]{AGAV}.
This square is in fact commutative. Indeed, 
the vertical arrows send a torsion \'etale hypersheaf $F$
on $\Et_{\widehat{K}}$ (resp. $\Et_K$)
to the $T$-spectrum given in level $n$ by 
$F(n)[2n]$ left Kan extended to $\RigSm_K$ (resp. $\Sm_K$)
and hypersheafified.
It now clear how to conclude: 
the system 
$(M_i\otimes \Q/\Z)_i$
belongs to the essential image of the left vertical arrow
in the square 
\eqref{eq-lemma:Rig-lower-star-commutes-with-colimits-123}.
Since all the functors in this square are colimit-preserving, except possibly the horizontal bottom one, the result follows.
\end{proof}

We can now give the proof of Theorem 
\ref{thm:rig-upper-star-weil-coh-theory}.

\begin{proof}[Proof of Theorem 
\ref{thm:rig-upper-star-weil-coh-theory}]
Set $A=\Gamma(\pt;\Omega^{\infty}_T
\Rig^*\mathbf{\Gamma}_W)$.
By Proposition
\ref{prop:equivalent-cond-for-motivic-ring-spectr-rigan},
it is enough to show that the functor 
$$\Mod_A \to \RigSH_{\et}(K;\Rig^*\mathbf{\Gamma}_W)$$
is an equivalence. By Theorem 
\ref{thm:the-main-equivalence-for-new-Weil-coh},
we have an equivalence of $\infty$-categories
$$\RigSH_{\et}(K;\Rig^*\mathbf{\Gamma}_W)
\simeq \SH_{\et}(K;\mathbf{\Gamma}_W\otimes \Rig_*\Lambda).$$
Thus, setting $B=\Gamma(\pt;\Omega^{\infty}_T(
\mathbf{\Gamma}_W\otimes \Rig_*\Lambda))$,
it is enough to show that the functor 
$$\Mod_B \to \SH_{\et}(K;\mathbf{\Gamma}_W\otimes \Rig_*\Lambda)$$
is an equivalence.
By Proposition 
\ref{prop:equivalent-condition-for-a-motivic-ring-spectrum},
we need to show that 
$\mathbf{\Gamma}_W\otimes \Rig_*\Lambda$ is a Weil spectrum. 
Since $\mathbf{\Gamma}_W$ is a Weil spectrum, Proposition 
\ref{prop:every-module-is-extended-constant}
furnishes an equivalence of commutative algebras
$\mathbf{\Gamma}_W\otimes_{\Gamma_W(\pt)}B\simeq 
\mathbf{\Gamma}_W\otimes \Rig_*\Lambda$, showing that 
$\mathbf{\Gamma}_W\otimes \Rig_*\Lambda$ is indeed a Weil spectrum.
\end{proof}

The following is a corollary of the proof of Theorem 
\ref{thm:rig-upper-star-weil-coh-theory}.

\begin{cor}
\label{cor:coefficients-of-rig-upper-star-Gamma-W}
Let $\mathbf{\Gamma}_W\in \WSp(K;\Lambda)$ be a Weil spectrum for 
algebraic $K$-varieties, and let $\Rder^*_W:\SH_{\et}(K;\Lambda)
\to \Mod_{\Gamma_W(\pt)}$ be the associated realization. 
\begin{enumerate}

\item The coefficient ring of the 
Weil spectrum $\Rig^*\mathbf{\Gamma}_W$ is equivalent to 
$\Rder^*_W(\Rig_*\Lambda)$.

\item The realization associated to $\Rig^*\mathbf{\Gamma}_W$
is given by the composition of 
\begin{equation}
\label{eq-cor:coefficients-of-rig-upper-star-Gamma-W-1}
\RigSH_{\et}(K;\Lambda)
\xrightarrow{\widetilde{\Rig}_*}
\SH_{\et}(K;\Rig_*\Lambda)
\xrightarrow{\Rder^*_W}
\Mod_{\Rder^*_W(\Rig_*\Lambda)}.
\end{equation}

\end{enumerate}
\end{cor}

\begin{proof}
The coefficient ring of $\Rig^*\mathbf{\Gamma}_W$ is the commutative 
algebra $A=\Gamma(\pt;\Omega^{\infty}_T
\Rig^*\mathbf{\Gamma}_W)$ which is equivalent to 
the commutative algebra 
$B=\Gamma(\pt;\Omega^{\infty}_T(
\mathbf{\Gamma}_W\otimes \Rig_*\Lambda))$.
This proves the first claim since the realization functor $\Rder^*_W$
is given by $\Gamma(\pt;\Omega^{\infty}_T(\mathbf{\Gamma}_W\otimes -))$.
To prove the second statement, we need to show that the right 
adjoint to the functor 
\eqref{eq-cor:coefficients-of-rig-upper-star-Gamma-W-1}
sends $\Rder^*_W(\Rig_*\Lambda)$ to $\Rig^*\mathbf{\Gamma}_W$.
This is clear since the equivalence $\widetilde{\Rig}_*$
sends $\Rig^*\mathbf{\Gamma}_W$ to $\mathbf{\Gamma}_W\otimes \Rig_*\Lambda$
which is also the image of $\Rder^*_W(\Rig_*\Lambda)$
by the right adjoint functor 
$\Rder_{W,\,*}:\Mod_{\Rder_W^*(\Rig_*\Lambda)}\to 
\SH_{\et}(K;\Rig_*\Lambda)$.
\end{proof}

\begin{rmk}
\label{rmk:new-Weil-cohomology-theory-}
Corollary 
\ref{cor:coefficients-of-rig-upper-star-Gamma-W}
shows that the Weil cohomology theory represented by the 
Weil spectrum $\Rig^*\mathbf{\Gamma}_W$ is precisely the 
`new' Weil cohomology theory associated to $\Gamma_W$ 
in the sense of \cite[\S 2E]{Nouvelle-Weil}. Here, contrary
to loc.~cit., we allow Weil cohomology theories with non necessarily 
classical coefficients rings.  
\end{rmk}

In a similar vein, we have the following result.

\begin{thm}
\label{thm:xi-of-Weil-cohom-theory--K-algebraically-closed}
Assume that $K$ is algebraically closed.
If $\mathbf{\Gamma}_W\in \WSp(k;\Lambda)$ is a Weil spectrum
on algebraic $k$-varieties, then $\xi^*\mathbf{\Gamma}_W$ 
is a Weil spectrum on rigid analytic $K$-varieties
whose ring of coefficients is equivalent to
$\Rder_W^*(\xi_*\Lambda)$,
where $\Rder_W^*:\SH_{\et}(k;\Lambda) \to \Mod_{\Gamma_W(\pt)}$
is the realization functor associated to $\Gamma_W$.
Thus, we have a functor 
$$\xi^*:\WSp(k;\Lambda) \to \RigWSp(K;\Lambda)$$
which is left adjoint to the functor $\xi_*$ in 
Remark \ref{rmk:compatibility-of-WCT-rig-star-xi}.
\end{thm}

\begin{proof}
The proof of Theorem
\ref{thm:xi-of-Weil-cohom-theory--K-algebraically-closed}
is very similar to the proof of Theorem
\ref{thm:rig-upper-star-weil-coh-theory}.
Recall from \cite[Theorem 3.7.21]{AGAV} that the functor 
$$\widetilde{\xi}_*:\RigSH_{\et}(K;\Lambda)
\to \SH_{\et}(k;\xi_*\Lambda)$$
is an equivalence of $\infty$-categories. 
(In loc.~cit., the adjunction 
$(\xi^*,\xi_*)$ is denoted by $(\xi,\chi)$.)
Thus, the functor $\xi^*$ is equivalent to 
$-\otimes \xi_*\Lambda:\SH_{\et}(k;\Lambda)
\to \SH_{\et}(k;\xi_*\Lambda)$. To conclude, it remains to see that the 
functor $\Mod_{\Rder_W^*(\xi_*\Lambda)}
\to \SH_{\et}(k;\mathbf{\Gamma}_W\otimes \xi_*\Lambda)$
is an equivalence, which follows from the fact that 
$\mathbf{\Gamma}_W\otimes \xi_*\Lambda$ is a Weil spectrum on 
algebraic $k$-varieties with coefficient ring 
$\Rder_W^*(\xi_*\Lambda)$.
\end{proof}

To go further, we consider the following situation.

\begin{situ}
\label{situ:finite-rank-value-group-rho-for-psi}
Assume that $K$ is algebraically closed, and that the value
group of $K$ is finite dimensional over $\Q$.
Let $P$ be a set of primes and let
$P^{\times}$ be the submonoid of $\N^{\times}$ generated by $P$.
We assume that the characteristic of $k$ is not in 
$P$ and that $P$ contains all the primes that are
not invertible in $\pi_0(\Lambda)$. (This is possible since the exponent 
characteristic of $k$ is invertible in $\pi_0(\Lambda)$; 
note that if $\Lambda$ is
a $\Q$-algebra, we may take $P=\emptyset$ so to have $P^{\times}=\{1\}$.)
Let $R$ be a strictly henselian regular ring, 
let $a_1,\ldots, a_n\in R$ be a 
regular sequence in $R$, and let $\rho:R \to K^{\circ}$ be a local 
morphism such that $|\rho(a_1)|,\ldots, |\rho(a_n)|$ form a basis 
of the value group of $K$ over $\Q$.
(The existence of such a morphism follows from de Jong's
theorem on resolution of singularities by alterations
\cite{alteration-de-Jong}.) We consider the unique 
factorization 
$$R \to \overline{R} \xrightarrow{\overline{\rho}}
K^{\circ}$$ 
where $\overline{R}$ is the profinite $R$-algebra obtained 
by extracting all the $r$-th roots of the $a_i$'s for the 
integers $r \in P^{\times}$.
We set $S=\Spec(R)$ and, for $\emptyset\neq 
I\subset \{1,\ldots, n\}$, we set $D_I=\Spec(R/(a_s, s\in I))$.
We also write $C$ for $D_{\{1,\ldots, n\}}$, and set 
$D=\bigcup_{s=1}^n D_s$ and $U=S\smallsetminus D$.
We form the
commutative diagram with cartesian squares (up to nil-immersions)
$$\xymatrix{\Spec(K) \ar[r]^-{\widehat{j}} 
\ar[d] & \Spec(K^{\circ}) \ar[d]
& \Spec(K) \ar[l]_-{\widehat{i}} \ar[d]\\
\overline{U} \ar[r]^-{\overline{j}} 
\ar[d] & \overline{S} \ar[d] & C
\ar[l]_-{\overline{i}} \ar@{=}[d] \\
U \ar[r]^-j & S & C \ar[l]_-i}$$
where $\overline{S}=\Spec(\overline{R})$.
\end{situ}

As usual, given a symmetric monoidal $\infty$-category 
$\mathcal{C}^{\otimes}$, we denote by ${\rm S}(-):\mathcal{C}
\to \CAlg(\mathcal{C})$ the functor sending 
an object to the associated free commutative algebra.

\begin{thm}
\label{thm:the-functor-Psi-using-basis-of-value-group}
In Situation
\ref{situ:finite-rank-value-group-rho-for-psi},
we have equivalences
\begin{equation}
\label{eq-thm:the-functor-Psi-using-basis-of-value-group-1}
\xi_*\Lambda\simeq 
(\overline{i}{}^*\overline{j}_*\Lambda)|_{\Spec(k)}\simeq {\rm S}(\Lambda_{\Q}^{\oplus n}(-1)[-1])
\end{equation}
of commutative algebras in $\SH_{\et}(k;\Lambda)$.
(Here and below we write $\Lambda_{\Q}$ for $\Lambda \otimes \Q$.)
\end{thm}

\begin{proof}
We will deduce this from 
\cite[Theorem 3.8.1]{AGAV}.
Indeed, by loc. cit., we have 
$$\xi_*\Lambda=\widehat{i}{}^*\widehat{j}_*\Lambda,$$
and the task is to compute the 
commutative algebra $\widehat{i}{}^*\widehat{j}_*\Lambda$.
Using \cite{alteration-de-Jong}, 
we can write $K^{\circ}$ as a filtered colimit of 
strictly henselian local $R$-algebras $R_{\alpha}$ 
with local homomorphisms
$\rho_{\alpha}:R_{\alpha} \to  K^{\circ}$
such that $S_{\alpha}=\Spec(R_{\alpha})$
is regular and the inverse image of $D$
in $S_{\alpha}$ is a normal crossing divisor.
Thus, we can find a regular sequence 
$a_{\alpha,\,1},\ldots, a_{\alpha,\,n}$ in $R_{\alpha}$ 
such that the $|\rho_{\alpha}(a_{\alpha,\,s})|$, for $1\leq s \leq n$,
form a basis of the 
value group of $K$ over $\Q$ and the inverse image of $D$ is 
$D_{\alpha}=\bigcup_{s=1}^nD_{\alpha,\,s}$ with $D_{\alpha,\,s}=\Spec(R_{\alpha}/(a_{\alpha,\,s}))$.
Then we have relations in $R_{\alpha}$
$$a_s=u_s\cdot a_{\alpha,\,1}^{e_{s,\,1}}\cdots a_{\alpha,\,n}^{e_{s,\,n}}$$
with $u_s$ invertible in $R_{\alpha}$ and $e_{s,j}\in \N$.
In particular, letting $U_{\alpha}=S_{\alpha}\smallsetminus D_{\alpha}$
and $C_{\alpha}=D_{\alpha,\,1}\cap \cdots \cap D_{\alpha,\,n}$, 
we can form a commutative diagram with cartesian squares 
(up to nil-immersions) 
$$\xymatrix{\Spec(K) \ar[r]^-{\widehat{j}} \ar[d] 
& \Spec(K^{\circ}) \ar[d] & \Spec(k) \ar[l]_-{\widehat{i}} \ar[d] \\
\overline{U}_{\alpha} \ar[r]^-{\overline{j}_{\alpha}} \ar[d] 
& \overline{S}_{\alpha} \ar[d] & 
C_{\alpha}\ar[l]_-{\overline{i}_{\alpha}} \ar[d] \\
U_{\alpha} \ar[r]^-{j_{\alpha}} 
& S_{\alpha} & 
C_{\alpha}\ar[l]_-{i_{\alpha}}}$$
with $\overline{S}_{\alpha}=\Spec(\overline{R}_{\alpha})$
as in Situation
\ref{situ:finite-rank-value-group-rho-for-psi}.
It follows from \cite[Proposition 3.2.4]{AGAV} that
$$\widehat{i}{}^*\widehat{j}_*\Lambda
\simeq \underset{\alpha}{\colim}\;
(\overline{i}{}_{\alpha}^*\overline{j}_{\alpha,\,*}\Lambda)|_{\Spec(k)}.$$
Thus, to establish the first equivalence in
\eqref{eq-thm:the-functor-Psi-using-basis-of-value-group-1}, 
it suffices to show that the morphisms
$$(\overline{i}{}^*\overline{j}_*\Lambda)|_{C_{\alpha}} \to 
\overline{i}{}_{\alpha}^*\overline{j}_{\alpha,\,*}\Lambda$$
are equivalences. Consider the ring 
$A=\Z[\sqrt{-1},P^{-1}]$. Using purity in the form of 
\cite[Th\'eor\`eme 7.4]{real-etale} and continuity 
\cite[Proposition 3.2.4]{AGAV}, 
we are reduced to showing the analogous property for the morphism 
of $A$-algebras
$$A\left[x^{1/r}_1,\ldots, x^{1/r}_n\mid r\in P^{\times}\right]
\to A\left[u^{\pm 1/r}_1,\ldots, u^{\pm 1/r}_n,
y^{1/r}_1,\ldots, y^{1/r}_n\mid r\in P^{\times}\right],$$
sending $x^{1/r}_s$ to $u_s^{1/r}\cdot y_1^{e_{s,1}/r}\cdots 
y_n^{e_{s,\,n}/r}$,
in place of the morphism $\overline{R} \to \overline{R}_{\alpha}$.
Consider the schemes: 
\begin{itemize}

\item $T=\Spec(A[x^{\pm 1/r}_1,\ldots, x^{\pm 1/r}_n\mid r\in P^{\times}])$,

\item $B=\Spec(A[u^{\pm 1/r}_1,\ldots, u^{\pm 1/r}_n\mid r\in P^{\times}])$,

\item $T'=\Spec(A[u^{\pm 1/r}_1,\ldots, u^{\pm 1/r}_n,
y^{\pm 1/r}_1,\ldots, y^{\pm 1/r}_n\mid r\in P^{\times}])$,

\end{itemize}
and the commutative square
$$\xymatrix{T' \ar[r]^-h \ar[d]_-q & T \ar[d]^-p\\
B \ar[r] & \Spec(\Z[P^{-1}])}$$
where $h$ is induced by the morphism sending 
$x^{1/r}_s$ to $u_s^{1/r}\cdot y_1^{e_{s,1}/r}\cdots 
y_n^{e_{s,\,n}/r}$. 
By Lemma
\ref{lemma:i-upper-star-vs-p-lower-star-torus} 
below, it suffices to show that the morphism 
$$(p_*\Lambda)|_B \to q_*\Lambda$$
is an equivalence. This follows from Lemma
\ref{lemma:morphism-of-protorus-and-motives}
below by noticing that $(h,q):T'\to T\times B$ 
can be viewed as a morphism of 
pro-tori over $B$. (This requires a change of coordinates on $T'$ 
of the form
$y'_s=u_1^{a_{s,1}}\cdots u_n^{a_{s,n}}\cdot y_s$, 
with the $a_{s,j}$'s in $\Q$.)

It remains to establish the second equivalence in
\eqref{thm:the-functor-Psi-using-basis-of-value-group}.
Using purity in the form of 
\cite[Th\'eor\`eme 7.4]{real-etale}
and continuity 
\cite[Proposition 3.2.4]{AGAV},
we are again reduced to showing the analogous property 
for the $A$-algebra
$$A\left[x^{1/r}_1,\ldots, x^{1/r}_n\mid r\in P^{\times}\right].$$
The result follows from Lemma
\ref{lemma:i-upper-star-vs-p-lower-star-torus} below
and Proposition
\ref{prop:action-of-elevation-power-m-on-Gm}.
\end{proof}

\begin{lemma}
\label{lemma:i-upper-star-vs-p-lower-star-torus}
Let $S$ be a scheme, and consider the commutative diagram
of schemes
$$\xymatrix{\AA^n_S\smallsetminus H \ar[r]^-j \ar[dr]_-q 
& \AA^n_S \ar[d]^-p & S \ar@{=}[dl]
\ar[l]_i\\
& S,\! & }$$
where $H\subset \AA^n_S$ is the union of the standard hyperplanes,
$j$ is the obvious inclusion and $i$ is the zero section.
Then there is a natural equivalence
$$q_*q^* \xrightarrow{\sim} i^*j_*q^*$$
between endofunctors of $\SH_{\et}(S;\Lambda)$.
\end{lemma}

\begin{proof}
The natural morphism $q_*q^*\to i^*j_*q^*$ is obtained by applying 
$p_*$ to $j_*q^* \to i_*i^*j_*q^*$. We show that this morphism
is an equivalence by induction on $n$. When $n=1$, we need to show that
$p_*j_!q^*\simeq 0$ which follows from the fact that the morphism 
$p_*p^*\to p_*i_*i^*p^*$ is an equivalence by homotopy invariance.
The general case follows by induction using 
\cite[Th\'eor\`eme 3.3.10]{these-doctorat-II}
for the canonical specialisation system to conclude that 
$j_*q^*M$ restricted to the intersection of $n-1$ standard 
hyperplanes is the direct image along the inclusion 
$\AA^1_S\smallsetminus 0_S \to \AA_S^1$ of a motive 
pulled back from $S$. We leave the details to the reader.
\end{proof}

\begin{lemma}
\label{lemma:morphism-of-protorus-and-motives}
Consider a morphism of pro-tori
$$\xymatrix{T' \ar[rr]^e \ar[dr]_-{p'} & & T \ar[dl]^-p\\
& S & }$$
over a scheme $S$.
Let $L$ and $L'$ be the dual ind-lattices of $T$ and $T'$.
Assume that the induced morphism $e^*:L\otimes \Q \to L'\otimes \Q$
is an isomorphism and that $L'/L$ is $p$-torsion free for every prime
$p$ which is not invertible in $\pi_0(\Lambda)$.
Then the induced morphism
$p_*\Lambda \to p'_*\Lambda$ is an equivalence in 
$\SH_{\et}(S;\Lambda)$.
\end{lemma}

\begin{proof}
Using localization \cite[Corollaire 4.5.47]{these-doctorat-II} 
and \'etale descent \cite[Proposition 3.2.1]{AGAV}, 
we may replace $S$ by $S[\sqrt{-1}]$
and assume that $-1$ is a square in $\mathcal{O}(S)$.
If $m$ is invertible in $\pi_0(\Lambda)$, 
multiplication by $m$ on any pro-torus
$q:E \to S$ induces an autoequivalence of $q_*\Lambda$. 
Indeed, it suffices to prove this when $E$ is of finite type. 
Using \'etale descent \cite[Proposition 3.2.1]{AGAV}, 
we further reduce to the case $E=(\Gm_{,\,S})^{\times n}$. 
We then conclude using Proposition 
\ref{prop:action-of-elevation-power-m-on-Gm}.

Let $A\subset \Q$ be the localization of $\Z$ at all 
primes which are invertible in $\pi_0(\Lambda)$, 
and consider $A$ as an ind-lattice on $S$. Given a 
pro-torus $q:E\to S$, we set $\widetilde{E}=\underline{\Hom}(A,E)$
and form the commutative triangle
$$\xymatrix{\widetilde{E} \ar[rr] \ar[dr]_-{\widetilde{q}} & & E
\ar[dl]^-q\\
& S &}$$
defining a morphism of pro-tori over $S$.
By the previous discussion, the induced morphism
$q_*\Lambda \to \widetilde{q}_*\Lambda$ is 
an equivalence in $\SH_{\et}(S;\Lambda)$.
Going back to the statement of the lemma, we see that it suffices
to prove that the morphism 
$\widetilde{p}_*\Lambda \to \widetilde{p}{}'_*\Lambda$ 
is an equivalence. But, the morphism $\widetilde{e}:
\widetilde{T}{}'\to \widetilde{T}$ 
induces an isomorphism on the dual ind-lattices, and hence is an 
isomorphism. 
\end{proof}

\begin{rmk}
\label{rmk:dependence-of-equivalence-xi-star-Lambda}
The second equivalence in 
\eqref{eq-thm:the-functor-Psi-using-basis-of-value-group-1}
depends on the choice of a compatible sequence of 
roots of the $a_i$'s. Thus, the equivalence
$$\xi_*\Lambda\simeq {\rm S}(\Lambda_{\Q}^{\oplus n}(-1)[-1])$$
depends on the morphism $\rho:R \to K^{\circ}$, on the regular sequence
$a_1,\ldots,a_n$ and the compatible families of $r$-roots 
of the $a_i$'s for $r\in P^{\times}$.
The same type of dependency applies to the functor 
$\psi^*$ constructed below.
\end{rmk}

\begin{cons}
\label{cons:the-vanishing-cycles-functor-in-general}
In Situation 
\ref{situ:finite-rank-value-group-rho-for-psi},
we define the functor 
$$\psi^*:\RigSH_{\et}(K;\Lambda)
\to \SH_{\et}(k;\Lambda)$$
to be the composition of 
$$\RigSH_{\et}(K;\Lambda)
\xrightarrow{\widetilde{\xi}_*}
\SH_{\et}(k;\xi_*\Lambda) 
\xrightarrow{-\otimes_{\xi_*\Lambda}\Lambda}
\SH_{\et}(k;\Lambda),$$
where the second arrow is the base change functor along the morphism of
commutative algebras ${\rm S}(\Lambda_{\Q}^{\oplus n}(-1)[-1])\to \Lambda$
corresponding to the zero morphism $\Lambda_{\Q}^{\oplus n}(-1)[-1] 
\to \Lambda$ in $\SH_{\et}(k;\Lambda)$. Clearly, the functor 
$\psi^*$ underlies a symmetric monoidal functor and admits a right 
adjoint $\psi_*$.
\end{cons}

\begin{rmk}
\label{rmk:relation-with-nearby-cycles-}
The functor $\psi^*$ is a version of the motivic nearby 
functor. Compare with 
\cite[Scholie 1.3.26(2)]{mot-rig}.
\end{rmk}

\begin{lemma}
\label{lemma:psi-star-of-Weil-coho-theories}
In Situation
\ref{situ:finite-rank-value-group-rho-for-psi},
let $\mathbf{\Gamma}_W\in \WSp(k;\Lambda)$ 
be a Weil spectrum on algebraic $k$-varieties. 
Then $\psi_*\mathbf{\Gamma}_W$ is a Weil spectrum on 
rigid analytic $K$-varieties. This defines a functor 
$$\psi_*:\WSp(k;\Lambda) \to \RigWSp(K;\Lambda).$$
\end{lemma}

\begin{proof}
Indeed, we have a functor
$-\circ \psi^*:\Real(k;\Lambda) \to \RigReal(K;\Lambda)$.
\end{proof}

\begin{prop}
\label{prop:relation-psi-lower-star-xi-upper-star}
In Situation
\ref{situ:finite-rank-value-group-rho-for-psi}, 
let $\mathbf{\Gamma}_W\in \WSp(k;\Lambda)$ be a Weil spectrum 
on algebraic $k$-varieties. Then, there is an 
equivalence 
$$\psi_*\mathbf{\Gamma}_W\simeq (\xi^*\mathbf{\Gamma}_W)\otimes_{\Rder_W(\xi_*\Lambda)}\Rder^*_W(\Lambda)$$
in $\RigWSp(K;\Lambda)$.
\end{prop}

\begin{proof}
Recall that the realization functor associated
to the Weil spectrum $\xi^*\mathbf{\Gamma}_W$ is given by
the composition of 
$$\RigSH_{\et}(K;\Lambda)
\xrightarrow{\widetilde{\xi}_*}
\SH_{\et}(k;\xi_*\Lambda) \xrightarrow{\Rder_W^*}
\Mod_{\Rder_W^*(\xi_*\Lambda)}.$$
The result follows from the commutative diagram
$$\xymatrix{\RigSH_{\et}(K;\Lambda) 
\ar[r]^-{\widetilde{\xi}_*}_-{\sim} \ar[dr]_-{\psi^*} & 
\SH_{\et}(k;\xi_*\Lambda)
\ar[r]^-{\Rder_W^*} \ar[d]^-{-\otimes_{\xi_*\Lambda}\Lambda}& \Mod_{\Rder_W^*(\xi_*\Lambda)} \ar[d]^-{-\otimes_{\Rder_W(\xi_*\Lambda)}\Rder^*_W(\Lambda)}\\
& \SH_{\et}(k;\Lambda) \ar[r]^-{\Rder^*_W} 
& \Mod_{\Gamma_W(\pt)}}$$
showing that the realization functors associated to the 
Weil spectra under consideration
are naturally equivalent.
\end{proof}

\begin{thm}
\label{thm:essential-surjectivity-of-psi-lower-star}
We work in Situation
\ref{situ:finite-rank-value-group-rho-for-psi}.
Assume furthermore that $\Lambda$ is a $\Q$-algebra.
Let $\mathbf{\Gamma}_{W'}\in \RigWSp(K;\Lambda)$ 
be a Weil spectrum on rigid analytic $K$-varieties
such that $\Gamma_{W'}(\pt)$ and $\Gamma_{W'}(\pt)(1)$ 
are connective.
Then, there exist a Weil spectrum $\mathbf{\Gamma}_W\in \WSp(k;\Lambda)$
such that $\psi_*\mathbf{\Gamma}_W$ is equivalent to 
$\mathbf{\Gamma}_{W'}$. In fact, we may take $\mathbf{\Gamma}_W=\xi_*\mathbf{\Gamma}_{W'}$.
\end{thm}

\begin{proof}
Set $\mathbf{\Gamma}_{W}=\xi_*\mathbf{\Gamma}_{W'}$
and consider the counit morphism
$$\xi^*\mathbf{\Gamma}_W=\xi^*\xi_*\mathbf{\Gamma}_{W'} \to \mathbf{\Gamma}_{W'}.$$
Denote by 
$\theta:\Rder_W^*(\xi_*\Lambda)
\to \Gamma_{W'}(\pt)=\Gamma_W(\pt)$
the induced morphism on the ring of coefficients.
By Proposition
\ref{prop:the-right-fibration-WCT-to-CAlg-rigan}, 
we have an equivalence
$$\xi^*\mathbf{\Gamma}_W\otimes_{\Rder_W^*(\xi_*\Lambda),\,\theta}
\Gamma_W(\pt) \simeq \mathbf{\Gamma}_{W'}.$$
By Theorem
\ref{thm:the-functor-Psi-using-basis-of-value-group}
and because $\Lambda$ is a $\Q$-algebra, 
$\Rder^*_W(\xi_*\Lambda)\simeq {\rm S}_{\Gamma_W(\pt)}
(\Gamma_W(\pt)^{\oplus n}(-1)[-1])$ is the free commutative algebra
on the $\Gamma_W(\pt)$-module $\Gamma_W(\pt)^{\oplus n}(-1)[-1]$.
Thus, $\theta$ is uniquely determined by its restriction 
$$\theta_0:\Gamma_W(\pt)^{\oplus n}(-1)[-1] \to \Gamma_W(\pt).$$
To give a morphism $\theta_0$ is equivalent to given 
$n$ elements in $\pi_{-1}(\Gamma_W(\pt)(1))$ which is zero 
since $\Gamma_W(\pt)(1)$ is connective.
This proves that $\theta$ is homotopic to the realization 
of the obvious morphism $\xi_*\Lambda \to \Lambda$, i.e., 
the one used in Construction
\ref{cons:the-vanishing-cycles-functor-in-general}.
By Proposition
\ref{prop:relation-psi-lower-star-xi-upper-star}, 
$\xi^*\mathbf{\Gamma}_W
\otimes_{\Rder_W^*(\xi_*\Lambda),\,\theta}
\Gamma_W(\pt)$
is thus equivalent to $\psi_*\mathbf{\Gamma}_W$ as needed.
\end{proof}

\begin{rmk}
\label{rmk:relation-with-binda-gallauer-vezzani}
A statement, similar to Theorem 
\ref{thm:essential-surjectivity-of-psi-lower-star},
was obtained recently by Binda--Gallauer--Vezzani
in \cite[Corollary 4.34]{BGV}. In loc.~cit., the authors
consider realization functors valued in abstract stable symmetric monoidal 
$\infty$-categories (and not only $\infty$-categories 
of modules over commutative ring spectra) satisfying a certain compatibility
with the weight structure on rigid analytic motives. Contrary to Theorem 
\ref{thm:essential-surjectivity-of-psi-lower-star}
where the equivalence $\psi_*\xi_*\Gamma_{W'}\simeq \Gamma_{W'}$ 
depends on a path between $\theta$ and the obvious morphism,
the equivalence in \cite[Corollary 4.34]{BGV}
is canonical.
\end{rmk}

\section{The motivic Hopf algebroid of a Weil cohomology theory}

\label{sect:-motivic-Hopf-algebra-weil-coh}

In this section, we associate to every Weil cohomology theory $\Gamma_W$
a motivic Hopf algebroid $\mathcal{H}_{\mot}(\Gamma_W)$. 
When $\Gamma_W$ is the Betti cohomology theory 
associated to a complex embedding 
of the ground field, we recover the motivic Hopf algebra introduced 
and studied in \cite{gal-mot-1,gal-mot-2}.
We start by recalling the notions of group and groupoid objects
in a general $\infty$-category following \cite[Definition 6.1.2.7]{Lurie}.

\begin{dfn}
\label{dfn:group-object-in-infty-category}
Let $\mathcal{C}$ be an $\infty$-category. 
A groupoid in $\mathcal{C}$ is
a cosimplicial object $G:\Deltasimp^{\op} \to \mathcal{C}$ 
such that, for every integer $n\geq 0$ and 
every covering $\{0,\ldots, n\}=I\cup J$ 
with $I\cap J=\{m\}$ a singleton, the square
\begin{equation}
\label{eq-dfn:group-object-in-infty-category-1}
\xymatrix{G(\Delta^{\{0,\ldots,n\}}) \ar[r] \ar[d] & 
G(\Delta^{J})
\ar[d] \\
G(\Delta^I) \ar[r] & G(\Delta^{\{m\}})}
\end{equation}
is Cartesian. We say that $G$ is a group if moreover
$G(\Delta^0)$ is a final object.
\end{dfn}

\begin{dfn}
\label{dfn:spectral-group-scheme-equiv-to-hopf-algebra}
Let $\mathcal{C}^{\otimes}$ be a symmetric monoidal $\infty$-category.
A Hopf algebroid (resp. algebra) in $\mathcal{C}$ 
is a cosimplicial object 
$H:\Deltasimp\to \CAlg(\mathcal{C})$
such that the corresponding simplicial object in 
$\CAlg(\mathcal{C})^{\op}$ is a groupoid
(resp. group). We denote by $\Hopf(\mathcal{C})$
the full sub-$\infty$-category of $\CAlg(\mathcal{C})^{\Deltasimp}$
spanned by Hopf algebroids. When 
$\mathcal{C}^{\otimes}=\Mod_{\Lambda}^{\otimes}$,
for a commutative ring spectrum $\Lambda$, we write 
$\Hopf(\Lambda)$ instead of $\Hopf(\Mod_{\Lambda})$.
\end{dfn}

\begin{lemma}
\label{lemma:connectivity-hopf-algebroid}
Let $\Lambda$ be a commutative ring spectrum, and let
$H\in \Hopf(\Lambda)$ be a Hopf algebroid. The following conditions
are equivalent.
\begin{enumerate}

\item For every integer $n\geq 0$, the algebra 
$H(\Delta^n)$ is connective.

\item The algebras $H(\Delta^0)$ and $H(\Delta^1)$ are connective.  

\end{enumerate}
When these conditions are satisfied, we say that the Hopf algebroid
$H$ is connective. 
\end{lemma}

\begin{proof}
Indeed, by definition, there is an equivalence
$$H(\Delta^n)
\simeq 
\overbrace{H(\Delta^1)\otimes_{H(\Delta^0)}\cdots \otimes_{H(\Delta^0)}H(\Delta^1)}^{n\;\text{times}}.$$
This proves the implication (2)$\Rightarrow$(1).
The other implication is obvious.
\end{proof}

\begin{dfn}
\label{dfn:representation-groupoid-comodule-Hopf}
Let $\mathcal{C}^{\otimes}$ be a symmetric monoidal $\infty$-category
and let $H$ be a Hopf algebroid in $\mathcal{C}$.
An $H$-comodule is a module $M \in \Mod_H(\mathcal{C}^{\Deltasimp})$
such that, for all integers $0\leq m\leq n$, the natural morphism
$$M(\Delta^{\{m\}})\otimes_{H(\Delta^{\{m\}})}
H(\Delta^{\{0,\ldots,n\}})
\to M(\Delta^{\{0,\ldots,n\}})$$
is an equivalence. We denote by $\coMod_H(\mathcal{C})$
the full sub-$\infty$-category of $\Mod_H(\mathcal{C}^{\Deltasimp})$
spanned by $H$-comodules. When $\mathcal{C}^{\otimes}=
\Mod_{\Lambda}^{\otimes}$, for a commutative ring spectrum $\Lambda$, 
we write $\coMod_H$ instead of $\coMod_H(\Mod_{\Lambda})$.
\end{dfn}

\begin{rmk}
\label{rmk:relation-group-and-groupoid-and-hopf}
Let $G$ be a groupoid object in an $\infty$-category 
$\mathcal{C}$. If $\mathcal{C}$ admits finite limits, 
the simplicial object $G\times_{\Cech(G_0)}G_0$,
where $\Cech_{\bullet}(G_0)$ is the {\v C}ech nerve of $G_0$,
is a group object in $\mathcal{C}_{/G_0}$.
Similarly, let $\mathcal{C}^{\otimes}$ be a 
presentable symmetric monoidal $\infty$-category
and let $H$ be a Hopf algebroid in a  $\mathcal{C}$. Then 
the cosimplicial algebra $H\otimes_{\Cech(H^0)}H^0$,
where $\Cech^{\bullet}(H^0)$ is the {\v C}ech conerve of $H^0$, is a 
Hopf algebra in $\Mod_{H^0}(\mathcal{C})$. (The
{\v C}ech conerve of a commutative algebra is explicitated in Notation 
\ref{not:cobar-of-a-commutative-algebra-object-in-}
below.) Given a commutative ring spectrum $\Lambda$ and a Hopf algebroid 
$H\in \Hopf(\Lambda)$, if $\Lambda$ and $H$ are connective, then the associated Hopf $H^0$-algebra $H\otimes_{\Cech(H^0)}H^0$ is also
connective.
\end{rmk}

\begin{nota}
\label{not:cobar-of-a-commutative-algebra-object-in-}
Let $\mathcal{C}^{\otimes}$ be a symmetric monoidal 
$\infty$-category and let $A\in \CAlg(\mathcal{C})$ be a
commutative algebra in $\mathcal{C}$. 
The {\v C}ech conerve of $A$ is the cosimplicial
commutative algebra
$$\Cech^{\bullet}(A):\Deltasimp \to \CAlg(\mathcal{C})$$
which is the left Kan extension along the inclusion 
$\Deltasimp^{\leq 0}\subset \Deltasimp$
of the functor sending the unique object $[0]$ of $\Deltasimp^{\leq 0}$
to $A\in \CAlg(\mathcal{C})$.
Informally, $\Cech^{\bullet}(A)$ is given as follows.
\begin{enumerate}

\item For $n\in \N$, we have $\Cech^n(A)=A^{\otimes n+1}$.

\item For $0 \leq i \leq n+1$, the $i$-th face morphism 
$\Cech^n(A) \to \Cech^{n+1}(A)$ is given by 
$$A^{\otimes i} \otimes \one \otimes A^{\otimes n+1-i}
\xrightarrow{\rm u} 
A^{\otimes i} \otimes A \otimes A^{\otimes n+1-i}$$
where $u$ is the unit of $A$.

\item For $0\leq i \leq n-1$, the $i$-th codegeneracy morphism
$\Cech^n(A) \to \Cech^{n-1}(A)$ is given by
$$A^{\otimes i} \otimes (A \otimes A)\otimes A^{\otimes n-i-1}
\xrightarrow{\rm m}
A^{\otimes i} \otimes A \otimes  A^{\otimes n-i-1}$$
where ${\rm m}$ is the multiplication of $A$.

\end{enumerate}
The {\v C}ech conerve of $A$ has 
a natural augmentation given by 
$\Cech^{-1}(A)=\one$.
\end{nota}

\begin{situ}
\label{situ:for-having-hopf-algebra-in-general}
Let $e:\mathcal{C}^{\otimes} \to \mathcal{M}^{\otimes}$ be 
a colimit-preserving 
symmetric monoidal functor between presentable symmetric monoidal 
$\infty$-categories, and denote by 
$d:\mathcal{M} \to \mathcal{C}$ its right adjoint.
Let $A$ be a commutative
algebra in $\mathcal{M}^{\otimes}$, and 
assume that the induced functor 
$$\Mod_{d(A)}(\mathcal{C}) \to \Mod_A(\mathcal{M})$$
is an equivalence of $\infty$-categories.
More concretely, the following two conditions
are satisfied. 
\begin{enumerate}

\item For every $A$-module $M$, the obvious morphism 
$A\otimes_{ed(A)}ed(M) \to M$
is an equivalence.

\item For every $d(A)$-module $L$, the obvious morphism 
$L \to d(A\otimes_{ed(A)}e(L))$
is an equivalence.

\end{enumerate}
\end{situ}

\begin{rmk}
\label{rmk:main-case-of-situ-for-hopf-algebra}
With $k$, $K$ and $\Lambda$ as in the previous sections, 
we are mainly interested in the following two instances of 
Situation
\ref{situ:for-having-hopf-algebra-in-general}:
\begin{enumerate}

\item $\mathcal{C}^{\otimes}=
\Mod_{\Lambda}^{\otimes}$, $\mathcal{M}^{\otimes}=
\SH_{\et}(k;\Lambda)^{\otimes}$, $e$ the obvious functor
and $A=\mathbf{\Gamma}_W$ a Weil spectrum on algebraic $k$-varieties;

\item $\mathcal{C}^{\otimes}=
\Mod_{\Lambda}^{\otimes}$, $\mathcal{M}^{\otimes}=
\RigSH_{\et}(K;\Lambda)^{\otimes}$, $e$ the obvious functor
and $A=\mathbf{\Gamma}_W$ a Weil spectrum on rigid analytic
$K$-varieties.

\end{enumerate}
The assumption on $A$ is satisfied in both cases by Propositions
\ref{prop:every-module-is-extended-constant}
and
\ref{prop:every-module-is-extended-constant-rig-anal}.
Also, in these cases, the functor $d$ is given by 
$\Gamma(\pt;\Omega^{\infty}_T(-))$.
\end{rmk}

\begin{thm}
\label{thm:motivic-hopf-algebra-from-Weil-spectrum}
We work in Situation
\ref{situ:for-having-hopf-algebra-in-general}.
\begin{enumerate}

\item The cosimplicial commutative algebra
$d(\Cech(A))$ is a Hopf algebroid in 
$\mathcal{C}$. 

\item For every object $M\in \mathcal{M}$, the 
$d(\Cech(A))$-module 
$d(\Cech(A)\otimes M)$
is a comodule over the Hopf algebroid 
$d(\Cech(A))$.

\end{enumerate}
Moreover, we have a symmetric monoidal functor 
\begin{equation}
\label{eq-thm:motivic-hopf-algebra-from-Weil-spectrum-1}
d(\Cech(A)\otimes -):
\mathcal{M}^{\otimes}
\to \coMod_{d(\Cech(A))}(\mathcal{C})^{\otimes}.
\end{equation}
\end{thm}

\begin{proof} 
To prove (1), we need to show that for every 
partition $\{0,\ldots, n\}=I\cup J$ with 
$I\cap J=\{m\}$ a singleton, the natural map 
\begin{equation}
\label{eq-thm:motivic-hopf-algebra-from-Weil-spectrum-3}
d(\Cech^I(A))
\otimes_{d(\Cech^{\{m\}}(A))} d(\Cech^J(A))
\to d(\Cech^{\{0,\ldots, n\}}(A))
\end{equation}
is an equivalence in $\mathcal{C}$. To do so, 
we start by noting that we have an equivalence
\begin{equation}
\label{eq-thm:motivic-hopf-algebra-from-Weil-spectrum-5}
\Cech^I(A)
\otimes_{\Cech^{\{m\}}(A)}
\Cech^J(A)
\xrightarrow{\sim} \Cech^{\{0,\ldots, n\}}(A)
\end{equation}
in $\mathcal{M}$.
By assumption, we also have equivalences
$$\Cech^{\{m\}}(A)\otimes_{ed(\Cech^{\{m\}}(A))}
ed(\Cech^L(A))
\xrightarrow{\sim} \Cech^L(A)$$
for all subsets $L\subset \{0,\ldots, n\}$ containing $m$.
This shows that the domain of the equivalence in 
\eqref{eq-thm:motivic-hopf-algebra-from-Weil-spectrum-5}
is equivalent to
$$\Cech^{\{m\}}(A)\otimes_{ed(\Cech^{\{m\}}(A))}
e(d(\Cech^I(A))\otimes_{d(\Cech^{\{m\}}(A))}
d(\Cech^J(A)))$$
whereas its codomain is equivalent to 
$$\Cech^{\{m\}}(A)\otimes_{ed(\Cech^{\{m\}}(A))}
ed(\Cech^{\{0,\ldots, n\}}(A)).$$
Thus, up to natural identifications, we see that the equivalence 
\eqref{eq-thm:motivic-hopf-algebra-from-Weil-spectrum-5}
can be obtained from the morphism 
\eqref{eq-thm:motivic-hopf-algebra-from-Weil-spectrum-3}
by applying $A\otimes_{ed(A)}e(-)$.
Since this functor is an equivalence by assumption, 
the result follows.

The proof of (2) is very similar to that of (1), but we include
it for the reader's convenience. Here, for every 
integers $0\leq m \leq n$, we need to check that
the obvious morphism 
\begin{equation}
\label{eq-thm:motivic-hopf-algebra-from-Weil-spectrum-11}
d(\Cech^{\{0,\ldots, n\}}(A))
\otimes_{d(\Cech^{\{m\}}(A))}
d(\Cech^{\{m\}}(A)\otimes M)
\to d(\Cech^{\{0,\ldots,n\}}(A)\otimes M)
\end{equation}
is an equivalence in $\mathcal{C}$. 
As above, we start by noting that we have an equivalence
\begin{equation}
\label{eq-thm:motivic-hopf-algebra-from-Weil-spectrum-13}
\Cech^{\{0,\ldots, n\}}(A)
\otimes_{\Cech^{\{m\}}(A)}
(\Cech^{\{m\}}(A)\otimes M)
\xrightarrow{\sim} \Cech^{\{0,\ldots,n\}}(A)
\otimes M
\end{equation}
in $\mathcal{M}$. By assumption, we have equivalences 
$$\Cech^{\{m\}}(A)
\otimes_{ed(\Cech^{\{m\}}(A))}
ed(\Cech^L(A)\otimes N)
\xrightarrow{\sim} 
\Cech^L(A)\otimes N$$
for all subsets $L\subset \{0,\ldots, n\}$ containing $m$ 
and all objects $N\in \mathcal{M}$. 
Using this for $L=\{m\}$ or $L=\{0,\ldots, n\}$
and $N=\one_{\mathcal{M}}$ or $N=M$, we see that, up to natural 
identifications, the equivalence in 
\eqref{eq-thm:motivic-hopf-algebra-from-Weil-spectrum-13}
can be obtained from the morphism in
\eqref{eq-thm:motivic-hopf-algebra-from-Weil-spectrum-11} 
by applying the functor $A\otimes_{ed(A)}e(-)$. We then conclude using that
this functor is an equivalence.
\end{proof}

\begin{dfn}
\label{dfn:the-motivic-Hopf-algebra-for-Weil-spectrum}
Applying Theorem
\ref{thm:motivic-hopf-algebra-from-Weil-spectrum}
to the situations described in Remark 
\ref{rmk:main-case-of-situ-for-hopf-algebra}
we obtain Hopf algebroids 
$$\Hmot(\Gamma_W)=
\Gamma(\pt;\Omega^{\infty}_T(\Cech(\mathbf{\Gamma}_W))).$$
These are the motivic Hopf algebroid associated to the Weil
cohomology theory $\Gamma_W=\Omega^{\infty}_T(\mathbf{\Gamma}_W)$.
The functor corresponding to
\eqref{eq-thm:motivic-hopf-algebra-from-Weil-spectrum-1}
is called the motivic realization associated to $\Gamma_W$
and will be denoted by
\begin{equation}
\label{eq-dfn:the-motivic-Hopf-algebra-for-Weil-spectrum-1}
\Rder^*_{W,\,\mot}:
\SH_{\et}(k;\Lambda) \to 
\coMod_{\Hmot(\Gamma_W)}
\end{equation}
in the algebraic setting, 
and similarly in the rigid analytic setting.
\end{dfn}

\begin{rmk}
\label{rmk:motivic-realization-refines-plain-one}
Clearly, the motivic realization refines the plain realization 
associated to $\Gamma_W$: we have a commutative triangle
$$\xymatrix{\SH_{\et}(k;\Lambda) \ar[r]^-{\Rder^*_{W,\,\mot}} 
\ar[dr]_-{\Rder^*_W} & \coMod_{\Hmot(\Gamma_W)} \ar[d]^-{\rm ff}\\
& \Mod_{\Gamma_W(\pt)}}$$
where ${\rm ff}$ is the obvious forgetful functor.
The same applies in the rigid analytic setting.
\end{rmk}

\begin{conj}
\label{conj:genralization-of-conserv-conj}
Let $\Gamma_W\in \WCT(k;\Lambda)$ 
be a Weil cohomology theory for algebraic $k$-varieties.
Assume that the commutative $\Lambda$-algebra
$\Gamma_W(k)$ is connective and faithfully flat.
Then, the motivic realization functor
\eqref{eq-dfn:the-motivic-Hopf-algebra-for-Weil-spectrum-1}
associated to $\Gamma_W$ becomes fully faithful when restricted 
to the thick stable sub-$\infty$-category 
$\SH_{\et,\,\ct}(k;\Lambda) \subset \SH_{\et}(k;\Lambda)$ 
generated by $\M(X)(n)$, for $X\in \Sm_k$ and $n\in \Z$.
\end{conj}

\begin{rmk}
\label{rmk:the-localization-property-too-optimistic}
The full faithfulness of the functor
$\SH_{\et,\,\ct}(k;\Lambda)
\to \Mod_{\Hmot(\Gamma_W)}$,
when $\Gamma_W$ is Betti cohomology with rational coefficients, 
is equivalent to  
\cite[\S2.4, Conjecture B]{gal-mot-1}.
A similar conjecture is also expected in the rigid analytic setting.
\end{rmk}

\begin{lemma}
\label{lem:comparing-of-hopf-algebroid-for-wct}
Let $\Gamma_W \to \Gamma_W'$ be a morphism 
of Weil cohomology theories for algebraic $k$-varieties
(or rigid analytic $K$-varieties).
Then, we have an equivalence of Hopf algebroids 
$$\Hmot(\Gamma_W)\otimes_{\Cech(\Gamma_W(\pt))}\Cech(\Gamma_{W'}(\pt))
\xrightarrow{\sim}
\Hmot(\Gamma_{W'}).$$
\end{lemma}

\begin{proof}
Indeed, by Proposition 
\ref{prop:the-right-fibration-WCT-to-CAlg},
we have an equivalence 
$\Gamma_W\otimes_{\Gamma_W(\pt)}\Gamma_{W'}(\pt)\simeq 
\Gamma_{W'}$. This induces an equivalence of cosimplicial algebras in 
$\SH_{\et}(k;\Lambda)$:
$$\Cech(\mathbf{\Gamma}_W)\otimes_{\Cech(\Gamma_W(\pt))}\Cech(\Gamma_{W'}(\pt))\simeq \Cech(\mathbf{\Gamma}_{W'}).$$
Applying $\Gamma(\pt;\Omega_T^{\infty}(-))$ to this equivalence
yields the desired result.
\end{proof}

We note the following simple but useful fact,
where our running assumption that $\Lambda$ is connective
is actually important.

\begin{cor}
\label{cor:simple-useful-criterion-connectedness}
Let $\Gamma_W \to \Gamma_W'$ be a morphism 
of Weil cohomology theories for algebraic $k$-varieties
(or rigid analytic $K$-varieties).
Assume that $\Gamma_{W'}(\pt)$ and $\Hmot(\Gamma_W)$ 
are connective. Then, $\Hmot(\Gamma_{W'})$ is also connective. 
\end{cor}

\begin{proof}
Since $\Lambda$ is connective, the cosimplicial commutative
algebras  
$\Cech(\Gamma_W(\pt))$ and $\Cech(\Gamma_{W'}(\pt))$
are degreewise connective.
Thus, the result follows from Lemma
\ref{lem:comparing-of-hopf-algebroid-for-wct}.
\end{proof}

To ease comparison with the construction in 
\cite{gal-mot-1}, we end this 
section with an alternative construction of the
motivic Hopf algebroid of a Weil cohomology theory 
in term of the associated realization functor.

\begin{cons}
\label{cons:hopf-algebra-in-term-of-realization-funtcor}
In Situation 
\ref{situ:for-having-hopf-algebra-in-general},
we let $f:\mathcal{M}^{\otimes} \to 
\Mod_{d(A)}(\mathcal{C})^{\otimes}$
be the composite functor
$$\mathcal{M}^{\otimes} \to \Mod_A(\mathcal{M})^{\otimes}
\simeq \Mod_{d(A)}(\mathcal{C})^{\otimes},$$
and $g:\Mod_{d(A)}(\mathcal{C}) \to \mathcal{M}$ 
its right adjoint. (Note that in the situations described in Remark 
\ref{rmk:main-case-of-situ-for-hopf-algebra}, the functor $f$ is the 
associated realization functor $\Rder^*_W$.)
By \cite[Proposition 4.7.3.3]{Lurie-HA}
applied to the functor $(f)^{\op}$, the composite functor 
$g\circ f$ underlies a coalgebra structure in the 
$\infty$-category ${\rm EndFun}(\mathcal{M}^{\otimes})$
of lax symmetric monoidal endofunctors of 
$\mathcal{M}^{\otimes}$. Said differently, there is a cosimplicial 
object $\Phi^{\bullet}_f:\Deltasimp\to {\rm EndFun}(\mathcal{M}^{\otimes})$
which is informally given as follows.
\begin{enumerate}

\item For $n\in \N$, we have 
$\Phi^n_f=(g\circ f)^{\circ n+1}$.

\item For $0 \leq i \leq n+1$, the $i$-th face morphism 
$\Phi_f^n \to \Phi_f^{n+1}$ is given by 
$$(g\circ f)^{\circ i} \circ \id \circ (g\circ f)^{\circ n+1-i}
\xrightarrow{\eta} 
(g\circ f)^{\circ i} \circ (g\circ f) \circ (g\circ f)^{\circ n+1-i}$$
where $\eta$ is the unit of the adjunction $(f,g)$.

\item For $0\leq i \leq n-1$, the $i$-th codegeneracy morphism
$\Phi_f^n \to \Phi_f^{n-1}$ is given by
$$g\circ (f\circ g)^{\circ i} \circ (f\circ g)\circ (f\circ g)^{\circ n-i-1}
\circ f \xrightarrow{\delta}
g\circ (f\circ g)^{\circ i} \circ \id \circ (f\circ g)^{\circ n-i-1}
\circ f$$
where $\delta$ is the counit of the adjunction 
$(f,g)$.

\end{enumerate}
We may think of $\Phi_f^{\bullet}$ as a right-lax symmetric monoidal
functor 
$\Phi_f:\mathcal{M}^{\otimes} \to 
(\mathcal{M}^{\Deltasimp})^{\otimes}$.
In particular, $\Phi_f(\one)$ is a cosimplicial 
commutative algebra in $\mathcal{M}^{\otimes}$. 
\end{cons}

\begin{lemma}
\label{lemma:equivalence-phif-one-with-cobar-on-A}
We work in Situation
\ref{situ:for-having-hopf-algebra-in-general}
and we keep the notation from Construction
\ref{cons:hopf-algebra-in-term-of-realization-funtcor}.
There is an equivalence of cosimplicial commutative algebras
$\Cech(A) \xrightarrow{\sim} \Phi_f(\one)$.
Moreover, for $M\in \mathcal{M}$, the $\Cech(A)$-module 
$\Cech(A)\otimes M$ is equivalent to the $\Phi_f(\one)$-module 
$\Phi_f(M)$. 
\end{lemma}

\begin{proof}
By construction, we have $\Phi_f(\one)(\Delta^0)=g\circ f(\one)=A$.
By the universal property of $\Cech(A)$, we deduce a morphism of 
cosimplicial commutative algebras
$\Cech(A) \to \Phi_f(\one)$.
More generally, for $M\in \mathcal{M}$, we have a morphism 
$\Cech(A)\otimes M \to \Phi_f(M)$
given by the composition of 
$$\Cech(A)\otimes M \to \Phi_f(\one)\otimes M
\to \Phi_f(\one)\otimes \Phi_f(M) \to \Phi_f(M).$$
(Recall that the cosimplicial functor $\Phi_f$ is
right-lax monoidal.) 
Using the natural equivalence 
$g\circ f\simeq A\otimes -$, it is easy to see that 
$\Cech(A)\otimes M \to \Phi_f(M)$ is an equivalence.
\end{proof}

\begin{thm}
\label{thm:hopf-algebroid-using-plain-realization}
We work in Situation
\ref{situ:for-having-hopf-algebra-in-general}
and we keep the notation from Construction
\ref{cons:hopf-algebra-in-term-of-realization-funtcor}.
\begin{enumerate}

\item The cosimplicial commutative algebra $d(\Phi_f(\one))$
is a Hopf algebroid in $\mathcal{C}^{\otimes}$ which is equivalent
to $d(\Cech(A))$.

\item For $M\in \mathcal{M}$, the $d(\Phi_f(\one))$-module 
$d(\Phi_f(M))$ is a comodule over the Hopf algebroid 
$d(\Phi_f(\one))$ which is equivalent to the comodule 
$d(\Cech(A)\otimes M)$ over $d(\Cech(A))$.

\end{enumerate}
Moreover, we have a symmetric monoidal functor
\begin{equation}
\label{eq-thm:hopf-algebroid-using-plain-realization-1}
d(\Phi_f(-)):\mathcal{M}^{\otimes} \to 
\coMod_{d(\Phi_f(\one))}(\mathcal{C})^{\otimes}
\end{equation}
which is equivalent to the functor 
\eqref{eq-thm:motivic-hopf-algebra-from-Weil-spectrum-1}.
\end{thm}

\begin{proof}
This follows immediately from Theorem 
\ref{thm:motivic-hopf-algebra-from-Weil-spectrum}
and Lemma 
\ref{lemma:equivalence-phif-one-with-cobar-on-A}.
\end{proof}

\begin{rmk}
\label{rmk:alternative-construction-of-motivic-hopf-alg}
Let $\Gamma_W$ be a Weil cohomology theory for algebraic 
$k$-varieties (or rigid analytic $K$-varieties), 
and let $\Rder^*_W$ be the associated realization functor. 
Then the motivic Hopf algebroid $\Hmot(\Gamma_W)$ is
equivalent to $\Gamma(\pt;\Omega^{\infty}_T(\Phi_{\Rder^*_W}(\Lambda)))$
and the associated motivic realization functor is equivalent to 
$\Gamma(\pt;\Omega^{\infty}_T(\Phi_{\Rder^*_W}(-)))$.
In particular, note that $\Hmot(\Gamma_W)(\Delta^1)=
\Rder^*_W\Rder_{W,\,*}\Gamma_W(\pt)$.
\end{rmk}

\section{Complements}

\label{sect:complements}

In this section, we gather a few more facts about 
Weil cohomology theories and their motivic Hopf algebroids. 
As in the previous sections, we fix a connective commutative
ring spectrum $\Lambda$, and let $k$ be 
a ground field whose exponent characteristic 
is invertible in $\pi_0(\Lambda)$. We also let 
$K$ be a field endowed with a rank $1$ valuation, 
and whose residue field is $k$.

\begin{thm}
\label{thm:relation-between-weil-coh-theories-and-extension}
Let $k'/k$ be a field extension and let 
$e:\Spec(k') \to \Spec(k)$ be the associated morphism. 
The right-lax monoidal functor $e_*:\SH_{\et}(k';\Lambda)
\to \SH_{\et}(k;\Lambda)$ takes Weil spectra to Weil spectra 
inducing a functor 
\begin{equation}
\label{eq-thm:relation-between-weil-coh-theories-and-exten-1}
e_*:\WSp(k';\Lambda) \to \WSp(k;\Lambda).
\end{equation}
Assume now that the extension $k'/k$ is algebraic, then
the same is true for the symmetric monoidal functor 
$e^*:\SH_{\et}(k;\Lambda) \to \SH_{\et}(k';\Lambda)$
and the induced functor
\begin{equation}
\label{eq-thm:relation-between-weil-coh-theories-and-exten-3}
e^*:\WSp(k;\Lambda) \to \WSp(k';\Lambda)
\end{equation}
is left adjoint to the functor
\eqref{eq-thm:relation-between-weil-coh-theories-and-exten-1}.
Moreover, given a Weil spectrum $\mathbf{\Gamma}_W\in 
\WSp(k;\Lambda)$, the coefficient ring of the Weil 
spectrum $e^*\mathbf{\Gamma}_W$ is equivalent to 
$\Rder_W^*(e_*\Lambda)$ where $\Rder^*_W$
is the realization functor associated to $\Gamma_W$.
\end{thm}

The first assertion is obvious. The proof of the part concerning 
$e^*$ is similar to the proof of 
Theorem \ref{thm:rig-upper-star-weil-coh-theory}. We will need
the following analog of Theorem
\ref{thm:the-main-equivalence-for-new-Weil-coh}.

\begin{prop}
\label{prop:SH-et-of-k-prime-is-module-over-e-star}
Let $k'/k$ be an algebraic extension and let 
$e:\Spec(k') \to \Spec(k)$ be the associated morphism. 
Then the functor 
$$\widetilde{e}_*:\SH_{\et}(k';\Lambda)
\to \SH_{\et}(k;e_*\Lambda)$$
is an equivalence of $\infty$-categories. 
\end{prop}

\begin{proof}
The functor $\widetilde{e}_*$ admits a left adjoint 
$\widetilde{e}{}^*$ sending an $e_*\Lambda$-module $M$ 
to $e^*(M)\otimes_{e^*e_*\Lambda}\Lambda$.
It is enough to prove the following two properties.
\begin{enumerate}

\item The image of the functor 
$\widetilde{e}^*$ generates $\SH_{\et}(k';\Lambda)$
under colimits. 

\item The unit of the adjunction 
$\id \to \widetilde{e}_*\widetilde{e}{}^*$
is an equivalence. 

\end{enumerate}
The first property follows from Lemma
\ref{lemma:e-star-generates-sh-k-prime} below. 
To prove the second property, we use Lemmas
\ref{lemma:generation-by-strongly-dualizable-objects}
and \ref{lemma:commutation-with-colimits-for-e-lower-star}
to reduce to showing that the unit of the adjunction is 
an equivalence when evaluated at objects of the form
$M\otimes e_*\Lambda$, with $M\in \SH_{\et}(k;\Lambda)$
dualizable. The resulting morphism coincides with
$M\otimes e_*\Lambda
\to e_*e^*(M)$ which is indeed an equivalence by
\cite[Lemme 2.8]{gal-mot-1}.
\end{proof}

\begin{lemma}
\label{lemma:e-star-generates-sh-k-prime}
Let $k'/k$ be an algebraic extension and let 
$e:\Spec(k') \to \Spec(k)$ be the associated morphism. 
The image of the functor 
$e^*:\SH_{\et}(k;\Lambda)\to \SH_{\et}(k';\Lambda)$
generates $\SH_{\et}(k';\Lambda)$ under colimits.
\end{lemma}

\begin{proof}
Using \cite[Theorem 2.9.7]{AGAV}, we may replace 
$k$ and $k'$ with their perfections and assume that $k'/k$ is a 
separable extension. If $X\in \Sm_{k'}$, then $X$ can be defined
over a finite separable sub-extension $l/k$ of $k'/k$. 
But, if $X_0\in \Sm_l$ is such that $X_0\otimes_l k'\simeq X$, 
then $X$ is a clopen subscheme of $X_0\otimes_k k'$.
It follows that the motive of $X$ is a direct summand of the base change
along $e$ of the motive of $X_0$. This enables us to conclude.
\end{proof}

\begin{lemma}
\label{lemma:commutation-with-colimits-for-e-lower-star}
Let $k'/k$ be an algebraic extension and let 
$e:\Spec(k') \to \Spec(k)$ be the associated morphism. 
The functor 
$e_*:\SH_{\et}(k';\Lambda)\to \SH_{\et}(k;\Lambda)$
is colimit-preserving.
\end{lemma}

\begin{proof}
The functor 
$e_*:\Shv_{t_{\amalg}}(\Sm_{k'}) \to
\Shv_{t_{\amalg}}(\Sm_k)$
preserves \'etale local equivalences. 
(Indeed, if $A$ is a strictly henselian ring over $k$, then 
$A\otimes_k k'$ can be written as a filtered colimit 
of finite products of strictly henselian rings.)
It follows that 
\begin{equation}
\label{eq-lemma:commut-with-colimits-for-e-lower-star-1}
e_*:\Shv^{\wedge}_{\et}(\Sm_{k'};\Lambda)
\to \Shv^{\wedge}_{\et}(\Sm_k;\Lambda)
\end{equation}
is also colimit-preserving, and the same applies for 
the $\infty$-categories of $T$-prespectra. Since 
$\SH_{\et}(k;\Lambda)$
and $\SH_{\et}(k';\Lambda)$ are obtained from these 
$\infty$-categories of $T$-prespectra by localizations
that are compatible with the functor
\eqref{eq-lemma:commut-with-colimits-for-e-lower-star-1}, 
the result follows.
\end{proof}

\begin{proof}[Proof of Theorem 
\ref{thm:relation-between-weil-coh-theories-and-extension}]
We need to show that $e^*\mathbf{\Gamma}_W$ is a Weil spectrum. 
By Proposition
\ref{prop:SH-et-of-k-prime-is-module-over-e-star},
the functor $e^*$ is equivalent to 
$-\otimes e_*\Lambda:\SH_{\et}(k;\Lambda)
\to \SH_{\et}(k;e_*\Lambda)$.
Thus, it suffices to prove that the functor
$\Mod_{\Rder_W^*(e_*\Lambda)} \to \SH_{\et}(k;\mathbf{\Gamma}_W\otimes e_*\Lambda)$
is an equivalence, which follows from the fact that the commutative algebra
$\mathbf{\Gamma}_W\otimes e_*\Lambda\simeq 
\mathbf{\Gamma}_W\otimes_{\Gamma_W(\pt)}\Rder_W^*(e_*\Lambda)$
is a Weil spectrum over $k$.
\end{proof}

We also have the analogs of Theorem
\ref{thm:relation-between-weil-coh-theories-and-extension}
and Proposition
\ref{prop:SH-et-of-k-prime-is-module-over-e-star}
in the rigid analytic setting.
For the ease of reference, we state these results leaving the 
details of their proofs to the reader.

\begin{thm}
\label{thm:relation-bet-weil-coh-theories-and-exten-rigan}
Let $K'/K$ be an extension of height $1$ valued fields and let
$e:\Spa(\widehat{K}{}')
\to \Spa(\widehat{K})$ be the induced morphism. 
The right-lax monoidal 
functor $e_*:\RigSH_{\et}(K';\Lambda)
\to \RigSH_{\et}(K;\Lambda)$ takes Weil spectra to Weil spectra 
inducing a functor 
\begin{equation}
\label{thm:rel-bet-weil-coh-theories-and-exten-rigan-1}
e_*:\RigWSp(K';\Lambda) \to \RigWSp(K;\Lambda).
\end{equation}
Assume now that the extension $K'/K$ is algebraic, then
the same is true for the symmetric monoidal functor 
$e^*:\RigSH_{\et}(K;\Lambda) \to 
\RigSH_{\et}(K';\Lambda)$
and the induced functor
\begin{equation}
\label{thm:rel-bet-weil-coh-theories-and-exten-rigan-3}
e^*:\RigWSp(K;\Lambda) \to \RigWSp(K';\Lambda)
\end{equation}
is left adjoint to the functor
\eqref{thm:rel-bet-weil-coh-theories-and-exten-rigan-1}.
Moreover, given a Weil spectrum $\mathbf{\Gamma}_W\in 
\RigWSp(K;\Lambda)$, the coefficient ring of the Weil 
spectrum $e^*\mathbf{\Gamma}_W$ is equivalent to 
$\Rder_W^*(e_*\Lambda)$ where $\Rder^*_W$
is the realization functor associated to $\Gamma_W$.
\end{thm}

\begin{prop}
\label{prop:RigSH-et-of-k-prime-is-module-over-e-star}
Let $K'/K$ be an algebraic extension of height $1$ valued fields and let
$e:\Spa(\widehat{K}{}')
\to \Spa(\widehat{K})$ be the induced morphism.
Then the functor 
$$\widetilde{e}_*:\RigSH_{\et}(K';\Lambda)
\to \RigSH_{\et}(K;e_*\Lambda)$$
is an equivalence of $\infty$-categories. 
\end{prop}

The next results describe the relation between motivic 
Hopf algebroids and the classical Galois groups of algebraic extensions.

\begin{prop}
\label{prop:hopf-algebroid-of-Gamma-in-finite-ext}
Let $k'/k$ be a quasi-Galois algebraic extension
with Galois group $G$. Let $\Gamma_{W'}$ be a Weil cohomology theory 
for algebraic $k'$-varieties and let $\Gamma_W=(k'/k)_*\Gamma_{W'}$.
Then, there is a cocartesian diagram of Hopf algebroids
$$\xymatrix{\Lambda^{\B_{\bullet}(G)} \ar[r] \ar[d] 
& \Hmot(\Gamma_W)
\ar[d]\\
\Lambda \ar[r] & \Hmot(\Gamma_{W'}).\!}$$
In particular, if $\Hmot(\Gamma_W)$ is connective, then so it 
$\Hmot(\Gamma_{W'})$.
\end{prop}

\begin{proof}
Modulo the equivalence in Proposition 
\ref{prop:SH-et-of-k-prime-is-module-over-e-star},
the Weil spectrum $\mathbf{\Gamma}_{W'}\in \WSp(k';\Lambda)$
corresponds to an $e_*\Lambda$-algebra structure on the 
Weil spectrum $\mathbf{\Gamma}_W\in \SH_{\et}(k;\Lambda)$. 
Thus, we see that 
$$\Hmot(\Gamma_W)=\Gamma(\pt;\Omega^{\infty}_T\Cech(\mathbf{\Gamma}_W))
\quad \text{and} \quad
\Hmot(\Gamma_{W'})=\Gamma(\pt;\Omega^{\infty}_T\Cech(\mathbf{\Gamma}_W/e_*\Lambda)).$$
Now, there is a commutative diagram with cocartesian squares 
$$\xymatrix{\Lambda^{\B(G)} \ar[r] \ar[d] & \Cech(e_*\Lambda)
\ar[r] \ar[d] & \Cech(\mathbf{\Gamma}_W) \ar[d] \\
\Lambda \ar[r] & e_*\Lambda \ar[r] & \Cech(\mathbf{\Gamma}_W/e_*\Lambda)}$$
in $\CAlg(\SH_{\et}(k;\Lambda))^{\Deltasimp}$, where 
$\B_{\bullet}(G)$ is the classifying space of the profinite group 
$G$ which can be identified with the simplicial profinite set
$\pi_0(\Cech^{\bullet}(k'/k))$.
It follows that 
$\Cech(\mathbf{\Gamma}_W/e_*\Lambda)\simeq 
\Cech(\mathbf{\Gamma}_W)\otimes_{\Lambda^{\B(G)}}\Lambda$.
To conclude, it remains to see that the obvious morphism
$$\Rder\Gamma(\pt;\Omega^{\infty}_T(\Cech(\mathbf{\Gamma}_W)))
\otimes_{\Lambda^{\B(G)}}\Lambda
\to \Rder\Gamma(\pt;\Omega^{\infty}_T(\Cech(\mathbf{\Gamma}_W)
\otimes_{\Lambda^{\B(G)}}\Lambda))$$
is an equivalence, and we may prove this degreewise.
For $n\geq 0$, 
$\Cech^n(\mathbf{\Gamma}_W)$ is a Weil spectrum 
with ring of coefficients 
$A=\Rder\Gamma(\pt;\Omega^{\infty}_T(\Cech^n(\mathbf{\Gamma}_W)))$.
Using Corollary
\ref{cor:algebra-over-weil-spectrum-weil-spectrum},
we deduce that 
$\Cech^n(\mathbf{\Gamma}_W)\otimes_{\Lambda^{\B_n(G)}}\Lambda$
is a Weil spectrum with ring of coefficients
$A\otimes_{\Lambda^{\B_n(G)}}\Lambda$.
This gives the desired equivalence.
\end{proof}

\begin{prop}
\label{prop:hopf-algebroid-of-Gamma-in-finite-ext-rigan}
Let $K'/K$ be an algebraic extension of height $1$
valued fields, such that $\widehat{K}{}'/\widehat{K}$
is quasi-Galois with Galois group $G$.
Let $\Gamma_{W'}$ be a Weil cohomology theory 
for rigid analytic $K'$-varieties and let $\Gamma_W=(K'/K)_*\Gamma_{W'}$.
Then, there is a cocartesian diagram of Hopf algebroids
$$\xymatrix{\Lambda^{\B_{\bullet}(G)} \ar[r] \ar[d] 
& \Hmot(\Gamma_W)
\ar[d]\\
\Lambda \ar[r] & \Hmot(\Gamma_{W'}).\!}$$
In particular, if $\Hmot(\Gamma_W)$ is connective, then so it 
$\Hmot(\Gamma_{W'})$.
\end{prop}

\begin{proof}
The proof is identical to the proof of Proposition
\ref{prop:hopf-algebroid-of-Gamma-in-finite-ext}.
\end{proof}

\begin{thm}
\label{thm:completion-of-motivic-Hopf-algebras}
Let $k^{\alg}/k$ be an algebraic closure and let 
$G$ be the Galois group of $k^{\alg}/k$. Let $\Gamma_{W'}$ 
be a Weil cohomology theory for algebraic $k^{\alg}$-varieties
and let $\Gamma_W=(k^{\alg}/k)_*\Gamma_{W'}$.
Then, the natural morphism
$\Lambda^{\B(G)} \to \Hmot(\Gamma_W)$
from Proposition
\ref{prop:hopf-algebroid-of-Gamma-in-finite-ext}
induces an equivalence
\begin{equation}
\label{eq-thm:completion-of-motivic-Hopf-algebras-1}
\left(\Cech(\Gamma_W(\pt))\otimes \Lambda^{\B(G)}\right)^{\wedge}_{\ell} \xrightarrow{\sim} 
\Hmot(\Gamma_W)^{\wedge}_{\ell}
\end{equation}
after $\ell$-completion. 
\end{thm}

\begin{proof}
The morphism
\eqref{eq-thm:completion-of-motivic-Hopf-algebras-1}
is clearly an equivalence in cosimplicial degree zero. 
We need to see that
it is also an equivalence in cosimplicial degree $1$, i.e., 
that the morphism
$$\left(\Gamma_W(\pt)^{\otimes 2}\otimes \Lambda^G\right)^{\wedge}_{\ell}\to 
\left(\Hmot(\Gamma_W)(\Delta^1)\right)^{\wedge}_{\ell}$$
is an equivalence.
Recall from Remark
\ref{rmk:alternative-construction-of-motivic-hopf-alg}
that $\Hmot(\Gamma_W)(\Delta^1)$
is equivalent to $\Rder_W^*\Rder_{W,\,*}\Gamma_W(\pt)$
where $\Rder_W^*:\SH_{\et}(k;\Lambda) \to \Mod_{\Gamma_W(\pt)}$
is the realization functor associated to $\Gamma_W$. 
The right adjoint $\Rder_{W,\,*}:\Mod_{\Gamma_W(\pt)}
\to \SH_{\et}(k;\Lambda)$ is colimit-preserving since it is 
equivalent to the forgetful functor 
$\SH_{\et}(k;\mathbf{\Gamma}_W) \to \SH_{\et}(k;\Lambda)$.
In particular, it belongs to $\Prl$ and we have a commutative diagram
$$\xymatrix{\Mod_{\Gamma_W(\pt)}
\ar[r]^-{\Rder_{W,\,*}} \ar[d]^-{(-)^{\wedge}_{\ell}} & 
\SH_{\et}(k;\Lambda) \ar[r]^-{\Rder_W^*} \ar[d]^-{(-)^{\wedge}_{\ell}}
& \Mod_{\Gamma_W(\pt)} \ar[d]^-{(-)^{\wedge}_{\ell}}\\
(\Mod_{\Gamma_W(\pt)})_{\ellcpl}
\ar[r]^-{\widehat{\Rder}_{W,\,*}} & 
\SH_{\et}(k;\Lambda)_{\ellcpl} \ar[r]^-{\widehat{\Rder}{}_W^*} 
& (\Mod_{\Gamma_W(\pt)})_{\ellcpl}}$$
where $\widehat{\Rder}_{W,\,*}$ is the 
right adjoint of $\widehat{\Rder}{}_W^*$.
This gives an equivalence
$$\left(\Hmot(\Gamma_W)(\Delta^1)\right)^{\wedge}_{\ell}
\simeq 
\widehat{\Rder}{}^*_W\widehat{\Rder}{}_{W,\,*}\Gamma_W(\pt)^{\wedge}_{\ell}.$$
By rigidity \cite{bachmann-rigidity} (see also 
\cite[Theorem 2.10.4]{AGAV}), 
we have a commutative diagram where the 
vertical arrows are equivalences: 
$$\xymatrix{(\Mod_{\Gamma_W(\pt)})_{\ellcpl} 
\ar[r]^-{\widehat{f}_*} \ar@{=}[d] & 
\Shv_{\et}(\Et_k;\Lambda)_{\ellcpl} \ar[r]^-{\widehat{f}{}^*} 
\ar[d]^-{\sim}
& (\Mod_{\Gamma_W(\pt)})_{\ellcpl} \ar@{=}[d]\\
(\Mod_{\Gamma_W(\pt)})_{\ellcpl}
\ar[r]^-{\widehat{\Rder}_{W,\,*}} & 
\SH_{\et}(k;\Lambda)_{\ellcpl} \ar[r]^-{\widehat{\Rder}{}_W^*} 
& (\Mod_{\Gamma_W(\pt)})_{\ellcpl}.\!}$$
Recall that $\Shv_{\et}(\Et_k;\Lambda)$ is the $\infty$-category of 
\'etale hypersheaves of $\Lambda$-modules on the small \'etale site 
$\Et_k$ of $k$. 
This gives an equivalence
$$\left(\Hmot(\Gamma_{\ell,\,k})(\Delta^1)\right)^{\wedge}_{\ell}
\simeq \widehat{f}{}^*\widehat{f}_*\Gamma_W(\pt)^{\wedge}_{\ell}.$$

We claim that $\widehat{f}{}^*$ 
is the $\ell$-completion of the composite functor  
$$f^*:\Shv_{\et}(\Et_k;\Lambda) \xrightarrow{(k^{\alg}/k)^*} 
\Mod_{\Lambda} \to \Mod_{\Gamma_W(\pt)}.$$
Indeed, since $\Gamma_W=(k^{\alg}/k)_*\Gamma_{W'}$, we have 
a natural factorization $\Rder_W^*=\Rder_{W'}^*\circ (k^{\alg}/k)^*$
where $\Rder_{W'}^*:\SH_{\et}(k^{\alg};\Lambda) \to \Mod_{\Gamma_W(\pt)}$
is the realization functor associated to $\Gamma_{W'}$.
It follows that $\widehat{f}{}^*$ is equal to the composition of 
$$\Shv_{\et}(\Et_k;\Lambda)_{\ellcpl}
\xrightarrow{(k^{\alg}/k)^*}
\Shv_{\et}(\Et_{k^{\alg}};\Lambda)_{\ellcpl}
\xrightarrow{\sim} 
\SH_{\et}(k^{\alg};\Lambda)_{\ellcpl}
\xrightarrow{\widehat{\Rder}{}^*_{W'}}
(\Mod_{\Gamma_W(\pt)})_{\ellcpl}.$$
But the composition of the last two functors has to be the obvious one
since it is a morphism of $\Mod^{\otimes}_{\Lambda}$-modules 
in $\Prl$. This proves our claim. 
To conclude, we use the following commutative diagram
$$\xymatrix{\Mod_{\Gamma_W(\pt)}
\ar[r]^-{f_*} \ar[d]^-{(-)^{\wedge}_{\ell}} & 
\Shv_{\et}(\Et_k;\Lambda) \ar[r]^-{f^*} 
\ar[d]^-{(-)^{\wedge}_{\ell}} & \Mod_{\Gamma_W(\pt)}
\ar[d]^-{(-)^{\wedge}_{\ell}}\\
(\Mod_{\Gamma_W(\pt)})_{\ellcpl} 
\ar[r]^-{\widehat{f}_*} & 
\Shv_{\et}(\Et_k;\Lambda)_{\ellcpl} \ar[r]^-{\widehat{f}{}^*} 
& (\Mod_{\Gamma_W(\pt)})_{\ellcpl}}$$
showing that 
$$\widehat{f}{}^*\widehat{f}_*\Gamma_W(\pt)^{\wedge}_{\ell}
\simeq (f^*f_*\Gamma_W(\pt))^{\wedge}_{\ell}
\simeq (\Gamma_W(\pt)^{\otimes}\otimes \Lambda^G)^{\wedge}_{\ell}.$$
This finishes the proof.
\end{proof}

\begin{thm}
\label{thm:completion-of-motivic-Hopf-algebras-rigan}
Let $K^{\alg}/K$ be an algebraic closure of $K$.
Choose an extension of the valuation of $K$ to $K^{\alg}$, and
let $G$ be the Galois group of $\widehat{K}{}^{\alg}/\widehat{K}$.
Let $\Gamma_{W'}$ be a Weil cohomology theory for rigid analytic 
$K^{\alg}$-varieties and let $\Gamma_W=(K^{\alg}/K)_*\Gamma_{W'}$.
Then, the natural morphism 
$\Lambda^{\B(G)} \to \Hmot(\Gamma_{\ell,\,K})$
from Proposition 
\ref{prop:hopf-algebroid-of-Gamma-in-finite-ext-rigan}
induces an equivalence 
\begin{equation}
\label{eq-thm:completion-of-motivic-Hopf-algebras-rigan-2}
\left(\Cech(\Gamma_W(\pt))\otimes \Lambda^{\B(G)}\right)^{\wedge}_{\ell} \xrightarrow{\sim} 
\Hmot(\Gamma_{\ell,\,K})^{\wedge}_{\ell}
\end{equation}
after $\ell$-completion.
\end{thm}

\begin{proof}
The proof is identical to the proof of Theorem 
\ref{thm:completion-of-motivic-Hopf-algebras}.
\end{proof}

In the remainder of this section, we will establish some
criteria for proving the connectivity of motivic
Hopf algebras. We start with a converse to the last assertion in
Proposition
\ref{prop:hopf-algebroid-of-Gamma-in-finite-ext}.

\begin{prop}
\label{prop:reduction-algeb-closed-connectivity}
Let $k'/k$ be a quasi-Galois algebraic extension
with Galois group $G$. Let $\Gamma_{W'}$ be a Weil cohomology theory 
for algebraic $k'$-varieties and let $\Gamma_W=(k'/k)_*\Gamma_{W'}$.
Assume that the following properties are satisfied.
(Below, we write $\pt'$ for $\Spec(k')$.)
\begin{enumerate}

\item The Hopf algebroid $\Hmot(\Gamma_{W'})$ is connective. 

\item For every $g\in G$, there is a faithfully flat 
$\Gamma_{W'}(\pt')$-algebra $A$ such that the Weil cohomology theories 
$\Gamma_{W'}\otimes_{\Gamma_{W'}(\pt')}A$ and 
$g_*(\Gamma_{W'})\otimes_{\Gamma_{W'}(\pt')}A$
are equivalent.

\end{enumerate}
Then, the Hopf algebroid $\Hmot(\Gamma_W)$ is connective.
\end{prop}

\begin{proof}
Using \cite[Theorem 2.9.7]{AGAV}, we may replace 
$k$ and $k'$ with their perfections and assume that $k'/k$ is a 
separable extension. Then $G$ is precisely the group of automorphisms
of the extension $k'/k$.
By Lemma
\ref{lemma:connectivity-hopf-algebroid},
it is enough to show that 
$$\Hmot(\Gamma_W)(\Delta^1)=\Rder \Gamma(\pt;\Omega^{\infty}_T(
\mathbf{\Gamma}_W\otimes \mathbf{\Gamma}_W))$$
is connective. By Proposition
\ref{prop:hopf-algebroid-of-Gamma-in-finite-ext},
there is a morphism 
$\Lambda^G \to \Hmot(\Gamma_W)(\Delta^1)$,
and it is enough to show that, for every 
prime ideal $\mathfrak{p}$ of $\pi_0(\Lambda)^{G}$,
the localized algebra 
$\Hmot(\Gamma_W)(\Delta^1)_{\mathfrak{p}}$ is connective.
Since every prime ideal of $\pi_0(\Lambda)^{G}$ is in the image of 
the morphism $\Spec(\Lambda^{\{g\}}) \to \Spec(\Lambda^G)$ 
for a unique $g\in G$, it is enough to show that 
$\Hmot(\Gamma_W)(\Delta^1)\otimes_{\Lambda^{G}}\Lambda^{\{g\}}$ 
is connective for every $g\in G$. Arguing as we did at the end
of the proof of Proposition
\ref{prop:hopf-algebroid-of-Gamma-in-finite-ext}, 
we have:
$$\begin{array}{rcl}
\Hmot(\Gamma_W)(\Delta^1)\otimes_{\Lambda^{G}}\Lambda^{\{g\}}
& = & 
\Rder \Gamma(\pt;\Omega^{\infty}_T(
\mathbf{\Gamma}_W\otimes \mathbf{\Gamma}_W))
\otimes_{\Lambda^{G}}\Lambda^{\{g\}}\\
& \simeq & \Rder \Gamma(\pt;\Omega^{\infty}_T(
(\mathbf{\Gamma}_W\otimes \mathbf{\Gamma}_W)
\otimes_{\Lambda^{G}}\Lambda^{\{g\}})).
\end{array}$$
On the other hand, we have the following equivalences
of commutative algebras in
$\SH_{\et}(k;e_*\Lambda)$:
$$\begin{array}{rcl}
(\mathbf{\Gamma}_W\otimes \mathbf{\Gamma}_W)
\otimes_{\Lambda^{G}}\Lambda^{\{g\}}
& \simeq & 
(\mathbf{\Gamma}_W\otimes \mathbf{\Gamma}_W)
\otimes_{e_*\Lambda\,\otimes\, e_*\Lambda, \id \otimes g^*}e_*\Lambda\\
&\simeq & \mathbf{\Gamma}_W \otimes_{e_*\Lambda}(\mathbf{\Gamma}_W\otimes_{e_*\Lambda,\,g^*}e_*\Lambda),
\end{array}$$
where $g^*=\Spec(g)^*:e_*\Lambda \to e_*\Lambda$ is the action of 
$g\in G$ on the commutative algebra $e_*\Lambda$ which is  
induced from the morphism $\Spec(g):\Spec(k') \to \Spec(k')$.
Modulo the equivalence of $\infty$-categories
$\SH_{\et}(k;e_*\Lambda)\simeq \SH_{\et}(k';\Lambda)$, 
the $e_*\Lambda$-algebra $\mathbf{\Gamma}_W$ 
corresponds to $\mathbf{\Gamma}_{W'}$ and the
$e_*\Lambda$-algebra $\mathbf{\Gamma}_W\otimes_{e_*\Lambda,\,g^*}e_*\Lambda$ 
corresponds to $g^*\mathbf{\Gamma}_{W'}\simeq g^{-1}_*\mathbf{\Gamma}_{W'}$.
From this, we deduce the equivalence
$$\Hmot(\Gamma_W)(\Delta^1)\otimes_{\Lambda^{G}}\Lambda^{\{g\}}
\simeq \Rder\Gamma(\pt';\Omega^{\infty}_T(\mathbf{\Gamma}_{W'}
\otimes g^*\mathbf{\Gamma}_{W'})).$$
By the second assumption in the statement, we can find a faithfully 
flat $\Gamma_{W'}(\pt')$-algebra $A$ such that 
$g^*(\mathbf{\Gamma}_{W'})\otimes_{\Gamma_{W'}(\pt')}A$
is equivalent to $\mathbf{\Gamma}_{W'}\otimes_{\Gamma_{W'}(\pt')}A$. It follows that 
$$\begin{array}{rcl}
\left(\Hmot(\Gamma_W)(\Delta^1)\otimes_{\Lambda^{G}}
\Lambda^{\{g\}}\right)\otimes_{\Gamma_{W'}(\pt')^{\otimes 2}}A^{\otimes 2}
& \simeq & \Rder\Gamma(\pt';\Omega^{\infty}_T(\mathbf{\Gamma}_{W'}
\otimes \mathbf{\Gamma}_{W'}))\otimes_{\Gamma_{W'}(\pt')^{\otimes 2}}A^{\otimes 2}\\
& \simeq & \left(\Hmot(\Gamma_{W'})(\Delta^1)\right)\otimes_{\Gamma_{W'}(\pt')^{\otimes 2}}A^{\otimes 2}.
\end{array}$$
Since $\Hmot(\Gamma_{W'})(\Delta^1)$ is connective by assumption, the 
result follows.
\end{proof}

\begin{prop}
\label{prop:reduction-algeb-closed-connectivity-rigan}
Let $K'/K$ be an algebraic extension of height $1$
valued fields, such that $\widehat{K}{}'/\widehat{K}$
is quasi-Galois with Galois group $G$. Let $\Gamma_{W'}$ 
be a Weil cohomology theory 
for rigid analytic $K'$-varieties and let $\Gamma_W=(K'/K)_*\Gamma_{W'}$.
Assume that the following properties are satisfied.
(Below, we write $\pt'$ for $\Spa(\widehat{K}{}')$.)
\begin{enumerate}

\item The Hopf algebroid $\Hmot(\Gamma_{W'})$ is connective. 

\item For every $g\in G$, there is a faithfully flat 
$\Gamma_{W'}(\pt')$-algebra $A$ such that the Weil cohomology theories 
$\Gamma_{W'}\otimes_{\Gamma_{W'}(\pt')}A$ and 
$g_*(\Gamma_{W'})\otimes_{\Gamma_{W'}(\pt')}A$
are equivalent.

\end{enumerate}
Then, the Hopf algebroid $\Hmot(\Gamma_W)$ is connective.
\end{prop}

\begin{proof}
The proof is identical to the proof of Proposition
\ref{prop:reduction-algeb-closed-connectivity}.
Notice that the functor 
$\sigma_*$ between $\RigWCT(K';\Lambda)$ and
$\RigWCT(K;\Lambda)$ is well defined since these categories 
depend only on the completed fields $\widehat{K}$ and $\widehat{K}{}'$.
\end{proof}

We will need also another criterion for the connectivity
of motivic Hopf algebras.

\begin{thm}
\label{thm:connectivity-rigid-anal-vs-algebraic-}
We work in Situation
\ref{situ:finite-rank-value-group-rho-for-psi}.
Let $\mathbf{\Gamma}_{W'}\in \RigWSp(K;\Lambda)$ 
be a Weil spectrum 
and set $\mathbf{\Gamma}_W=\xi_*\mathbf{\Gamma}_{W'}$. 
Assume that $\Lambda$ is a $\Q$-algebra, and that
$\Gamma_{W'}(\pt)$ and $\Gamma_{W'}(\pt)(1)$ 
are connective.
Then $\Hmot(\Gamma_W)$ is connective if and only if 
$\Hmot(\Gamma_{W'})$ is connective. 
\end{thm}

Theorem 
\ref{thm:connectivity-rigid-anal-vs-algebraic-}
follows from Theorem
\ref{thm:essential-surjectivity-of-psi-lower-star},
asserting that $\mathbf{\Gamma}_{W'}$ is equivalent 
to $\psi_*\mathbf{\Gamma}_W$, and Theorem 
\ref{thm:connectivity-algebraic-vs-rigid-analytic}
below describing the motivic Hopf algebroid of 
$\psi_*\mathbf{\Gamma}_W$ in term of the motivic Hopf algebroid
of $\mathbf{\Gamma}_W$. To state Theorem 
\ref{thm:connectivity-algebraic-vs-rigid-analytic},
we need the notion of a semi-direct tensor product of Hopf algebroids.

\begin{cons}
\label{cons:semi-direct-product-Hopf-algebroid}
Let $\mathcal{C}^{\otimes}$ be a presentable symmetric monoidal 
$\infty$-category. Let $H$ be a Hopf algebroid in $\mathcal{C}$ and
let $L$ be a Hopf algebroid in $\coMod_H(\mathcal{C})$. 
Then $H$ is given by a cosimplicial commutative algebra 
$H:\Deltasimp\to \CAlg(\mathcal{C})$ and 
$L$ is given by a cosimplicial commutative algebra
$$L:\Deltasimp\to \CAlg(\Mod_H(\mathcal{C}^{\Deltasimp}))=\CAlg(\mathcal{C}^{\Deltasimp})_{H\backslash}.$$
Equivalently, we can view $L$ as a bi-cosimplicial commutative algebra 
$L:\Deltasimp^2\to \CAlg(\mathcal{C})$
together with an augmentation $H \to L$. It is easy to see that 
$$L(\Delta^m,\Delta^n)  \simeq  H(\Delta^m)\otimes L(\Delta^0,\Delta^n)
\simeq  H(\Delta^1)^{\otimes m}\otimes L(\Delta^0,\Delta^1)^{\otimes n}.$$
This implies easily that the diagonal cosimplicial algebra 
${\rm diag}(L)$ is a Hopf algebra in $\mathcal{C}$. 
We denote it by $H\otimes^{\rhd} L$ and call it the semi-direct
tensor product of $H$ with $L$.
\end{cons}

\begin{thm}
\label{thm:connectivity-algebraic-vs-rigid-analytic}
We work in Situation
\ref{situ:finite-rank-value-group-rho-for-psi}.
Let $\mathbf{\Gamma}_W\in \WSp(k;\Lambda)$ be a Weil 
spectrum and set $\mathbf{\Gamma}_{W'}=\psi_*\mathbf{\Gamma}_W$.
There is an equivalence of Hopf algebroids
\begin{equation}
\label{eq-thm:connectivity-algebraic-vs-rigid-analytic-1}
\Hmot(\Gamma_{W'})\simeq \Hmot(\Gamma_W)\otimes^{\rhd} {\rm S}
(\Gamma_W(\pt)_{\Q}^{\oplus n}(-1)).
\end{equation}
In particular, 
$\Hmot(\mathbf{\Gamma}_W)$ is connective if and only if 
$\Hmot(\mathbf{\Gamma}_{W'})$ is connective.
\end{thm}

\begin{proof}
Recall that we have an equivalence
$\RigSH_{\et}(K;\Lambda)\simeq \SH_{\et}(k;\xi_*\Lambda)$, 
modulo which $\psi^*$ is given by base change along the 
obvious co-augmentation 
$\xi_*\Lambda \simeq {\rm S}(\Lambda^{\oplus n}_{\Q}(-1)[-1])
\to \Lambda$. (See Construction
\ref{cons:the-vanishing-cycles-functor-in-general}.)
Unravelling the definitions, it follows that 
$\mathbf{\Gamma}_{W'}=\psi_*\mathbf{\Gamma}_W$ is equivalent to 
$\mathbf{\Gamma}_W$ viewed as a commutative algebra in 
$\SH_{\et}(k;\xi_*\Lambda)$ using the composite morphism
$$\xi_*\Lambda \to \Lambda \to \mathbf{\Gamma}_W.$$
Thus, the motivic Hopf algebra $\Hmot(\Gamma_{W'})$ is obtained 
by applying $\Gamma(\pt;\Omega^{\infty}_T(-))$ to the cosimplicial algebra
$\Cech^{\bullet}(\mathbf{\Gamma}_W/\xi_*\Lambda)$ 
in $\SH_{\et}(k;\xi_*\Lambda)$. Consider the bi-cosimplicial algebra 
$$A^{\bullet,\bullet}=
\Cech^{\bullet}(\mathbf{\Gamma}_W)\otimes\Cech^{\bullet}(\Lambda/\xi_*\Lambda).$$
Since $\Cech^{\bullet}(\Lambda/\xi_*\Lambda)$
is naturally a cosimplicial $\xi_*\Lambda$-algebra, 
we may view $A^{\bullet,\bullet}$ as a bi-cosimplicial 
algebra in $\SH_{\et}(k;\xi_*\Lambda)$.
Moreover, the $\xi_*\Lambda$-algebra 
$A^{0,0}$ is equivalent to $\mathbf{\Gamma}_W$. 
Thus, by the universal property of left Kan extensions, there is a 
unique morphism of cosimplicial $\xi_*\Lambda$-algebras
$$\Cech(\mathbf{\Gamma}_W/\xi_*\Lambda)
\to {\rm diag}(A).$$
We claim that this morphism is an equivalence.
Indeed, in degree $n$, we have
$$\begin{array}{rcl}
A^{n,n} & = & \overbrace{\mathbf{\Gamma}_W\otimes \cdots \otimes \mathbf{\Gamma}_W}^{n+1\,\text{times}}
\otimes \overbrace{\Lambda \otimes_{\xi_*\Lambda} \cdots 
\otimes_{\xi_*\Lambda}\Lambda}^{n+1\,\text{times}}\\
& \simeq & \overbrace{
(\mathbf{\Gamma}_W\otimes \Lambda)\otimes_{\xi_*\Lambda}\cdots 
\otimes_{\xi_*\Lambda}(\mathbf{\Gamma}_W\otimes \Lambda)}^{n+1\,\text{times}}.
\end{array}$$
Modulo this identification, 
the morphism $A^{0,0} \to A^{n,n}$ corresponding to $\Delta^{\{i\}}
\to \Delta^{\{0,\ldots,n\}}$
is given by
$$\mathbf{\Gamma}_W\otimes \Lambda
\simeq 
(\xi_*\Lambda)^{\otimes_{\xi_*\Lambda}\,i} 
\otimes_{\xi_*\Lambda}(\mathbf{\Gamma}_W\otimes \Lambda)
\otimes_{\xi_*\Lambda}(\xi_*\Lambda)^{\otimes_{\xi_*\Lambda}\,n-i}
\xrightarrow{u\otimes \id \otimes u} (\mathbf{\Gamma}_W\otimes \Lambda)^{\otimes_{\xi_*\Lambda}\,n+1}.$$
This proves our claim. Thus, at this point, we have proven that 
$$\Hmot(\Gamma_{W'})\simeq 
{\rm diag}(\Rder\Gamma(\pt;\Omega^{\infty}_T(A))).$$

Recall that $\xi_*\Lambda={\rm S}(\Lambda_{\Q}^{\oplus n}(-1)[-1])$.
Since ${\rm S}(-)$ is a left adjoint functor, it follows that 
$\Cech^{\bullet}(\Lambda/\xi_*\Lambda)$ is equivalent to
the free commutative algebra on the cosimplicial object of $\SH_{\et}(k;\xi_*\Lambda)$:
$$\Cech^{\bullet}(0/\Lambda_{\Q}^{\oplus n}(-1)[-1])
\simeq 
\underline{\Hom}(\B_{\bullet}(\Z),\Lambda_{\Q}^{\oplus n}(-1)).$$
Said differently, $\Cech(\Lambda/\xi_*\Lambda)$ is a Hopf algebra
whose underlying commutative algebra is 
${\rm S}(\Lambda_{\Q}^{\oplus n}(-1))$.
By abuse of notation, we will write 
${\rm S}(\Lambda_{\Q}^{\oplus n}(-1))$
instead of $\Cech(\Lambda/\xi_*\Lambda)$.
Now, notice that the bi-cosimplicial 
algebra $\Rder\Gamma(\pt;\Omega^{\infty}_T(A))$
encodes the Hopf algebra 
$$\Rder_{W,\,\mot}^*({\rm S}(\Lambda_{\Q}^{\oplus n}(-1)))=
{\rm S}(\Gamma_W(\pt)_{\Q}^{\oplus n}(-1))$$
in $\coMod_{\Hmot(\Gamma_W)}$
obtained by applying the motivic realization functor 
(of Definition 
\ref{dfn:the-motivic-Hopf-algebra-for-Weil-spectrum})
to the Hopf algebra ${\rm S}(\Lambda_{\Q}^{\oplus n}(-1))$.
Since $\Hmot(\Gamma_{W'})$ is the diagonal of 
$\Rder\Gamma(\pt;\Omega^{\infty}_T(A))$, this gives the 
equivalence 
\eqref{eq-thm:connectivity-algebraic-vs-rigid-analytic-1}
in the statement.
\end{proof}

\section{Examples of Weil cohomology theories}

\label{sect:example-of-weil-coh-theo}

In this section, we recall the constructions of the 
classical Weil cohomology 
theories, and we revisit the new Weil cohomology theories 
introduced in \cite{Nouvelle-Weil}.
As in the previous sections, we fix a connective commutative
ring spectrum $\Lambda$, and let $k$ be 
a ground field whose exponent characteristic 
is invertible in $\pi_0(\Lambda)$. We also let 
$K$ be a field endowed with a rank $1$ valuation, 
and whose residue field is $k$.
We start with the $\ell$-adic cohomology theories.

\begin{nota}
\label{nota:ell-complete-in-stable-infty-cat}
Let $\ell$ be a prime number. Given a presentable stable 
$\infty$-category $\mathcal{C}$, we denote by 
$\mathcal{C}_{\ell}\subset \mathcal{C}$
its full sub-$\infty$-category of $\ell$-complete objects
and by $(-)^{\wedge}_{\ell}:\mathcal{C} \to \mathcal{C}_{\ell}$ 
the $\ell$-completion functor.
(See \cite[Chapter 7]{Lurie-SAG} for a detailed discussion
and \cite[Notation 2.10.1 \& Remark 2.10.2]{AGAV} for a very quick
review.) Below, we will assume that $\ell$ is invertible in $k$
and we denote by 
$\Lambda_{\ell}$ the $\ell$-completion of $\Lambda$.
\end{nota}

\begin{cons}
\label{cons:the-ell-adic-realization-functor-}
Assume that $k$ is algebraically closed. By rigidity
\cite{bachmann-rigidity} (see also 
\cite[Theorem 2.10.4]{AGAV}), we have an equivalence
of $\infty$-categories
$$(\Mod_{\Lambda})_{\ellcpl} \xrightarrow{\sim} \
\SH_{\et}(k;\Lambda)_{\ellcpl}.$$
Recall that 
$\SH_{\et,\,\ct}(k;\Lambda) \subset \SH_{\et}(k;\Lambda)$
is the thick stable sub-$\infty$-category
generated by the motives $\M(X)(n)$, for $X\in \Sm_k$ and $n\in \Z$.
In fact, $\SH_{\et}(k;\Lambda)$
is compactly generated, and
$\SH_{\et,\,\ct}(k;\Lambda)$ is its sub-$\infty$-category of 
compact objects by \cite[Proposition 3.2.3]{AGAV}. 
By Lemma 
\ref{lemma:generation-by-strongly-dualizable-objects},
the composite functor
\begin{equation}
\label{eq-cons:the-ell-adic-realization-functor-17}
\SH_{\et,\,\ct}(k;\Lambda) \to 
\SH_{\et}(k;\Lambda) \xrightarrow{(-)^{\wedge}_{\ell}}
\SH_{\et}(k;\Lambda)_{\ellcpl}\simeq (\Mod_{\Lambda})_{\ellcpl}
\end{equation}
lands in the full sub-$\infty$-category of $(\Mod_{\Lambda})_{\ellcpl}$
spanned by dualizable objects. It also lands in the sub-$\infty$-category
spanned by eventually connective objects. Indeed, by Lemma 
\ref{lemma:generation-by-strongly-dualizable-objects},
it is enough to show that the image of $\pi_{X,\,*}\Lambda$,
for $X\in \Sm_k$, is eventually connective. But this image is given by 
$\lim_n \Rder\Gamma_{\et}(X;\Lambda/\ell^n\Lambda)$ which is
$-2\dim(X)$-connective since the \'etale cohomological dimension 
of $X$ is bounded 
by twice its dimension. 
On the other hand, by \cite[Corollary 8.3.5.9]{Lurie-SAG}, 
the functor 
$$(-)^{\wedge}_{\ell}:\Mod_{\Lambda_{\ell}}
\to (\Mod_{\Lambda})_{\ellcpl}$$
induces an equivalence between the full sub-$\infty$-category 
of dualizable $\Lambda_{\ell}$-modules and the full sub-$\infty$-category
of eventually connective and dualizable objects in 
$(\Mod_{\Lambda})_{\ellcpl}$. Thus, by the previous discussion, 
we see that the functor 
\eqref{eq-cons:the-ell-adic-realization-functor-17}
yields a functor 
\begin{equation}
\label{eq-cons:the-ell-adic-realization-functor-23}
\Rder_{\ell}:\SH_{\et,\,\ct}(k;\Lambda)
\to \Mod_{\Lambda_{\ell}}.
\end{equation}
Since $\SH_{\et}(k;\Lambda)$ is equivalent to the indization of 
$\SH_{\et,\,\ct}(k;\Lambda)$, the functor
\eqref{eq-cons:the-ell-adic-realization-functor-23} 
extends uniquely to a colimit-preserving functor 
$$\Rder^*_{\ell}:\SH_{\et}(k;\Lambda) \to \Mod_{\Lambda_{\ell}},$$
which we call the plain $\ell$-adic realization functor. 
Clearly, $\Rder^*_{\ell}$ underlies a symmetric monoidal 
functor. We denote by $\mathbf{\Gamma}_{\ell}=
\Rder_{\ell,\,*}\Lambda_{\ell}$ and 
$\Gamma_{\ell}=\Omega^{\infty}_T(\mathbf{\Gamma}_{\ell})$
the Weil spectrum and the Weil cohomology theory 
associated to $\Rder^*_{\ell}$.
To stress the dependence on $k$, we would write 
`$\Rder^*_{\ell,\,k}$', `$\mathbf{\Gamma}_{\ell,\,k}$'
and `$\Gamma_{\ell,\,k}$'.
\end{cons}

\begin{rmk}
\label{rmk:ell-adic-realization-for-rigid-analytic-mot}
Construction 
\ref{cons:the-ell-adic-realization-functor-}
extends to the rigid analytic setting. 
Indeed, if $K$ is algebraically closed, there is 
an equivalence of $\infty$-categories
$$(\Mod_{\Lambda})_{\ellcpl} \xrightarrow{\sim} 
\RigSH_{\et}(K;\Lambda)_{\ellcpl}$$
by \cite[Theorem 2.10.3]{AGAV}.
Moreover, Lemma
\ref{lemma:generation-good-reduction-dualizability}
implies that the composite functor analogous to
\eqref{eq-cons:the-ell-adic-realization-functor-17}
lands in the full sub-$\infty$-category of eventually connective and 
dualizable objects. Using \cite[Corollary 8.3.5.9]{Lurie-SAG},
we obtain the plain $\ell$-adic realization functor 
$$\Rder^*_{\ell}:\RigSH_{\et}(K;\Lambda) \to \Mod_{\Lambda_{\ell}}.$$
We also denote by $\mathbf{\Gamma}_{\ell}=
\Rder_{\ell,\,*}\Lambda_{\ell}$ and 
$\Gamma_{\ell}=\Omega^{\infty}_T(\mathbf{\Gamma}_{\ell})$
the Weil spectrum and the Weil cohomology theory 
associated to $\Rder^*_{\ell}$.
To stress the dependence on $K$, we would write 
`$\Rder^*_{\ell,\,K}$', `$\mathbf{\Gamma}_{\ell,\,K}$'
and `$\Gamma_{\ell,\,K}$'.
\end{rmk}

\begin{lemma}
\label{lemma:ell-adic-cohomology-invariant-galois}
Let $\sigma:k \to k'$ be a morphism between algebraically 
closed fields. Then, there is an equivalence 
$\sigma_*\Gamma_{\ell,\,k'}\simeq \Gamma_{\ell,\,k}$
of Weil cohomology theories
on algebraic $k$-varieties. A similar statement is also true in the 
rigid analytic setting.
\end{lemma}

\begin{proof}
We prove that the associated plain 
realization functors agree, and it is enough 
to do so after restrictions to compact objects
since these functors are colimit-preserving.
Thus, we need to show that the following triangle commutes
$$\xymatrix{\SH_{\et,\,\ct}(k;\Lambda) \ar[r]^-{\sigma^*}
\ar[dr]_-{\Rder_{\ell,\,k}^*} 
& \SH_{\et,\,\ct}(k';\Lambda) \ar[d]^-{\Rder^*_{\ell,\,k'}}\\
& \Mod_{\Lambda_{\ell}}.\, }$$
Since the functor $\sigma^*$
commutes with $\ell$-completion, it is enough to show that 
the following triangle commutes
$$\xymatrix{\SH_{\et}(k;\Lambda)_{\ellcpl} 
\ar[r]^-{\sigma^*} &  \SH_{\et}(k';\Lambda)_{\ellcpl}\\
& (\Mod_{\Lambda})_{\ellcpl},\! \ar[u]_-{\sim} \ar[ul]^-{\sim} }$$
which is obvious.
\end{proof}

\begin{dfn}
\label{dfn:ell-adic-cohomo-theory-when-k-nonclosed}
We extend the $\ell$-adic cohomology theories to non necessary 
algebraically closed fields
in the usual way. Given an algebraic closure $k^{\alg}/k$
of $k$, we set 
$\Gamma_{\ell,\,k}=(k^{\alg}/k)_*\Gamma_{\ell,\,k^{\alg}}$.
Similarly, given an algebraic closure $K^{\alg}/K$
and an extension of the valuation of $K$ to $K^{\alg}$,
we set
$\Gamma_{\ell,\,K}=(K^{\alg}/K)_*
\Gamma_{\ell,\,K^{\alg}}$.
\end{dfn}

\begin{rmk}
\label{rmk:independence-of-algebraic-closure-up-to-equiv}
Lemma
\ref{lemma:ell-adic-cohomology-invariant-galois}
implies that the equivalence class of the 
$\ell$-adic cohomology theory $\Gamma_{\ell,\,k}$ 
is independent of the choice of the algebraic closure $k^{\alg}/k$. 
For this reason, when discussing
$\ell$-adic cohomology, we often keep the choice of the algebraic closure 
implicit. The same applies in the rigid analytic setting.
\end{rmk}

\begin{prop}
\label{prop:various-relations-between-ell-adic-coh}
The three $\ell$-adic cohomology theories
$\Gamma_{\ell,\,k}\in \WCT(k;\Lambda)$,
$\Gamma_{\ell,\,K}\in \WCT(K;\Lambda)$ and  
$\Gamma_{\ell,\,K}\in \RigWCT(K;\Lambda)$
are related as follows.
\begin{enumerate}

\item There is an equivalence of Weil spectra on algebraic $K$-varieties
$\mathbf{\Gamma}_{\ell,\,K} \simeq \Rig_*(\mathbf{\Gamma}_{\ell,\,K})$.

\item There is an equivalence of Weil spectra on algebraic $k$-varieties
$\mathbf{\Gamma}_{\ell,\,k}\simeq \xi_*(\mathbf{\Gamma}_{\ell,\,K})$.

\item In Situation 
\ref{situ:finite-rank-value-group-rho-for-psi} 
(and with the notation of Construction 
\ref{cons:the-vanishing-cycles-functor-in-general}),
there is an equivalence of Weil spectra on rigid analytic $K$-varieties
$\mathbf{\Gamma}_{\ell,\,K}
\simeq \psi_*(\mathbf{\Gamma}_{\ell,\,k})$.

\end{enumerate}
\end{prop}

\begin{proof}
We can assume that $K$ is algebraically closed:
for (3) this is the only case we need to consider
and, for (1) and (2), 
we can reduce easily to this case. 
For each equivalence, we will show that the two associated 
plain realizations agree, and it is enough to show this 
after restriction to compact objects since these realizations
are colimit-preserving. 
Thus, we need to show that the following triangles
commute
$$\xymatrix{\SH_{\et,\,\ct}(K;\Lambda) \ar[r]^-{\Rig^*}
\ar[dr]_-{\Rder_{\ell}^*} 
&  \RigSH_{\et,\,\ct}(K;\Lambda) \ar[d]^-{\Rder^*_{\ell}}\\
& \Mod_{\Lambda_{\ell}},\!}
\qquad
\xymatrix{\SH_{\et,\,\ct}(k;\Lambda) \ar[r]^-{\xi^*}
\ar[dr]_-{\Rder_{\ell}^*} 
& \RigSH_{\et,\,\ct}(K;\Lambda) \ar[d]^-{\Rder^*_{\ell}}\\
& \Mod_{\Lambda_{\ell}},\!}$$
$$\xymatrix{\RigSH_{\et,\,\ct}(K;\Lambda) \ar[r]^-{\psi^*} 
\ar[dr]_-{\Rder^*_{\ell}}  & 
\SH_{\et,\,\ct}(k;\Lambda) \ar[d]^-{\Rder^*_{\ell}} \\
& \Mod_{\Lambda_{\ell}}.\!}$$
Since the functors $\Rig^*$, $\xi^*$ and $\psi^*$
commute with $\ell$-completion, it is enough to show that 
the following triangles commute
$$\xymatrix{\SH_{\et}(K;\Lambda)_{\ellcpl} 
\ar[r]^-{\Rig^*} 
& \RigSH_{\et}(K;\Lambda)_{\ellcpl} \\
& (\Mod_{\Lambda})_{\ellcpl},\! \ar[lu]^-{\sim} 
\ar[u]_-{\sim}}
\qquad
\xymatrix{\SH_{\et}(k;\Lambda)_{\ellcpl} 
\ar[r]^-{\xi^*}
& \RigSH_{\et}(K;\Lambda)_{\ellcpl} \\
& (\Mod_{\Lambda})_{\ellcpl},\! 
\ar[lu]^-{\sim} \ar[u]_-{\sim}}$$
$$\xymatrix{\RigSH_{\et}(K;\Lambda)_{\ellcpl} 
\ar[r]^-{\psi^*} & 
\SH_{\et}(k;\Lambda)_{\ellcpl}  \\
& (\Mod_{\Lambda})_{\ellcpl},\! 
\ar[lu]^-{\sim} \ar[u]_-{\sim}}$$
which is obvious.
\end{proof}

We now discuss the Betti cohomology theory.

\begin{cons}
\label{cons:betti-cohomology-in-case-k=complex-number}
Let $\CpSm$ be the category of complex smooth varieties,
and denote by $\cl$ the classical topology on this category. 
The $\infty$-category 
$\AnSH(\Lambda)$, obtained from the $\infty$-category of 
hypersheaves of $\Lambda$-modules $\Shv_{\cl}(\CpSm;\Lambda)$
by $\DD^1$-localisation and $\otimes$-inversion of the Tate object
$T=\Lmot(\Lambda_{\cl}(\PP^1,\infty))$, is equivalent to the 
$\infty$-category $\Mod_{\Lambda}$ via the obvious tensor functor 
$\Mod_{\Lambda} \to \AnSH(\Lambda)$. 
(See for example \cite[Th\'eor\`eme 1.8]{real-betti}.)
On the other hand, the analytification 
functor $\Sm_{\C} \to \CpSm$ gives rise to a functor 
$\An^*:\SH_{\et}(\C;\Lambda) \to \AnSH(\Lambda)$.
We define the Betti realization functor 
$$\Betti^*:\SH_{\et}(\C;\Lambda) \to \Mod_{\Lambda}$$
to be the composite of 
$$\SH_{\et}(\C;\Lambda) \xrightarrow{\An^*}
\AnSH(\Lambda) \simeq \Mod_{\Lambda}.$$
We set $\mathbf{\Gamma}_{\Betti}=\Betti_*\Lambda$ 
and $\Gamma_{\Betti}=\Omega^{\infty}_T(\mathbf{\Gamma}_{\Betti})$.
\end{cons}

\begin{dfn}
\label{dfn:betti-cohomology-theory-on-more-general-fields}
We extend the Betti cohomology theory to any field $k$ endowed with a 
complex embedding $\sigma:k\hookrightarrow \C$ by setting 
$\Gamma_{\Betti,\,k}=\sigma_*(\Gamma_{\Betti})$.
\end{dfn}

\begin{prop}
\label{prop:comparison-thm-betti-coh-l-adic-coh}
Let $\sigma:k\hookrightarrow \C$ be a complex embedding. 
Then, there is a morphism of Weil cohomology theories 
$\Gamma_{\Betti,\,k} \to \Gamma_{\ell,\,k}$
in $\WCT(k;\Lambda)$. (In fact, there is a canonical such a morphism
if we take for $\Gamma_{\ell,\,k}$ the $\ell$-adic cohomology theory 
associated to the algebraic closure $k^{\alg}/k$ of $k$ in $\C$.)
\end{prop}

\begin{proof}
This is a reformulation of the classical comparison theorem 
between Betti cohomology and $\ell$-adic cohomology.
In our framework, it suffices to show that the associated 
plain realization functors coincide after the 
appropriate scalar extension, 
and it is enough to do so for $k=\C$.
But we have a commutative diagram
$$\xymatrix{\SH_{\et,\,\ct}(\C;\Lambda)
\ar[r] \ar[dd]^-{\Rder^*_{\ell}} 
\ar@/^2.5pc/[rrr]^-{\Betti^*} & \SH_{\et}(\C;\Lambda) \ar[d]^-{(-)^{\wedge}_{\ell}}
\ar[r]^-{\An^*} & \AnSH(\Lambda) \ar[d]^-{(-)^{\wedge}_{\ell}} 
\ar[r]^-{\sim} & \Mod_{\Lambda}
\ar[d]^-{(-)^{\wedge}_{\ell}} \\
& \SH_{\et}(\C;\Lambda)_{\ellcpl} \ar[r]^-{\An^*}
& \AnSH(\Lambda)_{\ellcpl} \ar[r]^-{\sim} & (\Mod_{\Lambda})_{\ellcpl}
\ar@{=}[dl]\\
\Mod_{\Lambda_{\ell}} \ar[r]^-{(-)^{\wedge}_{\ell}} & (\Mod_{\Lambda})_{\ellcpl} \ar[u]_-{\sim} \ar[r] & (\Mod_{\Lambda})_{\ellcpl}
\ar[u]_-{\sim}}$$
showing that the symmetric monoidal functors 
$\Betti^*(-) \otimes \Lambda_{\ell}$ and
$\Rder_{\ell}^*$ are equivalent on $\SH_{\et,\,\ct}(\C;\Lambda)$,
and hence also on $\SH_{\et}(\C;\Lambda)$ by indization.
\end{proof}

Next, we discuss the de Rham cohomology theory.

\begin{cons}
\label{cons:de-Rham-cohomology-theory-car-0}
Assume that $k$ has characteristic zero. We fix a morphism 
of commutative ring spectra $\Lambda\to k$ and use it to 
view $k$-modules as $\Lambda$-modules. 
The Zariski sheafification
of the algebraic de Rham complex $\Omega^{\bullet}_{/k}$,
viewed as a presheaf of $\Lambda$-modules 
on $\Sm_k$, 
defines a Weil cohomology theory 
$\Gamma_{\dR}$. This is the algebraic de Rham cohomology
theory. 
We denote by $\mathbf{\Gamma}_{\dR}$
the associated Weil spectrum and
$\dR^*:\SH_{\et}(k;\Lambda) \to \Mod_k$
the associated plain realization. 
To stress the dependence on $k$, we would write `$\Gamma_{\dR,\,k}$', 
`$\mathbf{\Gamma}_{\dR,\,k}$'
and `$\dR^*_k$'.
\end{cons}

The following is a reformulation of the Betti-de Rham comparison 
theorem of Grothendieck.

\begin{prop}
\label{prop:comparison-betti-vs-de-rham}
Let $\sigma:k\hookrightarrow \C$ be a complex embedding
and fix a morphism $\Lambda\to k$.
Then, there is a morphism of Weil cohomology theories 
$\Gamma_{\Betti,\,k} \to \Gamma_{\dR}\otimes_k\C$
in $\WCT(k;\Lambda)$.
\end{prop}

\begin{proof}
See for example
\cite[Proposition 2.88 \& Corollaire 2.89]{gal-mot-1}.
\end{proof}

The last classical Weil cohomology theory that we mention 
is Berthelot's rigid cohomology. 

\begin{cons}
\label{cons:berthelot-weil-cohomology-theories-}
Assume that $K$ has characteristic zero, but allow $k$ to have 
any characteristic. We fix a morphism 
of commutative ring spectra $\Lambda\to \widehat{K}$ and use it to 
view $\widehat{K}$-modules as $\Lambda$-modules. 
There is a Weil cohomology theory $\Gamma_{\dR}^{\dagger}\in 
\RigWCT(K;\Lambda)$ on rigid analytic $K$-varieties
given by overconvergent de Rham cohomology
in the sense of \cite{Gross-Klonne-dR}.
A construction of $\Gamma_{\dR}^{\dagger}$ can be found in 
\cite[Proposition 5.12]{Vezz-MW}.
Roughly speaking, one considers the overconvergent de Rham complex
$\Omega^{\dagger,\,\bullet}_{/K}$
which is a presheaf on the category $\RigSm^{\dagger}_K$
of smooth dagger rigid analytic $K$-varieties.
The \'etale hypersheafification of 
$\Omega^{\dagger,\,\bullet}_{/K}$ 
being $\BB^1$-local, \cite[Theorem 4.23]{Vezz-MW}
shows that $\Lder_{\et}(\Omega^{\dagger,\,\bullet}_{/K})$
factors uniquely through the forgetful functor 
$\RigSm^{\dagger}_K \to \RigSm_K$
yielding the \'etale hypersheaf $\Gamma^{\dagger}_{\dR}$.
Moreover, by \cite[Proposition 5.12]{Vezz-MW}, the Weil spectrum 
$\mathbf{\Gamma}_{\rig}=
\xi_*\mathbf{\Gamma}^{\dagger}_{\dR}$
represents Berthelot's rigid cohomology on algebraic $k$-varieties.
To stress the dependence on $K$, we would write 
`$\Gamma^{\dagger}_{\dR,\,K}$', `$\Gamma_{\rig,\,K}$', etc. 
\end{cons}

\begin{prop}
\label{prop:berthelot-rigid-cohomology-vs-de-rham-}
Fix a morphism of commutative ring spectra $\Lambda\to K$.
The Weil cohomology theories $\Gamma_{\dR,\,K}\in\WCT(K;\Lambda)$
and $\Gamma^{\dagger}_{\dR,\,K}\in \RigWCT(K;\Lambda)$
are related by a morphism of Weil spectra
$\mathbf{\Gamma}_{\dR,\,K} \to 
\Rig_*(\mathbf{\Gamma}^{\dagger}_{\dR,\,K})$
in $\WSp(K;\Lambda)$.
\end{prop}

\begin{proof}
The rigid analytification $X^{\an}$ of an algebraic $K$-variety $X$
has a natural dagger
structure. Thus, there is a functor $(-)^{\an,\,\dagger}:
\Sm_K \to \RigSm^{\dagger}_K$ factoring the usual rigid analytification 
functor. Moreover, there is an obvious morphism of complexes
of presheaves 
$\Omega^{\bullet}_{/K}(-) \to 
\Omega^{\dagger,\,\bullet}_{/K}((-)^{\an,\,\dagger})$.
This induces a morphism of Weil cohomology theories
$\Gamma_{\dR,\,K}\to \Gamma^{\dagger}_{\dR,\,K}\circ (-)^{\an}$
as needed.
\end{proof}

Using Theorem 
\ref{thm:rig-upper-star-weil-coh-theory},
we can construct `new' Weil cohomology theories 
using the classical ones. This was undertaken in 
\cite{Nouvelle-Weil}, and we revisit the construction below.

\begin{cons}
\label{cons:new-weil-cohomology-theories-rigid-analytic}
We assume that $K$ has characteristic zero, but we allow
its residue field $k$ to have arbitrary characteristic. 
\begin{enumerate}

\item Let $K^{\alg}/K$ be an algebraic closure of $K$, 
and let $\Gamma_{\ell,\,K}$ be the associated $\ell$-adic 
cohomology theory on algebraic $K$-varieties. 
The Weil cohomology theory $\widehat{\Gamma}_{\new,\,\ell}$ 
on rigid analytic $K$-varieties 
corresponds to the Weil spectrum $\Rig^*\mathbf{\Gamma}_{\ell,\,K}$.
Its ring of coefficients is $\Rder_{\ell}^*(\Rig_*\Lambda)$.

\item Let $\sigma:K\hookrightarrow \C$ be a complex embedding, and let
$\Gamma_{\Betti,\,K}$ be the associated Betti cohomology theory on algebraic 
$K$-varieties. The Weil cohomology theory $\widehat{\Gamma}_{\new,\,\Betti}$
on rigid analytic $K$-varieties corresponds to the 
Weil spectrum $\Rig^*\mathbf{\Gamma}_{\Betti,\,K}$.
Its ring of coefficients is $\Betti^*(\Rig_*\Lambda)$.

\item Fix a morphism $\Lambda \to K$.
Let $\Gamma_{\dR,\,K}$ be the de Rham cohomology theory
on algebraic $K$-varieties. 
The Weil cohomology theory $\widehat{\Gamma}_{\new,\,\dR}$
on rigid analytic $K$-varieties corresponds to the Weil spectrum 
$\Rig^*\mathbf{\Gamma}_{\dR,\,K}$.
Its ring of coefficients is  
$\dR^*(\Rig_*\Lambda)$.

\end{enumerate}
The Weil spectra $\mathbf{\Gamma}_{\new,\,\ell}$, $\mathbf{\Gamma}_{\new,\,\Betti}$ and $\mathbf{\Gamma}_{\new,\,\dR}$ on algebraic $k$-varieties
are obtained form 
$\widehat{\mathbf{\Gamma}}_{\new,\,\ell}$, 
$\widehat{\mathbf{\Gamma}}_{\new,\,\Betti}$ and 
$\widehat{\mathbf{\Gamma}}_{\new,\,\dR}$ respectively by applying $\xi_*$.
To stress the dependence on $K$, we would write 
`$\widehat{\Gamma}_{\new,\,\ell,\,K}$', `$\Gamma_{\new,\,\ell,\,K}$', etc.
\end{cons}

\begin{prop}
\label{prop:diagram-of-weil-cohomology-theory-rig-an}
There is a diagram in $\RigWCT(K;\Lambda)$
as follows:
\begin{equation}
\label{eq-prop:diagram-of-weil-cohomology-theory-rig-an-1}
\xymatrix{\widehat{\Gamma}_{\new,\,\dR} \ar[r] \ar[d] &
\widehat{\Gamma}_{\new,\,\dR}\otimes_K\C  &
\ar[l] \widehat{\Gamma}_{\new,\,\Betti} \ar[r] & 
\widehat{\Gamma}_{\new,\,\ell} \ar[d]\\
\Gamma^{\dagger}_{\dR} & & & \Gamma_{\ell}.\!}
\end{equation}
Similarly, there is a diagram in $\WCT(k;\Lambda)$
as follows:
\begin{equation}
\label{eq-prop:diagram-of-weil-cohomology-theory-rig-an-2}
\xymatrix{\Gamma_{\new,\,\dR} \ar[r] \ar[d] &
\Gamma_{\new,\,\dR}\otimes_K\C  &
\ar[l] \Gamma_{\new,\,\Betti} \ar[r] & 
\Gamma_{\new,\,\ell} \ar[d]\\
\Gamma_{\rig} & & & \Gamma_{\ell}.\!}
\end{equation}
\end{prop}

\begin{proof}
The horizontal line in the diagram
\eqref{eq-prop:diagram-of-weil-cohomology-theory-rig-an-1}
is obtained from the following diagram 
$$\Gamma_{\dR} \to \Gamma_{\dR}\otimes_K\C \leftarrow 
\Gamma_{\Betti} \to \Gamma_{\ell}$$
by applying the functor 
$\Rig^*:\WCT(K;\Lambda)
\to \RigWCT(K;\Lambda)$.
(See Propositions
\ref{prop:comparison-thm-betti-coh-l-adic-coh}
and 
\ref{prop:comparison-betti-vs-de-rham}.)
The vertical arrows in the diagram
\eqref{eq-prop:diagram-of-weil-cohomology-theory-rig-an-1}
are deduced by adjunction from the natural morphisms
$\Gamma_{\dR} \to \Rig_*(\Gamma^{\dagger}_{\dR})$
and $\Gamma_{\ell} \to \Rig_*(\Gamma_{\ell})$
provided by Propositions
\ref{prop:various-relations-between-ell-adic-coh}
and 
\ref{prop:berthelot-rigid-cohomology-vs-de-rham-}.
Finally, the diagram 
\eqref{eq-prop:diagram-of-weil-cohomology-theory-rig-an-2}
is obtained from the diagram
\eqref{eq-prop:diagram-of-weil-cohomology-theory-rig-an-1}
by applying $\xi_*$ and using Proposition
\ref{prop:various-relations-between-ell-adic-coh}.
\end{proof}

\section{The connectivity theorem}

\label{sect:proof-of-main-thm}

In this final section, we prove the main new result of this paper, 
namely the connectivity of the motivic Hopf algebras associated to
the classical Weil cohomology theories. More precisely, we will 
prove the following.

\begin{thm}
\label{thm:motivic-connectivity-for-ell-adic}
Let $\Lambda$ be a connective commutative ring spectrum,
$k$ a field whose exponent characteristic is invertible
in $\pi_0(\Lambda)$, and $K$ a valued field of characteristic zero
whose valuation has height $1$ and whose residue field is $k$.
Below, $\ell$ is a prime invertible in $k$, a complex
embedding $\sigma:K\hookrightarrow \C$
is chosen, and a morphism $\Lambda \to K$ is fixed allowing us 
to view $K$-modules as $\Lambda$-modules. 
\begin{enumerate}

\item The motivic Hopf algebroids of the following Weil 
cohomology theories 
$$\Gamma_{\ell},\quad \Gamma_{\rig}, \quad
\Gamma_{\new,\,\ell},\quad 
\Gamma_{\new,\,\Betti} \quad \text{and} \quad
\Gamma_{\new,\,\dR}\quad \in \WCT(k;\Lambda)$$
are connective.

\item The motivic Hopf algebroids of the following Weil 
cohomology theories 
$$\Gamma_{\ell},\quad 
\Gamma^{\dagger}_{\dR}, \quad
\widehat{\Gamma}_{\new,\,\ell},\quad 
\widehat{\Gamma}_{\new,\,\Betti} \quad \text{and} \quad
\widehat{\Gamma}_{\new,\,\dR}\quad \in \RigWCT(K;\Lambda)$$
are connective.

\end{enumerate}
\end{thm}

We start by establishing the following reduction.

\begin{lemma}
\label{lemma:reduction-to-the-case-k-alg-closed}
It is enough to prove Theorem 
\ref{thm:motivic-connectivity-for-ell-adic}
for $\Gamma_{\new,\,\Betti}$ and
$\widehat{\Gamma}_{\new,\,\Betti}$ under the 
extra assumption that $K$ is algebraically closed. 
\end{lemma}

\begin{proof}
By Corollary
\ref{cor:simple-useful-criterion-connectedness}
and Proposition
\ref{prop:diagram-of-weil-cohomology-theory-rig-an},
it is enough to treat the Weil cohomology theories
$\Gamma_{\new,\,\Betti}$ and $\widehat{\Gamma}_{\new,\,\Betti}$.
It remains to explain why we can assume that $K$ is algebraically closed.
We do this in two steps: first we reduce the general case
to the henselian one, and then we reduce the henselian 
case to the algebraically closed one. For later use, 
it will be convenient to assume that 
$\Lambda=\SS_{(\ell)}$ is the localization of the sphere spectrum 
at a prime number $\ell$, which we can do without lost of generality. 

\subsubsection*{Step 1: reduction to the case where $K$ is henselian}
Let $K'$ be the henselization of $K$ with respect to its valuation. 
Thus, the ring of integers $K'^{\circ}$ in $K'$ 
is the henselization of the ring of integers $K^{\circ}$ in $K$,
and $K$ and $K'$ have the same completion $\widehat{K}$.
We denote by $e:\Spec(K') \to \Spec(K)$ the obvious morphism 
and, for the purpose of this proof, we denote by 
$\Rig'^*:\SH_{\et}(K';\Lambda)
\to \RigSH_{\et}(\widehat{K};\Lambda)$
the rigid analytification functor. Notice that 
$\Rig_*\Lambda\simeq e_*\Rig'_*\Lambda$
is naturally an $e_*\Lambda$-algebra.
It follows that 
$\Betti^*(\Rig_*\Lambda)$ is naturally an algebra over 
$\Betti^*(e_*\Lambda)=\Lambda^{\Hom_K(K',\C)}$.
(As usual, we view $\Hom_K(K',\C)$ as a profinite set 
and $\Lambda^{\Hom_K(K',\C)}$ is the $\Lambda$-algebra of locally
constant functions on $\Hom_K(K',\C)$ with values in $\Lambda$.)
Since the ring of coefficients of $\Gamma_{\new,\,\Betti}$
and $\widehat{\Gamma}_{\new,\,\Betti}$ is 
$\Betti^*(\Rig_*\Lambda)$, the Weil spectra
$\mathbf{\Gamma}_{\new,\,\Betti}$ and 
$\widehat{\mathbf{\Gamma}}_{\new,\,\Betti}$
have the structure of a $\Lambda^{\Hom_K(K',\C)}$-algebra.
Thus, the motivic Hopf algebroids 
$\Hmot(\Gamma_{\new,\,\Betti})$
and $\Hmot(\widehat{\Gamma}_{\new,\,\Betti})$
are algebras over the cosimplicial ring 
$\Lambda^{\Cech_{\bullet}(\Hom_K(K',\C))}$.
To prove that these motivic Hopf algebroids
are connective, it is enough to show the following.
\begin{enumerate}

\item Given $\sigma'\in\Hom_K(K',\C)$, the commutative algebra
\begin{equation}
\label{eq-lemma:reduction-to-the-case-k-alg-closed-1}
\Hmot(\Gamma_{\new,\,\Betti})(\Delta^0)
\otimes_{\Lambda^{\Hom_K(K',\C)}}
\Lambda^{\{\sigma'\}}=
\Hmot(\widehat{\Gamma}_{\new,\,\Betti})(\Delta^0)
\otimes_{\Lambda^{\Hom_K(K',\C)}}
\Lambda^{\{\sigma'\}}
\end{equation}
is connective.

\item Given $\sigma'_1,\sigma'_2\in\Hom_K(K',\C)$, 
the commutative algebras
\begin{equation}
\label{eq-lemma:reduction-to-the-case-k-alg-closed-2}
\Hmot(\Gamma_{\new,\,\Betti})(\Delta^1)
\otimes_{\Lambda^{\Hom_K(K',\C)^{\times 2}}}
\Lambda^{\{(\sigma'_1,\sigma'_2)\}}
\; \text{and} \;
\Hmot(\widehat{\Gamma}_{\new,\,\Betti})(\Delta^1)
\otimes_{\Lambda^{\Hom_K(K',\C)^{\times 2}}}
\Lambda^{\{(\sigma'_1,\sigma'_2)\}}
\end{equation}
are connective.

\end{enumerate}
Now, the algebra 
\eqref{eq-lemma:reduction-to-the-case-k-alg-closed-1} 
is obtained by applying 
$\Gamma(\widehat{K};\Omega^{\infty}_T(-))$ to 
$\widehat{\mathbf{\Gamma}}_{\new,\,\Betti}\otimes_{\Lambda^{\Hom_K(K',\C)}}\Lambda^{\{\sigma'\}}$, and the algebras
\eqref{eq-lemma:reduction-to-the-case-k-alg-closed-2} 
are obtained by applying $\Gamma(k;\Omega^{\infty}_T(-))$
and $\Gamma(\widehat{K};\Omega^{\infty}_T(-))$ to
$$(\widehat{\mathbf{\Gamma}}_{\new,\,\Betti}\otimes_{\Lambda^{\Hom_K(K',\C)}}\Lambda^{\{\sigma_1'\}})\otimes 
(\widehat{\mathbf{\Gamma}}_{\new,\,\Betti}\otimes_{\Lambda^{\Hom_K(K',\C)}}\Lambda^{\{\sigma_2'\}}) \quad \text{and}\hspace{2cm}$$
$$\hspace{2cm}(\mathbf{\Gamma}_{\new,\,\Betti}\otimes_{\Lambda^{\Hom_K(K',\C)}}\Lambda^{\{\sigma_1'\}})\otimes 
(\mathbf{\Gamma}_{\new,\,\Betti}\otimes_{\Lambda^{\Hom_K(K',\C)}}\Lambda^{\{\sigma_2'\}}).$$
Thus, we need a good description of the spectra $\widehat{\mathbf{\Gamma}}_{\new,\,\Betti}\otimes_{\Lambda^{\Hom_K(K',\C)}}\Lambda^{\{\sigma'\}}$,
for $\sigma'\in \Hom_K(K',\C)$. We claim that there is an equivalence
\begin{equation}
\label{eq-lemma:reduction-to-the-case-k-alg-closed-3}
\widehat{\mathbf{\Gamma}}_{\new,\,\Betti}\otimes_{\Lambda^{\Hom_K(K',\C)}}\Lambda^{\{\sigma'\}}
\simeq \widehat{\mathbf{\Gamma}}_{\new,\,\Betti'}
\end{equation}
where $\widehat{\Gamma}_{\new,\,\Betti'}\in \RigWCT(\widehat{K};\Lambda)$ 
is the new Weil cohomology theory obtained from the Betti cohomology 
theory $\Gamma_{\Betti'}\in \WCT(K';\Lambda)$ associated to the
complex embedding $\sigma':K'\to \C$.
Indeed, recall from Theorem
\ref{thm:the-main-equivalence-for-new-Weil-coh}
that we have equivalences of $\infty$-categories
$$\xymatrix{\RigSH_{\et}(\widehat{K};\Lambda) 
\ar[r]^-{\widetilde{\Rig}{}'_*}_-{\sim} 
\ar[dr]^-{\sim}_-{\widetilde{\Rig}_*}
& \SH_{\et}(K';\Rig'_*\Lambda) \ar[d]_-{\sim}^-{\widetilde{e}_*}\\
& \SH_{\et}(K;\Rig_*\Lambda)}$$
where the vertical arrow is induced by the equivalence
$$\widetilde{e}_*:\SH_{\et}(K';\Lambda) \xrightarrow{\sim}
\SH_{\et}(K;e_*\Lambda)$$
from Proposition
\ref{prop:SH-et-of-k-prime-is-module-over-e-star}.
Modulo $\widetilde{\Rig}_*$, the Weil spectrum 
$\widehat{\mathbf{\Gamma}}_{\new,\,\Betti}$ is given by 
$\Rig_*\Lambda\otimes \Betti_*\Lambda$.
Thus, modulo $\widetilde{\Rig}{}'_*$, the same Weil spectrum
$\widehat{\mathbf{\Gamma}}_{\new,\,\Betti}$ is given by 
\begin{equation}
\label{eq-lemma:reduction-to-the-case-k-alg-closed-5}
e^*(e_*\Rig'_*\Lambda\otimes \Betti_*\Lambda)
\otimes_{e^*e_*\Rig'_*\Lambda}\Rig'_*\Lambda
\simeq \Rig'_*\Lambda\otimes e^*\Betti_*\Lambda.
\end{equation}
Note that 
$$\begin{array}{rcl}
\Gamma(K';\Omega^{\infty}_T(e^*\Betti_*\Lambda))
& \simeq & \Gamma(K;\Omega_T^{\infty}(e_*e^*\Betti_*\Lambda))\\
& \simeq & \Gamma(K;\Omega_T^{\infty}(\Betti_*e_*e^*\Lambda))\\
& \simeq & \Gamma(K;\Omega_T^{\infty}(\Betti_*\Lambda^{\Hom_K(K',\C)}))\\
& \simeq & \Lambda^{\Hom_K(K',\C)}.
\end{array}$$
In particular, $e^*\Betti_*\Lambda$ has a natural structure of a 
$\Lambda^{\Hom_K(K',\C)}$-algebra, and it is easy to see that this structure
induces the one we have on 
$\widehat{\mathbf{\Gamma}}_{\new,\,\Betti}$ modulo the 
equivalence $\widetilde{\Rig}{}'_*$. 
For the purpose of this proof, we denote by
$\Betti'^*:\SH_{\et}(K';\Lambda) \to \Mod_{\Lambda}$
the Betti realization associated to the embedding $\sigma'$, so that 
$\mathbf{\Gamma}_{\Betti'}=\Betti'_*\Lambda$.
We have equivalences in $\SH_{\et}(K;e_*\Lambda)$:
$$\begin{array}{rcl}
(e_*e^*\Betti_*\Lambda)
\otimes_{\Lambda^{\Hom_K(K',\C)}}\Lambda^{\{\sigma'\}}
& \simeq & (\Betti_*e_*e^*\Lambda) \otimes_{\Lambda^{\Hom_K(K',\C)}}\Lambda^{\{\sigma'\}}\\
& \simeq & (\Betti_*\Lambda^{\Hom_K(K',\C)}) 
\otimes_{\Lambda^{\Hom_K(K',\C)}}\Lambda^{\{\sigma'\}}\\
& \simeq & \Betti_*\Lambda,
\end{array}$$
where the action of $e_*\Lambda$ on $\Betti_*\Lambda$ 
is the one deduced from the equivalence 
$\Betti_*\Lambda\simeq e_*\Betti'_*\Lambda$.
Using Proposition
\ref{prop:SH-et-of-k-prime-is-module-over-e-star},
this yields an equivalence in $\SH_{\et}(K';\Lambda)$:
$$(e^*\Betti_*\Lambda)
\otimes_{\Lambda^{\Hom_K(K',\C)}}\Lambda^{\{\sigma'\}}
\simeq \Betti'_*\Lambda.$$
Combining this with the fact that 
$\widehat{\mathbf{\Gamma}}_{\new,\,\Betti}$
is given by 
\eqref{eq-lemma:reduction-to-the-case-k-alg-closed-5}
modulo the equivalence $\widetilde{\Rig}{}'_*$, 
we deduce the claimed equivalence 
\eqref{eq-lemma:reduction-to-the-case-k-alg-closed-3}.

It is now easy to conclude. Indeed, recall that 
we want to check the properties (1) and (2) above.
Using the equivalence 
\eqref{eq-lemma:reduction-to-the-case-k-alg-closed-3}, 
the commutative algebra 
\eqref{eq-lemma:reduction-to-the-case-k-alg-closed-1}
can be identified with the ring of coefficients of 
$\widehat{\Gamma}_{\new,\,\Betti'}$ whose connectivity 
would be granted if Theorem
\ref{thm:motivic-connectivity-for-ell-adic}
was known for $K'$.
Similarly, denoting $\Gamma_{\new,\,\Betti'_i}\in \WCT(k;\Lambda)$
and $\widehat{\Gamma}_{\new,\,\Betti'_i}\in \RigWCT(K';\Lambda)$, 
for $i\in \{1,2\}$, the new 
Weil cohomology theories obtained from the Betti cohomology theories 
$\Gamma_{\Betti'_i}$
associated to the complex embeddings $\sigma'_i$, the commutative 
algebras in 
\eqref{eq-lemma:reduction-to-the-case-k-alg-closed-2}
can be identified with 
$$\Gamma(k;\Omega^{\infty}_T(\mathbf{\Gamma}_{\new,\,\Betti'_1}\otimes
\mathbf{\Gamma}_{\new,\,\Betti'_2})) \qquad \text{and}
\qquad \Gamma(\widehat{K};
\Omega^{\infty}_T(\widehat{\mathbf{\Gamma}}_{\new,\,\Betti'_1}\otimes
\widehat{\mathbf{\Gamma}}_{\new,\,\Betti'_2})).$$
If $\sigma'_1=\sigma'_2=\sigma'$, these algebras coincide with 
$\Hmot(\Gamma_{\new,\,\Betti'})(\Delta^1)$ and 
$\Hmot(\widehat{\Gamma}_{\new,\,\Betti'})(\Delta^1)$,
and we would be done. In general, it would suffices to know that 
$\widehat{\Gamma}_{\new,\,\Betti'_1}$ and 
$\widehat{\Gamma}_{\new,\,\Betti'_2}$ 
become equivalent after a common faithfully flat extension 
of their ring of coefficients. This would follow if 
the Betti cohomology theories 
$\Gamma_{\Betti'_1}$ and $\Gamma_{\Betti'_2}$
become equivalent after a faithfully flat extension 
of $\Lambda$. To prove this, recall from the beginning of the proof
that $\Lambda=\SS_{(\ell)}$. Proposition 
\ref{prop:comparison-thm-betti-coh-l-adic-coh}
implies that $\Gamma_{\Betti'_1}\otimes \SS_{\ell}$ and $\Gamma_{\Betti'_2}
\otimes \SS_{\ell}$
are equivalent to the $\ell$-adic cohomology theories 
$\Gamma_{\ell,1}$ and $\Gamma_{\ell,2}$ in $\WCT(K';\Lambda)$ 
associated the algebraic closures $\overline{K}_1/K'$ and 
$\overline{K}_2/K'$ deduced from the complex embeddings 
$\sigma'_1$ and $\sigma'_2$. Lemma
\ref{lemma:ell-adic-cohomology-invariant-galois}
implies that $\Gamma_{\ell,1}$ and $\Gamma_{\ell,2}$
is equivalent. This finishes the proof since 
$\SS_{\ell}$ is faithfully flat over $\SS_{(\ell)}$
by \cite[Corollary 7.3.6.9]{Lurie-SAG}.

\subsubsection*{Step 2: reduction to the case where 
$K$ is algebraically closed}
We now assume that $K$ is henselian. Let $K^{\alg}\subset \C$
be the algebraic closure of $K$ in $\C$. 
The valuation of $K$ extends uniquely to a valuation of 
$K^{\alg}$ and every element of the Galois group 
$G$ of $K^{\alg}/K$ induces an automorphism of the completion 
$\widehat{K}{}^{\alg}$ of $K^{\alg}$.
Let $e:\Spec(K^{\alg}) \to \Spec(K)$ and
$\widehat{e}:\Spa(\widehat{K}{}^{\alg}) \to \Spa(\widehat{K})$
be the obvious morphisms. For the purpose of this proof,
we denote by $\Rig'^*:\SH_{\et}(K^{\alg};\Lambda)
\to \RigSH_{\et}(K^{\alg};\Lambda)$
the rigid analytification functor. By Lemma 
\ref{lemma:adjointability-rig-upper-star-henselian}
below, we have an equivalence 
$$e^*\Rig_*\Lambda \simeq \Rig'_*\Lambda$$
of commutative algebras in $\SH_{\et}(K^{\alg};\Lambda)$.
Write $\Betti'^*:\SH_{\et}(K^{\alg};\Lambda)\to \Mod_{\Lambda}$
for the plain Betti realization functor associated to the 
complex embedding $K^{\alg}\subset \C$. The functor 
$$e_*:\SH_{\et}(K^{\alg};\Rig'_*\Lambda)
\to \SH_{\et}(K;\Rig_*\Lambda)$$
takes the algebra 
$\Rig'_*\Lambda\otimes \Betti'_*\Lambda$
to 
$$\begin{array}{rcl}
e_*(e^*\Rig_*\Lambda \otimes \Betti'_*\Lambda)
& \simeq & \Rig_*\Lambda\otimes e_*\Betti'_*\Lambda\\
& \simeq & \Rig_*\Lambda \otimes \Betti_*\Lambda,
\end{array}$$
where the first equivalence is deduced from the projection formula 
for the morphism $e$ which is a profinite morphism. 
This implies that the Weil cohomology theories
$\widehat{\Gamma}_{\new,\,\Betti}\in \RigWCT(K;\Lambda)$
and 
$\widehat{\Gamma}_{\new,\,\Betti'}\in \RigWCT(K^{\alg};\Lambda)$
are related by the formula:
\begin{equation}
\label{eq-lemma:reduction-to-the-case-k-alg-closed-53}
\widehat{\Gamma}_{\new,\,\Betti}\simeq 
\widehat{e}{}_*(\widehat{\Gamma}_{\new,\,\Betti'}).
\end{equation}
Letting $k^{\alg}$ be the residue field of $K^{\alg}$ and
$\overline{e}:\Spec(k^{\alg}) \to \Spec(k)$ the obvious morphism, 
we also have the formula
\begin{equation}
\label{eq-lemma:reduction-to-the-case-k-alg-closed-61}
\Gamma_{\new,\,\Betti}\simeq \overline{e}_*
(\Gamma_{\new,\,\Betti'})
\end{equation}
relating the Weil cohomology theories 
$\Gamma_{\new,\,\Betti}\in \WCT(k;\Lambda)$
and $\Gamma_{\new,\,\Betti'}\in \WCT(k^{\alg};\Lambda)$.
This follows immediately from the commutative square
$$\xymatrix{\SH_{\et}(k;\Lambda) \ar[r]^-{\xi^*} \ar[d]^-{\overline{e}{}^*}
& \RigSH_{\et}(K;\Lambda) \ar[d]^-{\hat{e}{}^*}\\
\SH_{\et}(k^{\alg};\Lambda) 
\ar[r]^-{\xi^*} & \RigSH_{\et}(K^{\alg};\Lambda)}$$
by considering the associated plain realization functors.

It is now easy to conclude. Indeed, we want to show that 
$\Hmot(\Gamma_{\new,\,\Betti})$ and 
$\Hmot(\widehat{\Gamma}_{\new,\,\Betti})$
are connective knowing that 
$\Hmot(\Gamma_{\new,\,\Betti'})$ and 
$\Hmot(\widehat{\Gamma}_{\new,\,\Betti'})$
are connective. Using the equivalences
\eqref{eq-lemma:reduction-to-the-case-k-alg-closed-53}
and
\eqref{eq-lemma:reduction-to-the-case-k-alg-closed-61},
and Proposition
\ref{prop:reduction-algeb-closed-connectivity},
we are reduced to showing that, given an automorphism 
$g$ of $K^{\alg}/K$, the Weil cohomology theories
$\widehat{\Gamma}_{\new,\,\Betti'}$ and 
$g_*\widehat{\Gamma}_{\new,\,\Betti'}$
become equivalent after a common faithfully flat extension of their
rings of coefficients. It is easy to see that 
$g_*\widehat{\Gamma}_{\new,\,\Betti'}\in \RigWCT(K^{\alg};\Lambda)$ 
is the new Weil cohomology theory obtained from the Weil 
cohomology theory $g_*\Gamma_{\Betti'} \in \WCT(K^{\alg};\Lambda)$.
Thus, it is enough to show that $\Gamma_{\Betti'}$ and 
$g_*\Gamma_{\Betti'}$ become equivalent after a failfully flat 
extension of $\Lambda$. 
To prove this, recall from the beginning of the proof
that $\Lambda=\SS_{(\ell)}$. Proposition 
\ref{prop:comparison-thm-betti-coh-l-adic-coh}
implies that $\Gamma_{\Betti'}\otimes \SS_{\ell}$ is 
equivalent to the $\ell$-adic cohomology theory
$\Gamma_{\ell} \in \WCT(K^{\alg};\Lambda)$. Lemma
\ref{lemma:ell-adic-cohomology-invariant-galois}
implies that $\Gamma_{\ell}$ and $g_*\Gamma_{\ell}$
are equivalent. This finishes the proof since 
$\SS_{\ell}$ is faithfully flat over $\SS_{(\ell)}$
by \cite[Corollary 7.3.6.9]{Lurie-SAG}.
\end{proof}

\begin{lemma}
\label{lemma:adjointability-rig-upper-star-henselian}
Assume that the valued field $K$ is henselian. Let 
$K'/K$ be an algebraic extension of $K$ and endow $K'$ with the unique
extension of the valuation of $K$. Let 
$e:\Spec(K') \to \Spec(K)$ and $\widehat{e}:\Spa(\widehat{K}{}')
\to \Spa(\widehat{K})$ be the obvious morphisms.  
The commutative square of $\infty$-categories
$$\xymatrix{\SH_{\et}(K;\Lambda) \ar[r]^-{e^*} \ar[d]^-{\Rig^*}
& \SH_{\et}(K';\Lambda) \ar[d]^-{\Rig^*}\\
\RigSH_{\et}(K;\Lambda) \ar[r]^-{\hat{e}{}^*} 
& \RigSH_{\et}(K';\Lambda)}$$
is right adjointable horizontally and vertically, i.e., the induced natural transformations 
$$\Rig^* e_* \to \widehat{e}_*\Rig^*
\qquad \text{and} \qquad e^*\Rig_* \to \Rig_*\widehat{e}{}^*$$
are equivalences.
\end{lemma}

\begin{proof}
Consider objects 
$M'\in \SH_{\et}(K';\Lambda)$ and $N\in \RigSH_{\et}(K;\Lambda)$.
We want to prove that 
$$\Rig^* e_*(M') \to \widehat{e}_*\Rig^*(M')
\qquad \text{and} \qquad e^*\Rig_*(N) \to \Rig_*\widehat{e}{}^*(N)$$
are equivalences. It is enough to do this with $M'$ replaced by 
$M'\otimes \Q$ and $M'\otimes \Q/\Z$, 
and similarly for $N$. Said differently, we may assume that 
$M'$ and $N$ are uniquely divisible or 
torsion. 
In the torsion case, the commutative square
$$\xymatrix{\SH_{\et}(K;\Lambda)_{\tor} \ar[r]^-{e^*} \ar[d]^-{\Rig^*}
& \SH_{\et}(K';\Lambda)_{\tor} \ar[d]^-{\Rig^*}\\
\RigSH_{\et}(K;\Lambda)_{\tor} \ar[r]^-{\hat{e}{}^*} 
& \RigSH_{\et}(K';\Lambda)_{\tor}}$$
can be identified, using 
\cite[Theorems 2.10.3 \& 2.10.4]{AGAV},
with the following commutative square
\begin{equation}
\label{eq-lemma:adjointability-rig-upper-star-henselian-135}
\begin{split}
\xymatrix{\Shv_{\et}(\Et_K;\Lambda)_{\tor} \ar[r]^-{e^*} 
\ar[d]^-{\sim} & \Shv_{\et}(\Et_{K'};\Lambda)_{\tor} \ar[d]^{\sim}\\
\Shv_{\et}(\Et_{\widehat{K}};\Lambda)_{\tor} \ar[r]^-{\hat{e}{}^*} & 
\Shv_{\et}(\Et_{\widehat{K}{}'};\Lambda)_{\tor}}
\end{split}
\end{equation}
where the vertical arrows are equivalences since $K$ and $K'$ are
henselian. Lemmas
\ref{lemma:Rig-lower-star-commutes-with-colimits}
and \ref{lemma:commutation-with-colimits-for-e-lower-star}
imply that $e_*$, $\widehat{e}_*$ and $\Rig_*$ preserve torsion objects.
Thus, it is enough to prove that the square 
\eqref{eq-lemma:adjointability-rig-upper-star-henselian-135}
is adjointable horizontally and vertically, which is clear.

We now assume that $M'$ and $N$ are uniquely divisible or, 
equivalently, that $\Lambda$ is a $\Q$-algebra. In this case, 
we can invoke \cite[Theorem 2.5.1 \& Proposition 2.5.11]{AGAV}
to reduce to the case where the extension $K'/K$ is finite.
Using \cite[Theorems 2.9.6 \& 2.9.7]{AGAV}, we can even assume 
that $K'/K$ is \'etale. 
Then, the result follows from 
\cite[Propositions 2.2.13 \& 2.2.14]{AGAV}.
\end{proof}

We are now ready to finish the proof of Theorem 
\ref{thm:motivic-connectivity-for-ell-adic}.

\begin{proof}[Proof of Theorem 
\ref{thm:motivic-connectivity-for-ell-adic}]
By Lemma 
\ref{lemma:reduction-to-the-case-k-alg-closed},
we may assume that $K$ is algebraically closed, and we only need
to consider $\Gamma_{\new,\,\Betti}$ and 
$\widehat{\Gamma}_{\new,\,\Betti}$.
We now claim that it suffices to treat the case of 
$\widehat{\Gamma}_{\new,\,\Betti}$. Indeed, assume that 
$\Hmot(\widehat{\Gamma}_{\new,\,\Betti})$ is connective.
Then, this already implies that 
$\Gamma_{\new,\,\Betti}(\pt)$ is a connective 
algebra. By Theorem 
\ref{thm:completion-of-motivic-Hopf-algebras},
the obvious morphism
$$\Cech(\Gamma_{\new,\,\Betti}(\pt))
\to \Hmot(\Gamma_{\new,\,\Betti})$$
induces an equivalence after $\ell$-completion, for every prime
$\ell$. 
It follows that the square
$$\xymatrix{\Cech(\Gamma_{\new,\,\Betti}(\pt)) 
\ar[r] \ar[d] & \Hmot(\Gamma_{\new,\,\Betti}) \ar[d]\\
\Cech(\Gamma_{\new,\,\Betti}(\pt))\otimes\Q \ar[r] & 
\Hmot(\Gamma_{\new,\,\Betti})\otimes\Q}$$
is cocartesian in $(\Mod_{\Lambda})^{\Deltasimp}$,
and hence it is enough to show that 
$\Hmot(\Gamma_{\new,\,\Betti}) \otimes \Q$
is connective. To do so, we may replace $\Lambda$ by $\Lambda_{\Q}$
and assume that $\Lambda$ is a $\Q$-algebra. Under this assumption, the 
connectivity of $\Hmot(\Gamma_{\new,\,\Betti})$ follows from the 
connectivity of $\Hmot(\widehat{\Gamma}_{\new,\,\Betti})$
by Theorem
\ref{thm:connectivity-rigid-anal-vs-algebraic-}. This proves our claim.

It remains to see that 
$\Hmot(\widehat{\Gamma}_{\new,\,\Betti})$
is connective. By Lemma 
\ref{lemma:connectivity-hopf-algebroid}, 
it is enough to show that the algebras
$\Hmot(\widehat{\Gamma}_{\new,\,\Betti})(\Delta^0)$ and 
$\Hmot(\widehat{\Gamma}_{\new,\,\Betti})(\Delta^1)$ 
are connective. These are obtained by applying
$\Gamma(\pt;\Omega^{\infty}_T(-))$ to the
two algebras $\Rig_*\Lambda\otimes \mathbf{\Gamma}_{\Betti}$
and $\Rig_*\Lambda\otimes \mathbf{\Gamma}_{\Betti}\otimes 
\mathbf{\Gamma}_{\Betti}$ in $\SH_{\et}(K;\Rig_*\Lambda)$.
We have an equivalence
$\mathbf{\Gamma}_{\Betti}\otimes \mathbf{\Gamma}_{\Betti}\simeq
\mathbf{\Gamma}_{\Betti}\otimes (\Hmot(\Gamma_{\Betti})(\Delta^1))$
which implies that 
$$\Hmot(\widehat{\Gamma}_{\new,\,\Betti})(\Delta^1)
\simeq \Rder\Gamma(\pt;\Omega^{\infty}_T(\Rig_*\Lambda\otimes \mathbf{\Gamma}_{\Betti}))\otimes (\Hmot(\Gamma_{\Betti})(\Delta^1)).$$
Since the motivic Hopf algebra $\Hmot(\Gamma_{\Betti})$ 
is known to be connective by \cite[Corollaire 2.105]{gal-mot-1},
we are left to check that 
$$\begin{array}{rcl}
\Hmot(\widehat{\Gamma}_{\new,\,\Betti})(\Delta^0) & = & 
\Rder\Gamma(\pt;\Omega^{\infty}_T(\Rig_*\Lambda\otimes \mathbf{\Gamma}_{\Betti}))\\
& \simeq & \Betti^*(\Rig_*\Lambda).
\end{array}$$
is connective. We claim that the morphism
$\Lambda \to \Rig_*\Lambda$ induces an equivalence in 
$\SH_{\et}(K;\Lambda)_{\ellcpl}$
after $\ell$-completion for every prime $\ell$. 
Indeed, since $\Rig_*$ is colimit-preserving 
(see Lemma 
\ref{lemma:Rig-lower-star-commutes-with-colimits}), 
we have 
$(\Rig_*\Lambda)_{\ell}^{\wedge}\simeq (\Rig_*)_{\ellcpl}\Lambda_{\ell}$
where $(\Rig_*)_{\ellcpl}$ is the right adjoint to the functor 
$$(\Rig^*)_{\ellcpl}:\SH_{\et}(K;\Lambda)_{\ellcpl}
\to \RigSH_{\et}(K;\Lambda)_{\ellcpl}.$$
It follows from \cite[Theorem 2.10.3 \& 2.10.4]{AGAV} that the functor 
$(\Rig^*)_{\ellcpl}$ is an equivalence, proving our claim.
That said, we have a cocartesian square in $\SH_{\et}(K;\Lambda)$
$$\xymatrix{\Lambda \ar[r] \ar[d] & \Rig_*\Lambda\ar[d]\\
\Lambda_{\Q} \ar[r] & \Rig_*\Lambda_{\Q}.\!}$$
Thus, we are reduced to showing that 
$\Betti^*(\Rig_*\Lambda_{\Q})$ is connective.
Said differently, we may assume that $\Lambda$ is a 
$\Q$-algebra and even that $\Lambda=\Q$. Using Proposition
\ref{prop:comparison-betti-vs-de-rham}, it is enough to show that the 
algebra 
$$\Rder_{\dR}^*(\Rig_*\Q)=\Rder\Gamma(\pt;\Omega^{\infty}_T(
\Rig^*\mathbf{\Gamma}_{\dR}))$$
is connective. This algebra is known to be connective 
by \cite[Corollaire 3.8]{Nouvelle-Weil}.
\end{proof}

\appendix

\section{Some computation with Milnor--Witt $K$-theory}

Given a scheme $S$, we denote by
$\SH(S;\Lambda)$ the stable Morel--Voevodsky $\infty$-category 
with coefficients in $\Lambda$. 
(This is the Nisnevich local counterpart of the 
$\infty$-category $\SH_{\et}(S;\Lambda)$.)
We do not assume in this appendix 
that the residual characteristics of $S$ are invertible 
in $\pi_0(\Lambda)$.
Our goal is to show the following result
which was used in the proof of Theorem 
\ref{thm:the-functor-Psi-using-basis-of-value-group}.

\begin{prop}
\label{prop:action-of-elevation-power-m-on-Gm}
Let $S$ be a scheme such that $-1\in \mathcal{O}(S)$ is a square.
Then, for every integer $m\in \Z$, elevation to the $m$-th power
on $\Gm_{,\,S}$ is given by the matrix
$$\begin{pmatrix} 1 & 0\\
0 & m
\end{pmatrix}:\Lambda \oplus \Lambda(1)[1] \to 
\Lambda \oplus \Lambda(1)[1]$$
on the associated homological motive in $\SH(S;\Lambda)$.
\end{prop}

The decomposition of the homological motive of $\Gm_{,\,S}$ 
alluded to in the statement of Proposition
\ref{prop:action-of-elevation-power-m-on-Gm}
is induced by the unit section of the group scheme $\Gm_{,\,S}$.
Modulo this decomposition, the multiplication morphism 
$m:\Gm_{,\,S} \times_S \Gm_{,\,S} \to \Gm_{,\,S}$
induces a matrix of the form
$$\begin{pmatrix} 1 & 0 & 0 & 0 \\
0 & 1 & 1 & \eta
\end{pmatrix}:
\Lambda\oplus \Lambda(1)[1]\oplus \Lambda(1)[1]
\oplus \Lambda(2)[2]
\to \Lambda\oplus \Lambda(1)[1]$$
on the associated homological motives. 
The morphism $\eta:\Lambda(1)[1] \to \Lambda$ obtained in this way
is known as the Hopf map. On the other hand, an element 
$a\in \mathcal{O}^{\times}(S)$, viewed as a section of $\Gm_{,\,S}$, 
gives rise to a matrix of the form
$$\begin{pmatrix} 
1\\
[a]
\end{pmatrix}:
\Lambda \to \Lambda \oplus \Lambda(1)[1].$$
It follows immediately from the above discussion that the identity
$$[ab]=[a]+[b]+\eta[a][b]$$
holds in the graded ring 
$\Hom_{\SH(S;\Lambda)}(\Lambda,\Lambda(\bullet)[\bullet])$
for all $a,b\in \mathcal{O}^{\times}(S)$. 
In fact, by \cite{Druzhinin-Milnor-Witt}, we even have a morphism
of graded rings
$${\rm K}_{\bullet}^{\text{"MW"}}(S)
\to \Hom_{\SH(S;\Lambda)}(\Lambda,\Lambda(\bullet)[\bullet])$$
where ${\rm K}_{\bullet}^{\text{"MW"}}(S)$ is the naive Milnor--Witt
$K$-theory ring of $S$, freely generated by symbols 
$\eta\in {\rm K}_{-1}^{\text{"MW"}}(S)$ and 
$[a]\in {\rm K}_1^{\text{"MW"}}(S)$, one for each 
$a\in \mathcal{O}^{\times}(S)$,
that satisfy Morel's relations in
\cite[Definition 3.1]{A1-alg-top}:
\begin{enumerate}

\item (Steinberg relation) For $a, b\in \mathcal{O}^{\times}(S)$
such that $a+b=1$, we have $[a][b]=0$.

\item For $a,b\in \mathcal{O}^{\times}(S)$, we have 
$[ab]=[a]+[b]+\eta[a][b]$.

\item For $a\in \mathcal{O}^{\times}(S)$, $[u]\eta=\eta[u]$.

\item $\eta(2+\eta[-1])=0$.

\end{enumerate}
Given $a\in \mathcal{O}^{\times}(S)$, 
we set $\langle a \rangle = 1+\eta[a]$ as in
\cite{A1-alg-top}.

\begin{rmk}
\label{rmk:easy-consequence-of-axioms-for-MW-K-th}
Over a base scheme $S$, extra care is needed 
when using the Steinberg relation since
there could be invertible elements 
$a\in \mathcal{O}^{\times}(S)\smallsetminus \{1\}$ 
such that $1-a$ is not invertible. 
However, the proof of \cite[Lemma 3.5]{A1-alg-top}
does not use the Steinberg relation and hence is 
valid over a general base $S$. Thus, the following relations
hold in ${\rm K}_{\bullet}^{\text{"MW"}}(S)$
for all $a,b\in \mathcal{O}^{\times}(S)$:
\begin{itemize}

\item $\langle 1 \rangle =1$ and $[1]=0$;

\item $[ab]=[a]+\langle a\rangle [b]=[a]\langle b \rangle +[b]$;

\item $\langle ab\rangle =\langle a \rangle \langle b \rangle$;

\item $\langle a\rangle$ is central in 
${\rm K}_{\bullet}^{\text{"MW"}}(S)$;

\item $[ab^{-1}]=[a]-\langle ab^{-1}\rangle [b]$
and, in particular, $[a^{-1}]=-\langle a^{-1}\rangle [a]$.

\end{itemize}
For $m\in \Z$, we set following 
\cite{A1-alg-top}:
$$m_{\epsilon}=\left\lceil \frac{m}{2} 
\right\rceil + \left\lfloor \frac{m}{2}
\right\rfloor\langle -1\rangle$$
where, for a real number $x$, $\lceil x \rceil$ is the smallest integer 
$\geq x$ and $\lfloor x \rfloor$ is the largest integer $\leq x$.
In particular, for $m\geq 0$, we have 
$$m_{\epsilon}=\overbrace{1+\langle 1 \rangle +1 +
\cdots }^{m\;\text{terms}}$$
and $(-m)_{\epsilon}=m_{\epsilon}\langle -1\rangle$.
It is also easy to check that $m_{\epsilon}n_{\epsilon}=(mn)_{\epsilon}$
for all $m,n\in \Z$. 
\end{rmk}

If $S$ is the spectrum of a field, Morel's relations imply
that $[a][-a]=0$. Although, this is unreasonable to expect 
in ${\rm K}_{\bullet}^{\text{"MW"}}(S)$, for a general $S$, 
it is nevertheless satisfied in the graded ring
$\Hom_{\SH(S;\Lambda)}(\Lambda,\Lambda(\bullet)[\bullet])$.

\begin{lemma}
\label{lemma:[a][-a]-is-zero-in-SH}
For $a\in \mathcal{O}^{\times}(S)$, we have 
$[a][-a]=0$ in 
$\Hom_{\SH(S;\Lambda)}(\Lambda,\Lambda(2)[2])$.
\end{lemma}

\begin{proof}
By functoriality, it is enough to prove that $[t][-t]=0$ in 
$\Hom_{\SH(S[t,t^{-1}];\Lambda)}(\Lambda,\Lambda(2)[2])$.
The cohomological motives of the $S$-schemes $S[t,t^{-1}]$ and 
$S[t,t^{-1},(1-t)^{-1}]$ are equal to 
$\Lambda\oplus \Lambda(-1)[-1]$ and 
$\Lambda\oplus \Lambda(-1)[-1]\oplus \Lambda(-1)[-1]$
respectively. By adjunction, the morphism 
\begin{equation}
\label{eq-lemma:[a][-a]-is-zero-in-SH-1}
\Hom_{\SH(S[t,t^{-1}];\Lambda)}(\Lambda,\Lambda(2)[2])
\to \Hom_{\SH(S[t,t^{-1},(1-t)^{-1}];\Lambda)}(\Lambda,\Lambda(2)[2])
\end{equation}
can be identified with the morphism obtained from 
$$\begin{pmatrix}1 & 0\\
0 & 1\\
0 & 0
\end{pmatrix}:\Lambda(2)[2]\oplus \Lambda(-1)[-1]
\to \Lambda(2)[2]\oplus \Lambda(-1)[-1]
\oplus \Lambda(-1)[-1]$$
by applying $\Hom_{\SH(S;\Lambda)}(\Lambda,-)$.
In particular, the morphism
\eqref{eq-lemma:[a][-a]-is-zero-in-SH-1}
is injective, and it is enough to show that 
$[t][-t]=0$ in $\Hom_{\SH(S[t,t^{-1},(1-t)^{-1}];\Lambda)}(\Lambda,\Lambda(2)[2])$. We will actually show this relation in 
${\rm K}_2^{\text{"MW"}}(S[t,t^{-1},(1-t)^{-1}])$
by repeating the proof in 
\cite[Lemma 3.7]{A1-alg-top}. Indeed, 
$1-t$ and $1-t^{-1}$ are invertible in 
$\mathcal{O}(S)[t,t^{-1},(1-t)^{-1}]$.
The Steinberg relation gives
$[t^{-1}][1-t^{-1}]=0$. But we have: 
$$[t^{-1}]=-\langle t^{-1}\rangle [t]
\qquad \text{and}\qquad 
[1-t^{-1}]=\left[\frac{1-t}{-t}\right]=[1-t]-\left\langle
\frac{1-t}{-t}\right \rangle[-t].$$
Putting these relations together, and using that $[t][1-t]=0$, 
we obtain the identity $[t][-t]=0$.
\end{proof}

\begin{cor}
\label{cor:angle-a-square-angle-is-one-}
For $a\in \mathcal{O}^{\times}(S)$, we have 
$$[a][a]=[-1][a]=[a][-1]
\qquad \text{and} \qquad \langle a^2\rangle=1$$
in $\Hom_{\SH(S;\Lambda)}(\Lambda,\Lambda(\bullet)[\bullet])$.
\end{cor}

\begin{proof}
This is identical to the proof of the second and fourth parts of 
\cite[Lemma 3.7]{A1-alg-top}.
We compute using the relation $[a][-a]=0$:
$$[a][a]=[a][-(-a)]=[a]([-1]+[-a]+\eta[-1][-a])=[a][-1].$$
The relation $[a][a]=[-1][a]$ is proven similarly. 
For the relation $\langle a^2\rangle=1$, we note that
$$1+\eta[a^2]=1+\eta([a]+[a]+\eta[a][a])=1+\eta(2+\eta[-1])[a]$$
and conclude using the fourth of Morel's relations.
\end{proof}

\begin{cor}
\label{cor:anti-commuting-of-[a][b]-in-SH}
For $a, b\in \mathcal{O}^{\times}(S)$, we have 
$[a][b]=(-1)_{\epsilon}[b][a]$ in
$\Hom_{\SH(S;\Lambda)}(\Lambda,\Lambda(2)[2])$.
\end{cor}

\begin{proof}
This is identical to the proof of the third part of 
\cite[Lemma 3.7]{A1-alg-top}.
Using Lemma
\ref{lemma:[a][-a]-is-zero-in-SH}
and Corollary 
\ref{cor:angle-a-square-angle-is-one-}, we have 
of equalities
$$\begin{array}{rcl}
0 & = & [ab][-ab]\\
& = & ([a]+\langle a \rangle [b])([-a]+\langle -a\rangle[b])\\
& = & \langle a \rangle [b][-a]+\langle -a\rangle [a][b]
+\langle -a^2\rangle [b][b]\\
& = & \langle a \rangle[b][-a]+\langle -a\rangle [a][b]+
\langle -1 \rangle [b][-1]\\
& = & \langle a \rangle[b]([a]+\langle a\rangle [-1])
+\langle -a\rangle [a][b]+
\langle -1 \rangle [b][-1]\\
& = & \langle a \rangle([b][a]+\langle -1\rangle [a][b])+
[b][-1]+\langle -1\rangle [b][-1].
\end{array}$$
So, to conclude, it remains to show that 
$[b][-1]+\langle -1\rangle [b][-1]=0$. 
Using again Lemma
\ref{lemma:[a][-a]-is-zero-in-SH} 
and Corollary 
\ref{cor:angle-a-square-angle-is-one-},
we have
$$\begin{array}{rcl}
0 & = & [b][-b]\\
& = & [b]([-1]+\langle -1 \rangle [b])\\
& = & [b][-1]+\langle -1\rangle [b][b]\\
& = & [b][-1]+\langle -1 \rangle [b][-1]
\end{array}$$
as needed. This finishes the proof.
\end{proof}

\begin{cor}
\label{cor:[a-power-m]-is-m-epsilon-[a]}
For $a\in \mathcal{O}^{\times}(S)$ and $m\in \Z$,
we have $[a^m]=m_{\epsilon}[a]$ in 
$\Hom_{\SH(S;\Lambda)}(\Lambda,\Lambda(1)[1])$.
\end{cor}

\begin{proof}
This is identical to the proof of 
\cite[Lemma 3.14]{A1-alg-top}. Arguing by induction, 
we have for $m\geq 1$:
$$\begin{array}{rcl}
[a^m] & = & [a]+[a^{m-1}]+\eta[a][a^{m-1}]\\
& = & [a]+ (m-1)_{\epsilon}[a]+(m-1)_{\epsilon}\eta[a][a]\\
& = & [a]+ (m-1)_{\epsilon}[a]+(m-1)_{\epsilon}\eta[-1][a]\\
& = & (1+(m-1)_{\epsilon}\langle -1\rangle)[a]\\
& = & m_{\epsilon}[a]. 
\end{array}$$
The case $m<0$ is obtained by obtained by applying the case 
$m>0$ to $a^{-1}$.
\end{proof}

If $-1$ is a square in $S$, Corollary
\ref{cor:angle-a-square-angle-is-one-}
implies that $\langle -1\rangle =1$
in $\Hom_{\SH(S;\Lambda)}(\Lambda,\Lambda)$.
It follows that $m_{\epsilon}=m$ for all $m\in \Z$.
This shows that Proposition
\ref{prop:action-of-elevation-power-m-on-Gm}
follows from the following more general statement.  

\begin{prop}
\label{prop:action-of-elevation-power-m-on-Gm-general}
Let $S$ be any scheme.
Then, for every integer $m\in \Z$, elevation to the $m$-th power
on $\Gm_{,\,S}$ is given by the matrix
$$\begin{pmatrix} 1 & 0\\
0 & m_{\epsilon}
\end{pmatrix}:\Lambda \oplus \Lambda(1)[1] \to 
\Lambda \oplus \Lambda(1)[1]$$
on the associated homological motive in $\SH(S;\Lambda)$.
\end{prop}

\begin{proof}
It is more convenient to show that elevation to the $m$-th power
induces the matrix 
$$\begin{pmatrix} 1 & 0\\
0 & m_{\epsilon}
\end{pmatrix}:\Lambda \oplus \Lambda(-1)[-1] \to 
\Lambda \oplus \Lambda(-1)[-1]$$
on cohomological motives (rather than homological motives).
The fact that the matrix is diagonal with first entry $1$
follows from the fact that elevation to the $m$-th power 
preserves the unit section. It remains to determine the last entry of the 
matrix. Using the isomorphisms
$$\begin{array}{rcl}
\Hom_{\SH(S[t,t^{-1}];\Lambda)}(\Lambda,\Lambda(1)[1]) & 
\simeq & 
\Hom_{\SH(S;\Lambda)}(\Lambda,\Lambda(1)[1]\oplus \Lambda)\\
& \simeq & \Hom_{\SH(S;\Lambda)}(\Lambda,\Lambda(1)[1])
\oplus 
\Hom_{\SH(S;\Lambda)}(\Lambda,\Lambda)\cdot [t],
\end{array}$$
the result follows from the equality 
$[t^m]=m_{\epsilon}[t]$ in 
$\Hom_{\SH(S[t,t^{-1}];\Lambda)}(\Lambda,\Lambda(1)[1])$.
\end{proof}

\bibliographystyle{alpha}
\bibliography{Bib}

\end{document}